\documentclass[12pt]{amsart}
\usepackage{amsfonts,amssymb,amsmath,amsthm,mathrsfs}
\usepackage{url}
\usepackage{enumerate}
\usepackage{lineno}
\usepackage[pdftex, bookmarksnumbered, bookmarksopen, colorlinks, citecolor=blue, linkcolor=blue]{hyperref}

\newtheorem{theorem}{Theorem}[section]

\newtheorem{lemma}[theorem]{Lemma}

\newtheorem{corollary}[theorem]{Corollary}
\theoremstyle{definition}
\newtheorem{definition}[theorem]{Definition}
\newtheorem{remark}[theorem]{Remark}
\numberwithin{equation}{section}
\newtheorem*{tha}{Theorem A}
\newtheorem{conjecture}[theorem]{Conjecture}
\usepackage{geometry}
\author[G. Hu]{Guoen Hu}
\address{Guoen Hu: School of Applied Mathematics, Beijing Normal University, Zhuhai 519087,
	P. R. China}
\email{guoenxx@163.com}
\author[X. Tao]{Xiangxing Tao}
\address{Xiangxing Tao: Department of Mathematics, School of Science, Zhejiang University of Science and Technology,
	Hangzhou 310023, People's Republic of China}
\email{xxtao@zust.edu.cn}
\author[Z. Wang]{Zhidan Wang}
\address{
	Zhidan Wang
	\\
	School of Mathematical Sciences
	\\
	Beijing Normal University
	\\
	Laboratory of Mathematics and Complex Systems
	\\
	Ministry of Education
	\\
	Beijing 100875
	\\
	People's Republic of China
}
\email{zdwang@mail.bnu.edu.cn}
\author[Q. Xue]{Qingying Xue$^\ast$}
\address{Qingying Xue
	\\
	School of Mathematical Sciences
	\\
	Beijing Normal University
	\\
	Laboratory of Mathematics and Complex Systems
	\\
	Ministry of Education
	\\
	Beijing 100875
	\\
	People's Republic of China
}
\email{qyxue@bnu.edu.cn}
\thanks{The first author was supported by the NNSF of
	China  (Nos. 11871108, 11971295), the  second   author was supported by
	the NNSF of
	China (No. 11771399), the fourth author was partly supported by the National Key Research and Development Program of China (Grant No. 2020YFA0712900) and NSFC (No. 11871101).}
\thanks{ $\ast$ Corresponding author: Qingying Xue }

\begin{document}
	
	\title[On the boundedness of non-standard rough singular integral operators]{On the boundedness of non-standard rough singular integral operators}
	
	\begin{abstract} Let $\Omega$ be homogeneous of degree zero, have vanishing moment of order one on the unit sphere $\mathbb {S}^{d-1}$($d\ge 2$). In this paper, our object of investigation is the following rough non-standard singular integral operator
		$$T_{\Omega,\,A}f(x)={\rm p.\,v.}\int_{\mathbb{R}^d}\frac{\Omega(x-y)}{|x-y|^{d+1}}\big(A(x)-A(y)-\nabla A(y)(x-y)\big)f(y){\rm d}y,$$
		where $A$ is a function defined on $\mathbb{R}^d$ with derivatives of order one in ${\rm BMO}(\mathbb{R}^d)$. We show that $T_{\Omega,\,A}$ enjoys the endpoint  $L\log L$ type estimate and is $L^p$ bounded if $\Omega\in L(\log L)^{2}(\mathbb{S}^{d-1})$. These resuts essentially improve the previous known results given by Hofmann \cite{hof} for the $L^p$ boundedness of $T_{\Omega,\,A}$ under the condition $\Omega\in  L^{q}(\mathbb {S}^{d-1})$ $(q>1)$, Hu and Yang \cite{hy} for the endpoint weak $L\log L$ type estimates when $\Omega\in {\rm Lip}_{\alpha}(\mathbb{S}^{d-1})$ for some $\alpha\in (0,\,1]$. Quantitative weighted strong and endpoint weak $L\log L$ type inequalities are proved whenever $\Omega\in L^{\infty}(\mathbb {S}^{d-1})$. The analysis of the weighted results relies heavily on two bilinear sparse dominations of $T_{\Omega,\,A}$ established herein.
	\end{abstract}
	\maketitle
	\section{Introduction}
	This paper will be devoted to study the boundedness of certain  non-standard Calder\'on-Zygmund operators with rough kernels. To begin with, let $d\ge 2$ and $x=(x_1,\cdots,x_n)\in \mathbb{R}^d.$ Let $\Omega$ be a function of homogeneous of degree zero, $\Omega \in L^1({\mathbb{S}^{d-1}})$ and satisfy the
	vanishing condition
	\begin{eqnarray}\label{equa:1.1}{\int_{\mathbb{S}^{d-1}}\Omega(x)x_j{ d}x=0},\quad \hbox{ }\  j=1,..., d.\end{eqnarray}
Define the  non-standard rough Calder\'on-Zygmund operator by
	\begin{eqnarray}\label{equa:1.2}
		\quad \quad T_{\Omega,\,A}f(x)={\rm p. v.}\int_{\mathbb{R}^d}\frac{\Omega(x-y)}{|x-y|^{d+1}}\big(A(x)-A(y)-\nabla A(y)(x-y)\big)f(y)dy,\end{eqnarray}
	where $A$ is a function on $\mathbb{R}^d$ with derivatives of order one in ${\rm BMO}(\mathbb{R}^d)$. The dual operator of  $T_{\Omega,\,A}$ is defined by
	\begin{eqnarray}\label{equation1.dualoperator}
		\qquad \widetilde{T}_{\Omega,A}f(x)={\rm p.\,v.}\int_{\mathbb{R}^d}\frac{\Omega(x-y)}{|x-y|^{d+1}}\big(A(x)-A(y)-\nabla A(x)(x-y)\big)f(y)dy.\end{eqnarray}
		This class of singular integrals is of interest in Harmonic analysis. It was well-known that $T_{\Omega,\,A}$ is closely related to the study of Calder\'on commutators \cite{cal1, BC}.
	Even for smooth kernel $\Omega$, since ${L^\infty(\mathbb{R}^d)}\subset {\rm BMO}(\mathbb{R}^d)$, the kernel of the operator $T_{\Omega,\,A}$  may fail to satisfy the classical standard kernel conditions. This is the main reason why one calls them nonstandard singular integral operators.	\vspace {0.1cm}
	
	Recall that if $\nabla A\in L^\infty(\mathbb{R}^d)$, then the $L^p(\mathbb{R}^d)$ boundedness of $T_{\Omega,\,A}$ follows by using the methods of rotation in the nice work of Cald\'{e}ron \cite{cal1},  Bainshansky  and Coifman \cite{BC}. Since the method of rotations doesn't work in the case of $\nabla A\in {\rm BMO}(\mathbb{R}^d)$,  Cohen \cite{cohen} and  Hu \cite{huy} obtained the  $L^p(\mathbb{R}^d)$ boundedness of  $T_{\Omega,\,A}$ with smooth kernels by means of a good-$\lambda$ inequality. More precisly, if $\Omega\in{\rm Lip}_\alpha(\mathbb{S}^{d-1})$ ($0<\alpha\le 1$), then Cohen \cite{cohen} proved that $T_{\Omega,A}$ is a bounded operator on $L^p(\mathbb{R}^d)$  for $1<p<\infty$. Later on,
	the result of Cohen \cite{cohen} was improved by Hofmann \cite{hof}. It was shown that $\Omega\in \cup_{q>1}L^q(\mathbb{S}^{d-1})$ is a sufficient condition for the $L^p(\mathbb{R}^d)$ boundedness of $T_{\Omega, A}$. If $\Omega\in L^{\infty}(\mathbb{S}^{d-1})$ and $w\in A_p(\mathbb{R}^d)$, Hofmann further demonstrated that $T_{\Omega,A}$ is bounded on $L^p(\mathbb{R}^d,\,w)$.
	\vspace {0.1cm}
	
	It is quite natural to ask if one can establish weak type inequalities for $T_{\Omega,\,A}$ or not. This question was first investigated by Hu and Yang \cite{hy} and they pointed out that $T_{\Omega,A}$ fails to be of weak type $(1,1)$, which differs in this aspect from the property of the classical singular integral operators. As a replacement of weak $(1,1)$ boundedness, it was shown that  $T_{\Omega,A}$ still enjoys the endpoint $L\log L$ type  estimates. Recently, if $\Omega\in{\rm Lip}_\alpha(\mathbb{S}^{d-1})$, Hu \cite{hu} considered the quantitative weighted bounds for $T_{\Omega,\,A}$ and obtain the following results.
	\begin{tha}[\cite{hu}]\label{thm1.1} Let $\Omega$ be homogeneous of degree zero, satisfy the vanishing moment \eqref{equa:1.1}, $A$ be a function  on $\mathbb{R}^d$ with derivatives of order one in ${\rm BMO}(\mathbb{R}^d)$ and ${\rm \Phi(t)=t\log ({e}+t)}$. Suppose that  $\Omega\in {\rm Lip}_{\alpha}(\mathbb{S}^{d-1})$ for some $\alpha\in (0,\,1]$, then
		\begin{enumerate}
			\item [\rm{(i).}] For $p\in (1,\,\infty)$ and  $w\in A_{p}(\mathbb{R}^d)$, it holds that			
			$$\|T_{\Omega,\,A}f\|_{L^p(\mathbb{R}^d,w)}
			\lesssim_{n,\,p}\|\nabla A\|_{{\rm BMO}(\mathbb{R}^d)} [w]_{A_p}^{\frac{1}{p}}\big([\sigma]_{A_{\infty}}^{\frac{1}{p}}+[w]_{A_{\infty}}^{\frac{1}{p'}}\big)[\sigma]_{A_{\infty}}
			\|f\|_{L^p(\mathbb{R}^d,\,w)},$$
			where $\|\nabla A\|_{{\rm BMO}(\mathbb{R}^d)}=\sum_{j=1}^d\|\partial_j A\|_{{\rm BMO}(\mathbb{R}^d)}$.
			
			\item [\rm{(ii).}] For $w\in A_1(\mathbb{R}^d)$ and $\lambda>0$, it holds that
			\begin{eqnarray}\label{equa:1.xx}w\big(\{x\in\mathbb{R}^d:\,|T_{\Omega,\,A}f(x)|>\lambda\}\big)
				\lesssim [w]_{A_1} \log^2 ({\rm e}+[w]_{A_{\infty}})\int_{\mathbb{R}^d}  \Phi\left(\frac{|f(x)|}{\lambda}\right) w(x)dx. \nonumber
			\end{eqnarray}
			\end{enumerate}
	\end{tha}
	Now, we recall some known results of classical singular integrals and make a comparative analysis.
	It was first shown by Calder\'{o}n and Zygmund \cite{CZ1} that the singular integrals $T_{\Omega}$ defined by 
	$$T_{\Omega}f(x)=\textrm{p.v.}\int_{\mathbb{R}^d}\frac{\Omega(y/|y|)}{|y|^d}f(x-y)dy$$
	 is  $L^p$ $(1<p<\infty)$ bounded either $\Omega$ is an odd function and $\Omega \in L^1(\mathbb{S}^{d-1})$, or $\Omega$ is an even function with $\int_{\mathbb{S}^{d-1}}\Omega \,d\sigma=0$ and $\Omega\in L\log L(\mathbb{S}^{d-1})$. Later on, the condition $\Omega\in L\log L(\mathbb{S}^{d-1})$ was improved to  $\Omega\in H^1(\mathbb{S}^{d-1})$ by Connett \cite{C1979}, Ricci and Weiss \cite{RW1979}, independently. 	Since then, great achievements have been made in this field.
	Among them are the celebrated works of  the weak type $(1,1)$ bounds given by Christ \cite{{C1988}} ,  Christ and Rubio~de Francia \cite{CR1988},  Hofmann, seeger \cite{se}, and Tao \cite{T1999}. It was shown that   $\Omega\in L\log L(\mathbb{S}^{d-1})$ is sufficient condition for the weak type $(1,1)$ estimate of $T_{\Omega}$. For other related nice contributions, we refer the readers to references \cite{D1993, DR1986, GS1999,V1996,W1990} and the references therein.	\vspace {0.1cm}

Cosider now the non-standard singular integral operator $T_{\Omega,\,A}$, things become more subtle. It is clear that the best known condition for the {$L^p(\mathbb{R}^{d})$} boundedness is $\Omega\in L^q(\mathbb{S}^{d-1}) $ and for the weak $L\log L$ type estimate is $Lip_\alpha(\mathbb{R}^{d})$ condition. Even if $\Omega \in L^{\infty}(\mathbb{S}^{d-1}) $, the weak type $L\log L$ type estimate of $T_{\Omega,A}$ holds or not is unknown. {Note that the following inclusion relationship holds
	\begin{equation}\label{relat}
Lip_\alpha(\mathbb{S}^{d-1}) (0<\alpha\le 1)\subsetneq {L^q(\mathbb{S}^{d-1})(q>1)\subsetneq L(\log L)^2(\mathbb{S}^{d-1}) \subsetneq L\log L(\mathbb{S}^{d-1})\subsetneq L^1(\mathbb{S}^{d-1}).}		
\end{equation} Therefore, it is quite natural to ask the following question:	
	\vspace{0.1cm}
	
	\textbf{Question 1:} What is the minimal condition for the
	weak type estimates of {$T_{\Omega,\,A}$} and $\widetilde{T}_{\Omega,A}f$?
			\vspace {0.2cm}
			
		The first main purpose of this paper is to show that $\Omega \in L(\log L)^2(\mathbb{S}^{d-1}) $ is a sufficient conditiont for the weak and strong boundedness of these two operators. We summarize our results as follows:
		
	\begin{theorem}\label{wmmthm1.1} Let $\Omega$ be homogeneous of degree zero, satisfy the vanishing moment (\ref{equa:1.1}), and $A$ be a function  in $\mathbb{R}^d$ with  derivatives of order one in ${\rm BMO}(\mathbb{R}^d)$. Suppose that  $\Omega\in L(\log L)^{2}(\mathbb{S}^{d-1})$. Then $T_{\Omega,\,A}$ is bounded on $L^2(\mathbb{R}^d)$.
		\end{theorem}

		\begin{theorem}\label{thm1.4} Let $\Omega$ be homogeneous of degree zero, satisfy the vanishing condition (\ref{equa:1.1}), and $A$ be a function  in $\mathbb{R}^d$ with  derivatives of order one in ${\rm BMO}(\mathbb{R}^d)$. Suppose that  $\Omega\in L(\log L)^{2}(\mathbb{S}^{d-1})$. Then
			for any $\lambda>0$ and ${\rm \Phi(t)=t\log ({e}+t)}$, the following inequalities hold
			\begin{eqnarray}\label{equation2.2}&&\big|\{x\in\mathbb{R}^d:\, |T_{\Omega,A}f(x)|>\lambda\}\big|\lesssim \int_{\mathbb{R}^d} \Phi\left(\frac{|f(x)|}{\lambda}\right) dx;\\
				\label{equation2.3}&&\big|\{x\in\mathbb{R}^d:\, |\widetilde{T}_{\Omega,A}f(x)|> \lambda\}\big|\lesssim \lambda^{-1}\|f\|_{L^1(\mathbb{R}^d)}.
			\end{eqnarray}
			\end{theorem}	\vspace {0.1cm}
		
		As an application of Theorem \ref{thm1.4},  we deduce that
		\begin{corollary} \label{cor2.1}Let $\Omega$ be homogeneous of degree zero, satisfy the vanishing condition (\ref{equa:1.1}), and $A$ be a function  in $\mathbb{R}^d$ with  derivatives of order one in ${\rm BMO}(\mathbb{R}^d)$. Suppose that  $\Omega\in L(\log L)^{2}(\mathbb{S}^{d-1})$. Then
			$$\|T_{\Omega,\,A}f\|_{L^p(\mathbb{R}^d)}\lesssim \Big \{\begin{array}{ll}
			p'^2\|f\|_{L^p(\mathbb{R}^d)},\,&p\in (1,\,2];\\
			p\|f\|_{L^p(\mathbb{R}^d)},\,&p\in (2,\,\infty).\end{array}
			$$
		\end{corollary}
			\begin{remark}
			Theorem \ref{wmmthm1.1}, along with Corollary \ref{cor2.1}, shows that $\Omega\in L(\log L)^{2}(\mathbb{S}^{d-1})$ is a sufficient condition such that $T_{\Omega,\,A}$ is bounded on $L^p(\mathbb{R}^d)$ for all $p\in (1,\,\infty)$. This improves essentially the result obtained in \cite[Theorem 1.1]{hof}, in which, it was shown that if $\Omega\in \cup_{q>1}L^q(\mathbb{S}^{d-1})$, then $T_{\Omega,\,A}$ is bounded on $L^p(\mathbb{R}^d)$ for all $p\in (1,\,\infty)$. By (\ref{relat}), Theorem \ref{thm1.4}
			also improves essentially the result in Therem A (ii) in the unweighted case.
					\end{remark}
	 We believe that the condition 	$\Omega\in L(\log L)^{2}(\mathbb{S}^{d-1})$ is the weakest condition for these weak type results to hold, in the following sense.
	 	\begin{conjecture} 
		$\Omega\in L(\log L)^{2}(\mathbb{S}^{d-1})$ is the minimal condition for the
		weak $L\log L$ type estimate of $T_{\Omega,\,A}$, and 	weak $(1,1)$ estimate of $\widetilde{T}_{\Omega,A}$, in the sense that the power 2 can't be replaced by any real numbers smaller than 2.
		\end{conjecture}
				
		In recent years, considerable attentions have been paid to the quantitative weighted bounds of the standard rough homogeneous singular integral $T_{\Omega}$,
		where $\Omega$ is homogeneous of degree zero, has mean value zero on the unit sphere  and $\Omega\in L^{1}(\mathbb{S}^{d-1})$.
			To state some known results, we begin by introducing the definition of $A_p$ weights class.
		
		\begin{definition}[\bf {$A_p$ weights class,\cite{M},\cite{wil}}]
			Let $p\in [1,\,\infty)$ and $w$ be a nonnegative, locally integrable function on $\mathbb{R}^d$. We say that   $w\in A_{p}(\mathbb{R}^d)$ if the $A_p$ constant $[w]_{A_p}$ is finite, where
			$$[w]_{A_p}:=\sup_{Q}\Big(\frac{1}{|Q|}\int_Qw(x)dx\Big)\Big(\frac{1}{|Q|}\int_{Q}w^{-\frac{1}{p-1}}(x)dx\Big)^{p-1},\,\,\,p\in (1,\,\infty),$$
			the  supremum is taken over all cubes in $\mathbb{R}^d$, and
			$[w]_{A_1}:=\sup_{x\in\mathbb{R}^d}\frac{Mw(x)}{w(x)}.$
			For a weight $u\in A_{\infty}(\mathbb{R}^d)=\cup_{p\geq 1}A_p(\mathbb{R}^d)$, $[u]_{A_{\infty}}$ is the $A_{\infty}$ constant of $u$, defined by
			$$[u]_{A_{\infty}}=\sup_{Q\subset \mathbb{R}^d}\frac{1}{u(Q)}\int_{Q}M(u\chi_Q)(x)dx.$$ 	
		\end{definition}
		
		We now summarize some known results, and divide them into two cases according to the kernels.\vspace{0.1cm}
		
		\noindent\textbf{The case $\Omega\in {\rm Lip}_{\alpha}(S^{d-1})$:} In 2012, Hyt\"onen and Lacey \cite{hyla} first showed that
		if $w\in A_p(\mathbb{R}^d)$, then
		\begin{eqnarray*}\|T_{\Omega}f\|_{L^p(\mathbb{R}^d,\,w)}\lesssim [w]_{A_p}^{\frac{1}{p}} \big([w]_{A_{\infty}}^{\frac{1}{p'}}+[\sigma]_{A_{\infty}}^{\frac{1}{p}}\big)\|f\|_{L^p(\mathbb{R}^d,\,w)}, \quad \hbox{for\ } 1<p<\infty,\end{eqnarray*}
		where and in the following  $\sigma=w^{-\frac{1}{p-1}}$.
		Later on, Hyt\"onen and P\'erez \cite{hp2} proved that if $w\in A_{1}(\mathbb{R}^d)$, then it holds that
		\begin{eqnarray*}
			\|T_{\Omega}f\|_{L^{1,\,\infty}(\mathbb{R}^d,w)}
			\lesssim [w]_{A_1}\log ({\rm e}+[w]_{A_{\infty}})
			\|f\|_{L^1(\mathbb{R}^d,\,w)},\end{eqnarray*}
		\textbf{The case $\Omega\in L^{\infty}(\mathbb{S}^{d-1})$}:
		In 2017, Hyt\"onen, et al. \cite{hrt} proved that if $w\in A_p(\mathbb{R}^d)$, then
		\begin{align}\label{equa:1.4}\|T_{\Omega}f\|_{L^p(\mathbb{R}^d,\,w)}\lesssim \|\Omega\|_{L^{\infty}(\mathbb{S}^{d-1})}[w]^{2\max\{1,\frac{1}{p-1}\}}\|f\|_{L^p(\mathbb{R}^d,\,w)}, \quad \hbox{for\ } 1<p<\infty.\end{align}
		This weighted bound is indeed optimal \cite{lno}.
		Inequality (\ref{equa:1.4}) was demonstrated again in \cite{ccdo} by usign the method of bilinear sparse domination.
		Recently, Li, et al. \cite{lpr} improved (\ref{equa:1.4}) and  showed that
		\begin{eqnarray*}\label{equa:1.6}&&\|T_{\Omega}f\|_{L^p(\mathbb{R}^d,\,w)}\lesssim [w]_{A_p}^{\frac{1}{p}} \big([w]_{A_{\infty}}^{\frac{1}{p'}}+[\sigma]_{A_{\infty}}^{\frac{1}{p}}\big)
			\min\{[\sigma]_{A_{\infty}},\,[w]_{A_{\infty}}\}\|f\|_{L^p(\mathbb{R}^d,\,w)},\quad  w\in A_p(\mathbb{R}^d) \end{eqnarray*}
		and
		\begin{eqnarray}\label{equa:1.7}
			\|T_{\Omega}f\|_{L^{1,\,\infty}(\mathbb{R}^d,w)}
			\lesssim [w]_{A_1}[w]_{A_{\infty}}\log ({\rm e}+[w]_{A_{\infty}})
			\|f\|_{L^1(\mathbb{R}^d,\,w)},\quad  w\in A_1(\mathbb{R}^d).
		\end{eqnarray}
	
		More recent results can be found in the work of Lerner \cite{ler4}, where some weak type estimate was given for certain grand maximal operator of $T_{\Omega}.$	\vspace {0.1cm}
		
		Based on the above analysis and compare Theorem A with the above results of the standard rough Calder\'on-Zygmund operators, a question arises naturally:
		\vspace{0.2cm}
		
			\textbf{Question 2:} {Does the operator $T_{\Omega,\,A}$ enjoy the quantitative weighted norm inequalities whenever the kernel $\Omega\in L^{\infty}(\mathbb{S}^{d-1})$?}	\vspace {0.1cm}
		
	Our second main purpose of this paper is to give a positive answer to Question 2, we obtain
	\begin{theorem}\label{thm1.2}
		Let $\Omega\in L^{\infty}(\mathbb{S}^{d-1})$ be homogeneous of degree zero, satisfy the vanishing condition (\ref{equa:1.1}), and $A$ be a function  on $\mathbb{R}^d$ with derivatives of order one in ${\rm BMO}(\mathbb{R}^d)$. Then for $p\in (1,\,\infty)$ and $w\in A_{p}(\mathbb{R}^d)$, the following weighted norm inequality holds
				\begin{eqnarray*}\|T_{\Omega,\,A}f\|_{L^p(\mathbb{R}^d,\,w)}\lesssim [w]_{A_p}^{\frac{1}{p}} \big([w]_{A_{\infty}}^{\frac{1}{p'}}+[\sigma]_{A_{\infty}}^{\frac{1}{p}}\big)[\sigma]_{A_{\infty}}\min\{[\sigma]_{A_{\infty}},\,
			[w]_{A_{\infty}}\}
			\|f\|_{L^p(\mathbb{R}^d,\,w)}.
		\end{eqnarray*}
	\end{theorem}
	
	\begin{theorem}\label{thm1.3}
		Let $\Omega\in L^{\infty}(\mathbb{S}^{d-1})$ be homogeneous of degree zero, satisfy the vanishing condition (\ref{equa:1.1}), and $A$ be a function  on $\mathbb{R}^d$ with derivatives of order one in ${\rm BMO}(\mathbb{R}^d)$. Then for $w\in A_1(\mathbb{R}^d)$, ${\rm \Phi(t)=t\log ({e}+t)}$ and $\lambda>0$, it holds that
		\begin{equation}\label{thm1.3'}w\big(\{x\in\mathbb{R}^d:\,|T_{\Omega,\,A}f(x)|>\lambda\}\big)
			\lesssim [w]_{A_1}[w]_{A_{\infty}}\log^2 ({\rm e}+[w]_{A_{\infty}})\int_{\mathbb{R}^d}  \Phi\left(\frac{|f(x)|}{\lambda}\right)w(x)dx.\nonumber
		\end{equation}
	\end{theorem}	
			\begin{remark}
		Corollary \ref{cor2.1} makes it possible for us to avoid proving the self-adjoint property as in \cite[Appendix A]{ccdo}).
		Moreover, Let $\eta\in C^{\infty}_0(\mathbb{R}^d)$ be a radial function on $\mathbb{R}^d$,  ${\rm supp}\,\eta\subset \{x:\, |x|\leq 2\}$, and $\eta(x)=1$ when $|x|\leq 1$. Consider the smooth truncated operator $T_{\Omega,\,A;\,\eta}$ by
		$$T_{\Omega,\,A;\,\eta}f(x)={\rm p. v.}\int_{\mathbb{R}^d}\eta(|x-y|)\frac{\Omega(x-y)}{|x-y|^{d+1}}\big(A(x)-A(y)-\nabla A(y)(x-y)\big)f(y)dy.
		$$
		It was proved in \cite[Lemma 3]{chenlu} that, under the hypothesis of Theorem \ref{thm1.4}, $T_{\Omega,\,A;\,\eta}$ is also bounded on $L^2(\mathbb{R}^d)$.  Furthermore, the same reasoning as in the proof of Theorem \ref{thm1.4} gives that the estimate \eqref{equation2.2} is also true for $T_{\Omega,\,A;\,\eta}$.
	\end{remark}
	The article is organized as follows. Section \ref {Sect 2} will be devoted to demonstrate the $L^2$ boundedness of $T_{\Omega,A}.$ In Section \ref{Sec3}, we will prove Theorem \ref{thm1.4} and Corollary \ref{cor2.1}. The proof of Theorem \ref{thm1.4} is not short and will be divided into several cases and steps.  Smoothness trunction method will play an important role and will be used several times.
	We will establish two pointwise sparse dominations for $T_{\Omega,A}$ in Section \ref{ssp1} and Section \ref{ssp2} respectively.  Finally, in Section \ref{sec5}, we would like to demonstrate Theorem \ref{thm1.2}-\ref{thm1.3} and hence obtain the quantitative weighted bounds for the operator $T_{\Omega,\,A}$.	\vspace{0.1cm}

	Let's explain a little bit about the proofs of the main results. In Section \ref{Sect 2}, we will introduce a convolution operator $Q_s$ with the property that \begin{equation*}
	\int^{\infty}_0Q_s^4\frac{ds}{s}=I.
	\end{equation*}
	This makes it possible to commutate with the paraproducts appeared in the proof and thus obtains more freedom in dealing with the estimates of the $L^2$ boundedness. Moreover, the method of dyadic analysis has been applied in the delicated decomposition of $L^2$ norm of   $T_{\Omega,\,A}$. At some key points, we will use some properties of Carleson measure.	\vspace {0.1cm}
	
	The key ingredient in our proof of Theorem \ref{thm1.4}  is to estimate the bad part in the Calder\'on-Zygmund decomposition of $f$.  In the  work of \cite{se}, Seeger showed that  if $\Omega\in L\log L(\mathbb{S}^{d-1})$, then $T_{\Omega}$ is bounded from $L^1(\mathbb{R}^d)$ to $L^{1,\,\infty}(\mathbb{R}^d)$. Recently,  the result of Seeger was generalized by Ding and Lai \cite{dinglai}. They considered the weak type endpoint estimate for more generalized operator $T_{\Omega}^*$ defined by
	$${T}^*_{\Omega}f(x)={\rm p.\,v.}\int_{\mathbb{R}^d}\Omega(x-y)K(x,\,y)f(y)dy.
	$$
	Ding and Lai proved that if $\Omega\in L\log L(\mathbb{S}^{d-1})$ and for some $\delta\in (0,\,1]$, the function
	$K$ satisfies
	\begin{eqnarray}\label{eq1.size}|K(x,\,y)|\lesssim \frac{1}{|x-y|^d};
	\end{eqnarray}
	\begin{eqnarray}\label{eq1.regular1}|K(x_1,\,y)-K(x_2,\,y)|\lesssim \frac{|x_1-x_2|^{\delta}}{|x_1-y|^{d+\delta}},\,\,|x_1-y|\geq 2|x_1-x_2|,
	\end{eqnarray}
	\begin{eqnarray}\label{eq1.regular2}|K(x,\,y_1)-K(x,\,y_2)|\lesssim \frac{|y_1-y_2|^{\delta}}{|x-y_1|^{d+\delta}},\,\,|x-y_1|\geq 2|y_1-y_2|,
	\end{eqnarray}
	and ${T}_{\Omega}^*$ is bounded on $L^2(\mathbb{R}^d)$, then ${T}_{\Omega}^*$ is bounded from $L^1(\mathbb{R}^d)$ to $L^{1,\,\infty}(\mathbb{R}^d)$. However, when $A$ has derivatives of order one in ${\rm BMO}(\mathbb{R}^d)$, the function
	$[{A(x)-A(y)-\nabla A(y)(x-y)}]{|x-y|^{-d-1}}$
	does not satisfy the conditions (\ref{eq1.size})-(\ref{eq1.regular2}). Let $f$ be a bounded function with compact support,  $b=\sum_{L}b_L$ be the bad part in the Calder\'on-Zygmund decomposition of $f$. In order to overcome this essential difficulty, we write
	$$T_{\Omega,\,A}b(x)=\sum_{L}\sum_{s}\int_{\mathbb{R}^d}\frac{\Omega(x-y)}{|x-y|^{d+1}}\phi_s(x-y)\big(A_L(x)-A_L(y)\big)b_{L}(y)dy+\hbox{ error\,\,terms},$$
	where $A_{L}(y)=A(y)-\sum_{n=1}^d\langle\partial _n A\rangle_Ly_n.$
	$\phi_s(x)=\phi(2^{-s}x)$. Here, $\phi$ is a smooth radial nonnegative function on $\mathbb{R}^d$ such that ${\rm supp}\, \phi\subset\{x:\frac{1}{4}\leq |x|\leq 1\}$ and
	$\sum_s\phi_s(x)=1$ for all $x\in \mathbb{R}^d\backslash\{0\}$. Then, our key observation is that, for each $s\in \mathbb{Z}$ and $L$ with side length $\ell(L)=2^{s-j}$, the kernel ${|x-y|^{-d-1}}\phi_s(x-y)\big(A_L(x)-A_L(y)\big)\chi_{L}(y)$ instead satisfies (\ref{eq1.regular1}) and (\ref{eq1.regular2}).	\vspace {0.1cm}

	In what follows, $C$ always denotes a
	positive constant which is independent of the main parameters
	involved but whose value may differ from line to line. We use the
	symbol $A\lesssim B$ to denote that there exists a positive constant
	$C$ such that $A\le CB$.  Specially, we use $A\lesssim_{n,p} B$ to denote that there exists a positive constant
	$C$ depending only on $n,\,p$ such that $A\le CB$. Constant with subscript such as $c_1$,
	does not change in different occurrences. For any set $E\subset\mathbb{R}^d$,
	$\chi_E$ denotes its characteristic function.  For a cube
	$Q\subset\mathbb{R}^d$, $\ell(Q)$ denotes the side length of $Q$, and $s_Q=\log_2\ell(Q)$, and for  $\lambda\in(0,\,\infty)$, we use
	$\lambda Q$ to denote the cube with the same center as $Q$ and whose
	side length is $\lambda$ times that of $Q$.  For a fixed cube $Q$, denote by $\mathcal{D}(Q)$ the set of dyadic cubes with respect to $Q$, that is, the cubes from $\mathcal{D}(Q)$ are formed by repeating subdivision of $Q$ and each of descendants into $2^n$ congruent subcubes.
	For $p\in[1,\,\infty]$, $p'$ denotes the dual exponent of $p$, namely, $1/p'=1-1/p$; for a locally integrable function $f$ and a cube $Q$,
	$\langle |f|\rangle_{Q}$ denotes the mean value of $|f|$ on $Q$ and $\langle |f|\rangle_{Q, r}=\big(\langle|f|^r\rangle_{Q}\big)^{\frac{1}{r}}$; for a suitable function $f$, $\widehat{f} $ denotes the Fourier transform of $f$.
	
		\vspace {0.1cm}
	
	The generalization of H\"older's inequality will be used in the proofs of our theorems. Let $\beta \geq 0$, $f$ be a locally integrable  function and $Q$ be a cube, define $\|f\|_{L(\log L)^{\beta},\,Q}$ by
	$$\|f\|_{L(\log L)^{\beta},\,Q}=\inf\Big\{\lambda>0:\,\frac{1}{|Q|}\int_{Q}\frac{|f(y)|}{\lambda}\log^{\beta}\Big({\rm e}+\frac{|f(y)|}{\lambda}\Big)dy\leq 1\Big\}.$$	
	We also define
	$$\|f\|_{{\rm exp}L,\,Q}=\inf\Big\{\lambda>0:\,\frac{1}{|Q|}\int_Q{\rm exp}\Big(\frac{|f(x)|}{\lambda}\Big)dx\leq 2\Big\}.
	$$
	The generalization of H\"older's inequality tells us that
	\begin{eqnarray}\label{generalholder}
		\frac{1}{|Q|}\int_{Q}|f(x)g(x)|dx\lesssim \|f\|_{L\log L,\,Q}\|g\|_{{\rm exp}L,\,Q}.
	\end{eqnarray}
	see \cite[p. 58]{rr}. Moreover, if $a\in {\rm BMO}(\mathbb{R}^d)$, then for a cube $Q$ and a suitable function $f$, it follows from (\ref{generalholder}) and the John-Nirenberg inequality that 
	\begin{eqnarray}\label{holderandjohnniren}\int_{Q}|a(x)-\langle a\rangle||f(x)|dx\lesssim \|a\|_{{\rm BMO}(\mathbb{R}^d)}|Q|\|f\|_{L\log L,\,Q}.\end{eqnarray}
	
We will need the maximal operators $M_{\beta}$ and $M_{L(\log L)^{\beta}}$ defined by
	$$M_{\beta}f(x)=\big[M(|f|^{\beta})(x)\big]^{\frac{1}{\beta}},\quad \hbox{and\ \ } M_{L(\log L)^{\beta}}g(x)=\sup_{Q\ni x}\|g\|_{L(\log L)^{\beta},\,Q}.$$
	where $M$ is the Hardy-Littlewood maximal operator.
	\vspace{0.3cm}
	\section{{$L^2$ boundedness with $ L(\log L)^{2}(\mathbb{S}^{d-1})$ kernels}}\label{Sect 2}

This section will be devoted to prove Theorem \ref{wmmthm1.1}, the {$L^2(\mathbb{R}^{d})$} boundedness of $T_{\Omega,\,A}$ when $\Omega\in L(\log L)^{2}(\mathbb{S}^{d-1})$.  {We will employ some ideas  from \cite{hof}, together with many more refined estimates.} We begin with some notions and lemmas. Let $\widetilde{\psi}\in C^{\infty}_0(\mathbb{R}^d)$ be a radial function with integral zero, ${\rm supp}\,\widetilde{\psi}\subset B(0,\,1)$, $\widetilde{\psi}_s(x)=s^{-d}\widetilde{\psi}(s^{-1}x)$ and assume that $$\int^{\infty}_0[\widehat{\widetilde{\psi}}(s)]^4\frac{ds}{s}=1.$$ 

Consider the convolution operator  $Q_sf(x)=\widetilde{\psi}_s*f(x).  $
It enjoys the property that
\begin{equation}\label{wmm3.1}
	\int^{\infty}_0Q_s^4\frac{ds}{s}=I.
\end{equation}
Moreover, by the classical Littlewood-Paley theory, it follows that
\begin{eqnarray}\label{lpp}\Big\|\Big(\int^{\infty}_0|Q_sf|^2\frac{ds}{s}\Big)^{1/2}\Big\|_{L^2(\mathbb{R}^d)}\lesssim \|f\|_{L^2(\mathbb{R}^d)}.\end{eqnarray}
	Let $\phi$ be a smooth radial nonnegative function on $\mathbb{R}^d$ with ${\rm supp}\, \phi\subset\{x:\frac{1}{4}\leq |x|\leq 1\}$,
$\sum_s\phi_s(x)=1$ with $\phi_j(x)=2^{-jd}\phi(2^{-j}x)$ for all $x\in \mathbb{R}^d\backslash\{0\}$. 
For each fixed $j\in\mathbb{Z}$, define
\begin{eqnarray}\label{TRAC OP}
	T_{\Omega,\,A;\,j}f(x)=\int_{\mathbb{R}^d}K_{A,\,j}(x,\,y)f(y)dy,\label{wmmhtg}
\end{eqnarray}
where
$$K_{A,\,j}(x,\,y)=\frac{\Omega(x-y)}{|x-y|^{d+1}}(A(x)-A(y)-\nabla A(y)(x-y)\big) \phi_j(|x-y|).
$$
The following lemmas are needed in our analysis.

\begin{lemma}[\cite{hof}]\label{wmmwmmlem3.2} Let $\Omega$ be homogeneous of degree zero, satisfies the vanishing momoent (\ref{equa:1.1}) and $\Omega\in L^1 (\mathbb{S}^{d-1})$. Let $A$ be a function on $\mathbb{R}^d$ with derivitives of order one in ${\rm BMO}(\mathbb{R}^d)$. Then  for any $k_1,\,k_2\in \mathbb{Z}$ with $k_1<k_2$, the following inequality holds
	$$\Big|\sum_{k_1\leq j\leq k_2}\int_{\mathbb{R}^d}K_{A,\,j}(x,\,y)dy\Big|\lesssim \|\Omega\|_{L^1(\mathbb{S}^{d-1})}.$$
\end{lemma}

\begin{lemma}[\cite{hof}]\label{wmmwmmlem3.3}Let $\Omega$ be homogeneous of degree zero,  integrable on $\mathbb{S}^{d-1}$ and satisfy the vanishing moment (\ref{equa:1.1}).   $A$ be a function on $\mathbb{R}^d$ with derivitives of order one in ${\rm BMO}(\mathbb{R}^d)$. Then
	there exists a constant $\epsilon\in (0,\,1)$, such that  for  $s\in (0,\,\infty)$ and $j\in\mathbb{Z}$ with $s2^{-j}\leq 1$,
	$$\|Q_sT_{\Omega,\,A;\,j}1\|_{L^{\infty}(\mathbb{R}^d)}\lesssim \|\Omega\|_{L^1(\mathbb{S}^{d-1})}(2^{-j}s)^{\epsilon}.
	$$
\end{lemma}
\begin{lemma}[\cite{hof}]\label{wmmwmmlem3.4kkk}Let $\Omega$ be homogeneous of degree zero and $\Omega\in L^{\infty}(\mathbb{S}^{d-1})$, $A$ be a function on $\mathbb{R}^d$ with derivitives of order one in ${\rm BMO}(\mathbb{R}^d)$. Then  there exists a constant $\varepsilon\in (0,\,1)$, such that for $s\in (0,\,\infty)$ and $j\in \mathbb{Z}$ with $2^{-j}s\leq 1$,
	$$\|Q_sT_{\Omega,\,A;\,j}f\|_{L^2(\mathbb{R}^d)}\lesssim \|\Omega\|_{L^{\infty}(\mathbb{S}^{d-1})}(2^{-j}s)^{\varepsilon}\|f\|_{L^2(\mathbb{R}^d)}.$$
\end{lemma}
\begin{lemma} [\cite{hu1}] \label{wmmlem2.6}Let $\Omega$ be homogeneous of degree zero, have mean value zero on $\mathbb{S}^{d-1}$ and $\Omega\in L(\log L)^2({S}^{d-1})$.
	Then for $b\in {\rm BMO}(\mathbb{R}^d)$, $[b,\,T_{\Omega}]$, the commutator of $T_{\Omega}$ with symbol $b$, defined by
	$$[b,\,T_{\Omega}]f(x)=b(x)T_{\Omega}f(x)-T_{\Omega}(bf)(x),\,\,\,f\in C^{\infty}_0(\mathbb{R}^d),$$
	is bounded on $L^p(\mathbb{R}^d)$.
\end{lemma}

\begin{lemma}[\cite{hof}]\label{wmmlem2.7}Let $\Omega$ be homogeneous of degree zero, and integrable on $\mathbb{S}^{d-1}$ and satisfy the vanishing moment (\ref{equa:1.1}), $A$ be a function  in $\mathbb{R}^d$ with  derivatives of order one in ${\rm BMO}(\mathbb{R}^d)$. Then for any $r\in (0,\,\infty)$, functions $\widetilde{\eta}_1,\,\widetilde{\eta}_2\in C^{\infty}_0(\mathbb{R}^d)$ whose supported on balls of radius $r$,
	$$\Big|\int_{\mathbb{R}^d}\widetilde{\eta}_2(x)T_{\Omega,\,A}\widetilde{\eta}_1(x)dx\Big|\lesssim \|\Omega\|_{L^1(\mathbb{S}^{d-1})}
	r^{-d}\prod_{j=1}^2\big(\|\widetilde{\eta}_j\|_{L^{\infty}(\mathbb{R}^d)}+r\|\nabla \widetilde{\eta}_j\|_{L^{\infty}(\mathbb{R}^d)}).
	$$
\end{lemma}
We need a lemma from the book of Grafakos.
\begin{lemma}[\cite{gra0}, p.539] \label{graf} Let $\Phi$ be a function on $\mathbb {R}^d$ satisfying for some $0<C,\delta<\infty,$ $|\Phi (x)|\le  C {(1+|x)|)^{-d-\delta}}$. Then a measure $\mu$ on $\mathbb {R}_+^{d+1}$ is a Carleson if and only if for every $p$ with $1<p<\infty$ there is a constant $C_{p,d,\mu}$ such that for all $f\in L^p(\mathbb {R}^d)$ we have
	$$\int_{\mathbb {R}_+^{d+1}} |\Phi_t*f(x)|^p d\mu (x,t)\le C_{p,d,\mu}\int_{\mathbb {R}^d}|f(x)|^pdx.
	$$
\end{lemma}
{\it Proof of Theorem \ref{wmmthm1.1}.} {Invoking \eqref{wmm3.1},}  to prove that  $T_{\Omega,\,A}$ is bounded on $L^2(\mathbb{R}^d)$,  it suffices to show the following inequalities hold for $f,\,g\in C^{\infty}_0(\mathbb{R}^d)$, 
\begin{eqnarray}\label{wmmlittlewood1}
	\Big|\int^{\infty}_0\int_{0}^t\int_{\mathbb{R}^d}Q_s^4T_{\Omega,\,A}Q_t^4f(x)g(x)dx\frac{ds}{s}\frac{dt}{t}\Big|\lesssim \|f\|_{L^2(\mathbb{R}^d)}\|g\|_{L^2(\mathbb{R}^d)};
\end{eqnarray}
\begin{eqnarray}\label{wmmlittlewood2}
	\Big|\int^{\infty}_0\int_{t}^{\infty}\int_{\mathbb{R}^d}Q_s^4T_{\Omega,\,A}Q_t^4f(x)g(x)dx\frac{ds}{s}\frac{dt}{t}\Big|\lesssim \|f\|_{L^2(\mathbb{R}^d)}\|g\|_{L^2(\mathbb{R}^d)}.
\end{eqnarray}
First, we will prove (\ref{wmmlittlewood1}). To this aim, the kernel $\Omega$ will be decomposed into disjoint forms. Let
$$E_0=\{\theta \in \mathbb{S}^{d-1}:\, |\Omega(\theta )|\leq 1\} \ \hbox{and\ } E_i=\{\theta \in \mathbb{S}^{d-1}:\,2^{i-1}<|\Omega(\theta )|\leq 2^i\},\quad i\in\mathbb{N}.$$
Set
$$\Omega_0(\theta )=\Omega(\theta )\chi_{E_0}(\theta ),\quad \,\Omega_i(\theta )=\Omega(\theta )\chi_{E_i}(\theta )\,\,(i\in\mathbb{N}).$$
For $i\in\mathbb{N}\cup\{0\}$, let $T_{\Omega,\,A;\,j}^i$ be the same as in \eqref{TRAC OP} for $T_{\Omega,\,A;\,j}$ with $\Omega$ replaced by $\Omega_i$. Then 
\begin{eqnarray}
	&&\label{wmmlittlewood100}	\int^{\infty}_0\int_{0}^t\int_{\mathbb{R}^d}Q_s^4T_{\Omega,\,A}Q_t^4f(x)g(x)dx\frac{ds}{s}\frac{dt}{t}=\sum_i\sum_j\int^{\infty}_0\int_{0}^t\int_{\mathbb{R}^d}Q_s^4T_{\Omega,\,A;\,j}^iQ_t^4f(x)g(x)dx\frac{ds}{s}\frac{dt}{t}.
\end{eqnarray}

Fix $j\in \mathbb {Z}$, partition the set of those $s,j$ in equation (\ref{wmmlittlewood100}) into three regions:
\begin{align*}
&E_1 (j, s,t)=\{(s,\,t):\,0\leq t\leq 2^j,\,0\leq s\leq t\};\\
&E_2 (j, s,t)=\big\{(s,\,t):\,2^j\leq t<(2^js^{-\alpha})^{\frac{1}{1-\alpha}},\,0\le s\le t\big\};\\
&E_3 (j, s,t)=\big\{(s,\,t):\,\max\{2^j,\, (2^js^{-\alpha})^{\frac{1}{1-\alpha}} \}\leq t<\infty,\,0\leq s\leq t\big\}.
\end{align*}
It remains to discuss the contribution of each $E_{j,s,t}$ on the right ride of  (\ref{wmmlittlewood100}) to inequality (\ref{wmmlittlewood1}) . 
\vspace{0.2cm}

\noindent{\textbf {Contribution of $E_1(j,,s,t)$.}} 

Let $\varepsilon$ be the same  constant appeared in Lemma  \ref{wmmwmmlem3.4kkk} and denote $N=2(\lfloor\varepsilon^{-1}\rfloor+1)$. 
For each fixed $i\in\mathbb{N}$, we introduce the notion $ E_{1,1}^i$ and $E_{1,2}^i$ as follows
\begin{align*}
&E_{1,1}^i(j,s,\,t)=\{(j,s,\,t):\,0\leq t\leq 2^j,\,0\leq s\leq t,\, 2^j\leq s2^{iN}\};\\
&E_{1,2}^i(j,s,\,t)=\{(j,s,\,t):\,0\leq t\leq 2^j,\,0\leq s\leq t,\,2^j>s2^{iN}\}.
\end{align*}
Then, one gets obviously that {$E_1(j,,s,t)=E_{1,1}^i(j,s,\,t)\cup E_{1,2}^i(j,s,\,t):=E_{1,1}^i\cup E_{1,2}^i.$}
Therefore
\begin{eqnarray*}
	&&	\Big|\sum_i\sum_j\int^{\infty}_0\int_{0}^t\int_{\mathbb{R}^d}\chi_{E_1(j,,s,t)}Q_s^4T_{\Omega,\,A;\,j}^iQ_t^4f(x)g(x)dx\frac{ds}{s}\frac{dt}{t}\Big| \\
	&&\quad\leq \sum_{i=1}^{\infty}\sum_j\int^{\infty}_0\int_{0}^{\infty}\chi_{E^i_{1,1}}
	\Big|\int_{\mathbb{R}^d}Q^4_s{T}^i_{\Omega,\,A;\,j}Q_t^4f(x)g(x)dx\Big|\frac{ds}{s}\frac{dt}{t}\\
	&&\quad \quad + \sum_{i=1}^{\infty}\sum_j\int^{\infty}_0\int_{0}^{\infty}\chi_{E^i_{1,2}}\Big|\int_{\mathbb{R}^d}	Q_s^4 {T}^i_{\Omega,\,A;\,j}Q_t^4f(x)g(x)\Big|\frac{ds}{s}\frac{dt}{t}\\&&\quad \quad +
	\sum_j\int^{\infty}_0\int_{0}^{\infty}\chi_{E_1(j,,s,t)}\Big|\int_{\mathbb{R}^d}
	Q_s^4 {T}^0_{\Omega,\,A;\,j}Q_t^4f(x)g(x)dx\Big|\frac{ds}{s}\frac{dt}{t}:={\rm I}+{\rm II}+{\rm III}.
\end{eqnarray*}

We first consider term I.
Let $\{I_l\}_{l}$ be a sequence of cubes having disjoint interiors and side lengths $2^j$, such that \begin{eqnarray}\label{eq3.decomp}
	\mathbb{R}^d=\mathop{\cup}\limits_{l}I_l.\end{eqnarray} 
	
	For each fixed $l$, let $\zeta_l\in C^{\infty}_0(\mathbb{R}^d)$ such that ${\rm supp}\,\zeta_l\subset 48dI_l$, $0\leq \zeta_l\leq 1$ and $\zeta_l(x)\equiv 1$ when $x\in 32dI_l$. Let $x_l$ be a point on the boundary of $50dI_l$ and 
	$$\widetilde{A}_{I_l}(y)=A(y)-\sum_{m=1}^m\langle\partial_m A\rangle_{I_l}y_m,\,\,A_{I_l}(y)=A_{I_l}^*(y)\zeta_l(y),$$
	with
	$A_{I_l}^*(y)=\widetilde{A}_{I_l}(y)-\widetilde{A}_{I_l}(x_l).
	$
	Note that for $x\in 30 dI_l$ and $y\in\mathbb{R}^d$ with $|x-y|\leq 2^{j+2}$, we have
	$$A(x)-A(y)-\nabla A(y)(x-y)=A_{I_l}(x)-A_{I_l}(y)-\nabla A_{I_l}(y)(x-y).$$
	An application of Lemma \ref{lem2.3} then implies that
	$\|A_{I_l}\|_{L^{\infty}(\mathbb{R}^d)}\lesssim 2^j.
	$

	For each fixed $j\in\mathbb{Z}$, consider the operators $W^i_{\Omega,\,j}$ and $U^i_{\Omega,\,m;j}$ defined by
$$W_{\Omega,j}^ih(x)=\int_{\mathbb{R}^d}\frac{\Omega_i(x-y)}{|x-y|^{d+1}}\phi_j(x-y)h(y)dy$$
and
$$U_{\Omega,m;j}^ih(x)=\int_{\mathbb{R}^d}\frac{\Omega_i(x-y)(x_m-y_m)}{|x-y|^{d+1}}\phi_j(x-y)h(y)dy.$$
 It is well known that for $p\in (1,\,\infty)$, they enjoy the following properties:
$$\|W_{\Omega,j}^ih\|_{L^p(\mathbb{R}^d)}\lesssim 2^{-j}\|\Omega_i\|_{L^1(S^{d-1})}\|h\|_{L^p(\mathbb{R}^d)};\quad\|U_{\Omega,\,m,j}^ih\|_{L^p(\mathbb{R}^d)}\lesssim \|\Omega_i\|_{L^1(S^{d-1})}
\|h\|_{L^p(\mathbb{R}^d)}.$$

For each fixed $l$, let $h_{s,l}(x)=Q_sg(x)\chi_{I_l}(x)$ and  $I_l^*=60dI_l$. For $x\in {\rm supp}h_{s,l}$, we have
\begin{eqnarray*}
	{T}_{\Omega,A,j}^iQ_t^4f(x)&=&A_{I_l}(x)W^i_{\Omega,j}Q_t^4f(x)-W^i_{\Omega,j}(A_{I_l}Q_t^4f)(x)-\sum_{m=1}^dU_{\Omega,m,j}^i(\partial_mA_{I_l}Q_t^4f)(x).
\end{eqnarray*}
Hence, to show the estimate for $I$, we need to consider the following three terms.
$${\rm R}_{i}^1=\sum_{j}\int^{2^j}_{2^{j-Ni}}\int^{2^j}_{2^{j-Ni}}
\Big|\sum_l\int_{\mathbb{R}^d}A_{I_l}(x)Q_s^3h_{s,l}(x)W^i_{\Omega,j}Q_t^4f(x)dx\Big|\frac{dt}{t}\frac{ds}{s};
$$
$${\rm R}_{i}^2=\sum_{j}\int^{2^j}_{2^{j-Ni}}\int^{2^j}_{2^{j-Ni}}
\Big|\sum_l\int_{\mathbb{R}^d}Q_s^3h_{s,l}(x)W^i_{\Omega,j}(A_{I_l}Q_t^4f)(x)dx\Big|\frac{dt}{t}\frac{ds}{s};
$$
and
\begin{eqnarray*}{\rm R}_{i}^3&=&\sum_{m=1}^d\sum_{j}\int^{2^j}_{2^{j-Ni}}\int^{2^j}_{2^{j-Ni}}
	\Big|\sum_l\int_{\mathbb{R}^d} Q_s^3h_{s,l}(x)U^i_{\Omega,m,j}(\partial_mA_{I_l}Q^4_tf\big)(x)dx\Big|\frac{ds}{s}\frac{dt}{t}:=\sum_{m=1}^d{\rm R}_{i,m}^3.
\end{eqnarray*}
For $	{\rm R}^1_{i},$ note that 
\begin{eqnarray*}
	\sum_j\sum_l\int^{2^j}_{2^{j-iN}}\|Q_s^3h_{s,l}\|_{L^2(\mathbb{R}^d)}^2\frac{ds}{s}
	&\leq &iN\sum_j\int^{2^j}_{2^{j-1}}\sum_l\|h_{s,l}\|_{L^2(\mathbb{R}^d)}^2\frac{ds}{s}
	\lesssim i\int^{\infty}_0
	\|Q_sg\|_{L^2(\mathbb{R}^d)}^2\frac{ds}{s}.
\end{eqnarray*}
Then, the well-known Littlewood-Paley theory for $g$-function leads to that
\begin{eqnarray*}
	\sum_j\sum_l\int^{2^j}_{2^{j-iN}}\|Q_s^3h_{s,l}\|_{L^2(\mathbb{R}^d)}^2\frac{ds}{s}
	\lesssim   i
{\bigg	\|\bigg(\int^{\infty}_0 |Q_sg( \cdot)|^2\frac{ds}{s}\bigg)^{1/2} \bigg\|_{L^2(\mathbb{R}^d)}}	\lesssim i\|g\|_{L^2(\mathbb{R}^d)}^2.
\end{eqnarray*}
{For $x\in 48dI_l,$ since $\sup \{\phi_j\}\subset [2^{j-2},2^j]$} and note that $\phi_j(x-y)Q_t^4f(y)=\chi_{I_l^*}(y)\phi_j(x-y)Q_t^4(y)$, then, $W_{\Omega,j}^i(Q_t^4f)=W_{\Omega,j}^i(\chi_{I_l^*}Q_t^4f)$.
It then follows from  H\"older's inequality, Cauchy-Schwarz inequality and the boundedness of $W_{\Omega,j}^i$ that
\begin{eqnarray*}
	|{\rm R}^1_{i}|&\leq&\Big(\sum_j\sum_l\int^{2^j}_{2^{j-iN}}\int^{2^j}_{2^{j-iN}}
	\|Q_s^3h_{s,l}\|_{L^2(\mathbb{R}^d)}^2
	\frac{ds}{s}\frac{dt}{t}\Big)^{1/2}\\
	&&\times\Big(\sum_j\sum_l\int^{2^j}_{2^{j-iN}}\int^{2^j}_{2^{j-iN}}
	\|A_{I_l}W_{\Omega,j}^i(\chi_{I_l^*}Q_t^4f)\|_{L^2(\mathbb{R}^d)}^2
	\frac{dt}{t}\frac{ds}{s}\Big)^{1/2}\\
	&\lesssim&\|\Omega_i\|_{L^1(S^{d-1})}\Big(\sum_j\sum_l\int^{2^j}_{2^{j-iN}}
	\int^{2^j}_{2^{j-iN}}
	\|Q_s^3h_{s,l}\|_{L^2(\mathbb{R}^d)}^2\frac{dt}{t}\frac{ds}{s}\Big)^{1/2}\\
	&&\quad\times\Big(\sum_j\sum_l\int^{2^j}_{2^{j-iN}}\int^{2^j}_{2^{j-iN}}\|\chi_{I_l^*}Q_t^4f\|_{L^2(\mathbb{R}^d)}^2\frac{dt}{t}\frac{ds}{s}\Big)^{1/2}\\
	&\lesssim& i^{2}\|\Omega_i\|_{L^1(S^{d-1})}\|f\|_{L^2(\mathbb{R}^d)}\|g\|_{L^2(\mathbb{R}^d)},
\end{eqnarray*}
where in the last inequality we have used the fact that the cubes $\{60 dI_l\}_l$ have bounded overlaps. 

The same reasoning applies to ${\rm R}^2_{i}$  with small and straightforward modifications yields that
\begin{eqnarray*}
	|{\rm R}^2_{i}|
	&\lesssim&i\|\Omega_i\|_{L^1(S^{d-1})}\Big(\sum_j
	\int^{2^j}_{2^{j-iN}}
	\|Q_sg\|_{L^2(\mathbb{R}^d)}^2\frac{ds}{s}\Big)^{1/2}\Big(\sum_j\sum_l\int^{2^j}_{2^{j-iN}}\|\zeta_lQ_t^4f\|_{L^2(\mathbb{R}^d)}^2\frac{dt}{t}\Big)^{1/2}\\
	&\lesssim& i^2\|\Omega_i\|_{L^1(S^{d-1})}\|f\|_{L^2(\mathbb{R}^d)}\|g\|_{L^2(\mathbb{R}^d)}.
\end{eqnarray*}

Now we are in a position to consider each term ${\rm R}_{i,m}^3$. For $x\in 32dI_l$, it is easy to check 
\begin{eqnarray*}
	\partial_mA_{I_l}(x)Q_t^4f(x)&=&\zeta_l(x)[\partial_m{A},\,Q_t]
	Q_t^3f(x)+\zeta_l(x)Q_t([\partial_m{A},\,
	Q_t]Q_t^2f(x)\\
	&&+\zeta_l(x)Q_t^2(\partial_m\widetilde{A}_{I_l}Q^2_t)f(x).
\end{eqnarray*}
Therefore 	${\rm R}_{i,m}^3$ can be controlled by the sum of the following terms:
\begin{eqnarray*}
\quad\quad	{\rm R}_{i,m}^{3,1}=&&\sum_{j}\int^{2^j}_{2^{j-Ni}}\int^{2^j}_{2^{j-iN}}
	\Big|\sum_l\int_{\mathbb{R}^d} Q_s^3h_{s,l}(x)U^i_{\Omega,m,j}\big([\partial_m A,\, Q_t]Q^3_tf\big)(x)dx\Big|\frac{dt}{t}\frac{ds}{s};\end{eqnarray*}
	\begin{eqnarray*}
	&&	{\rm R}_{i,m}^{3,2}=\sum_{j}\int^{2^j}_{2^{j-Ni}}\int^{2^j}_{2^{j-iN}}\Big|\sum_l\int_{\mathbb{R}^d} Q_s^3h_{s,l}(x)U^i_{\Omega,m,j}Q_t\big([\partial_m A,\,Q_t]Q_t^2f\big)(x)dx\Big|\frac{dt}{t}\frac{ds}{s};\end{eqnarray*}
\begin{eqnarray*}
\quad	{\rm R}_{i,m}^{3,3}=&&\sum_{j}\int^{2^j}_{2^{j-Ni}}\int^{2^j}_{2^{j-iN}}\Big|\sum_l\int_{\mathbb{R}^d} Q_s^3h_{s,l}(x)U^i_{\Omega,m,j}Q_t^2(\partial_m\widetilde{A}_{I_l}Q_t^2)f(x)dx\Big|\frac{dt}{t}\frac{ds}{s}.
\end{eqnarray*}
Observe that
$|[\partial_mA,\,Q_t]h(x)|\lesssim M_{\partial_mA}h(x),$
where $M_{\partial_m A}$ is the commutator of the Hardy-Littlewood maximal operator defined by
$$M_{\partial_m A}h(x)=\sup_{r>0}r^{-d}\int_{|x-y|<r}|\partial_m A(x)-\partial_m A(y)||h(y)dy.
$$
By  H\"older's inequality, along with the $L^2(\mathbb{R}^d)$ boundedness of $M_{\partial_m A}$, it yields that
\begin{eqnarray*}
	|{\rm R}_{i,m}^{3,1}|&\leq&\Big(\sum_j\int^{2^j}_{2^{j-Ni}}\int^{2^j}_{2^{j-iN}}\Big\|
	Q_s^3\big(\sum_{l}h_{s,l}\big)\Big\|_{L^2(\mathbb{R}^d)}^2\frac{dt}{t}\frac{ds}{s}\Big)^{1/2}\\
	&&\times\Big(\sum_j\int^{2^j}_{2^{j-Ni}}\int^{2^j}_{2^{j-iN}}\|U_{\Omega,j,m}^i([\partial_mA,\, Q_t]Q_t^3f\|_{L^2(\mathbb{R}^d)}^2
	\frac{dt}{t}\frac{ds}{s}\Big)^{1/2}\\
	&\lesssim&i^{2}\|\Omega_i\|_{L^1(S^{d-1})}\|f\|_{L^2(\mathbb{R}^d)}\|g\|_{L^2(\mathbb{R}^d)}.
\end{eqnarray*}
Exactly the same reasoning applies to $	{\rm R}_{i,m}^{3,2}$, we obtain
\begin{eqnarray*}
	|{\rm R}_{i,m}^{3,2}|\lesssim i^{2}\|\Omega_i\|_{L^1(S^{d-1})}\|f\|_{L^2(\mathbb{R}^d)}\|g\|_{L^2(\mathbb{R}^d)}.
\end{eqnarray*}
%Observe that
%$$|A_{I_l}^*(x)\partial_m\zeta_l(x)|\lesssim 2^{-j}|A(x)-A(x_l)\chi_{48dI_l}(x)\lesssim 1.$$
%It follows that
%\begin{eqnarray*}
%|{\rm R}_{i,m}^{3,4}|&\leq&i\Big(\sum_{j}\sum_l\int^{2^j}_{2^{j-iN}}
%\|\chi_{48dI_l}Q_s^4g\|_{L^2(\mathbb{R}^d)}^2\frac{ds}{s}\Big)^{1/2}\\
%&&\times\Big(\sum_{j}\sum_l\int^{2^j}_{2^{j-Ni}}
%\|U^i_{\Omega,m,j}Q^2_tf_l\|^2_{L^2(\mathbb{R}^d)}\frac{dt}{t}\Big)^{1/2}\\
%&\lesssim &i^2\|\Omega_i\|_{L^1(S^{d-1})}\|f\|_{L^2(\mathbb{R}^d)}\|g\|_{L^2(\mathbb{R}^d)}.
%\end{eqnarray*}
As for ${\rm R}_{i,m}^{3,3},$ observing that for fixed $l\in\mathbb{Z}$, $s,\,t\leq 2^j$, one gets
$$Q_t(\partial_m\widetilde{A}_{I_l}Q_t^2f)(x)
=Q_t(\partial_m\widetilde{A}_{I_l}\chi_{I_l^*}Q_t^2f)(x),\quad U^i_{\Omega,\,m,j}Q_s=Q_sU^i_{\Omega,m,j}\ \hbox{and}\  Q_sQ_t=Q_tQ_s.$$
Henceforth we have 
\begin{eqnarray*}
	{\rm R}_{i,m}^{3,3}&=&\sum_j\int^{2^j}_{2^{j-Ni}}\int^{2^j}_{2^{j-iN}}
	\Big|\sum_l\int_{\mathbb{R}^d}Q_tQ_s^2h_{s,l}(x)Q_sU^i_{\Omega,m,j}Q_t(\partial_m\widetilde{A}_{I_l}\chi_{I_l^*}Q_t^2f)(x)dx
	\frac{dt}{t}\frac{ds}{s}\Big|\\
	&\leq &\Big(\sum_j\sum_{l}\int^{2^j}_{2^{j-Ni}}\int^{2^j}_{2^{j-iN}}
	\| Q_tQ^2_sh_{s,l}\|^2_{L^2(\mathbb{R}^d)}\frac{dt}{t}\frac{ds}{s}\Big)^{1/2}\\
	&&\times\Big(\sum_j\sum_l\int^{2^j}_{2^{j-Ni}}\int^{2^j}_{2^{j-iN}}
	\Big\|Q_s\big(U^i_{\Omega,m,j}Q_t(\partial_m\widetilde{A}_{I_l}\chi_{I_l^*}Q_t^2f)\big)\Big\|^2_{L^2(\mathbb{R}^d)}
	\frac{dt}{t}\frac{ds}{s}\Big)^{1/2}.
\end{eqnarray*}
Let $x\in 48dI_l$, $q\in (1,\,2)$ and $s\in (2^{j-1}, 2^j)$. A straightforward computation involving H\"older's inequality and the John-Nirenberg inequality  gives us that
\begin{eqnarray}\label{eq3.qs}
	|Q_s(\partial_mA_{I_l}h)(x)|&\leq&\int_{\mathbb{R}^d}|\widetilde{\psi}_s(x-y)||\partial_mA(y)-\langle
	\partial_mA\rangle_{I(x,\,s)}||h(y)|dy\\
	&&+|\langle \partial_m A\rangle_{I_l}-\langle \partial_mA\rangle_{I(x,\,s)}|\int_{\mathbb{R}^d}|\widetilde{\psi}_s(x-y)||h(y)|dy \nonumber\\
	&\lesssim&M_qh(x)+\log (1+2^j/s)Mh(x)\nonumber \\
	&\lesssim& M_qh(x),\nonumber
\end{eqnarray}
where $I(x,\,s)$ is the cube center at $x$ and having side length $s$. 

This inequality, together with the boundedness of $U^i_{\Omega,m,j}$ and maximal funtion $M_qh$, implies that
\begin{eqnarray*}
	&&\Big(\sum_j\sum_l\int^{2^j}_{2^{j-Ni}}\int^{2^j}_{2^{j-iN}}
	\big\|Q_s\big(U^i_{\Omega,m,j}Q_t(\partial_mA_{I_l}\chi_{I_l^*}Q_t^2f)\big)\big\|^2_{L^2(\mathbb{R}^d)}
	\frac{dt}{t}\frac{ds}{s}\Big)^{1/2}\\
	&&\quad \lesssim i \Big(\sum_j\sum_l\int_{2^{j-1}}^{2^j}
	\big\|U^i_{\Omega,m,j}Q_t(\partial_mA_{I_l}\chi_{I_l^*}Q_t^2f)\big\|^2_{L^2(\mathbb{R}^d)}\frac{dt}{t}\Big)^{1/2}\\
	&&\quad\lesssim i\|\Omega_i\|_{L^1(S^{d-1})}
	\|f\|_{L^2(\mathbb{R}^d)}.
\end{eqnarray*}
On the other hand, by the $L^2$ boundedness of convolution operators and the Littlewood-Paley theory for $g$-function again, we have that
\begin{eqnarray*}
	\sum_j\sum_{l}\int^{2^j}_{2^{j-Ni}}\int^{2^j}_{2^{j-iN}}
	\| Q_tQ^2_sh_{s,l}\|^2_{L^2(\mathbb{R}^d)}\frac{dt}{t}\frac{ds}{s}&&\lesssim 
	i^2
	\int^{\infty}_{0}\|Q_sg\|^2_{L^2(\mathbb{R}^d)}\frac{ds}{s}\lesssim i^2\|g\|_{L^2(\mathbb{R}^d)}^2.
\end{eqnarray*}
Therefore
$${\rm R}_{i,m}^{3,3}\lesssim i^2\|\Omega_i\|_{L^1(S^{d-1})}\|f\|_{L^2(\mathbb{R}^d)}\|g\|_{L^2(\mathbb{R}^d)}.$$
Combining the estimates for ${\rm R}_{i}^1$, ${\rm R}_i^2$ and ${\rm R}_{i,m}^{3,n}$ (with $1\leq m\leq d, \,n=1,\,2,\,3$)  in all yields  that
\begin{eqnarray}\label{equation3.II}
	{\rm I}&\lesssim& 
	\sum_{i=1}^{\infty}i^2\|\Omega_i\|_{L^1(S^{d-1})}\|f\|_{L^2(\mathbb{R}^d)}\|g\|_{L^2(\mathbb{R}^d)}\lesssim \|f\|_{L^2(\mathbb{R}^d)}\|g\|_{L^2(\mathbb{R}^d)}.
\end{eqnarray}
It remains to discuss the contribution of terms II and III. For $i\in\mathbb{N}\cup\{0\}$, by Lemma \ref{wmmwmmlem3.4kkk}, one gets 
\begin{eqnarray}
	&& \sum_j\int^{\infty}_0\int_{0}^{\infty}\chi_{E^i_{1,2}}\|Q_sT^i_{\Omega,\,A;\,j}Q_t^4f\|_{L^2(\mathbb{R}^d)}
	\|Q_s^3g\|_{L^2(\mathbb{R}^d)}\frac{ds}{s}\frac{dt}{t}\\
	&&\quad\lesssim 2^i \Big(\int^{\infty}_0\int_{0}^{\infty}
	\sum_{j}
	\chi_{E^i_{1,2}}(2^{-j}s)^{\varepsilon}\|Q_s^3g\|^2_{L^2(\mathbb{R}^d)}
	\frac{ds}{s}\frac{dt}{t}\Big)^{\frac{1}{2}}\nonumber\\
	&&\qquad\times\Big(\int^{\infty}_0\int_{0}^{\infty}
	\sum_{j}
	\chi_{E^i_{1,2}}(2^{-j}s)^{\varepsilon}\|Q_t^4f\|_{L^2(\mathbb{R}^d)}^2
	\frac{ds}{s}\frac{dt}{t}\Big)^{\frac{1}{2}}.\nonumber
\end{eqnarray}
Note that
$$E_{1,2}^i(j,s,\,t)\subset\big\{(j,s,\,t):\,0\leq t\leq 2^j,\,0\leq s\leq t,\,2^j\geq \max\{t,\,s2^{iN}\}\big\},$$
Thus
\begin{eqnarray*}\sum_{j}
	\chi_{E_{1,2}^i}(2^{-j}s)^{\varepsilon}&\le 2^{-iN\varepsilon/2}\big(\frac{s}{t}\big)^{\varepsilon/2}\chi_{\{(s,t):\,s\leq t\}}(s,\,t), \end{eqnarray*}
which further implies that
$$\Big(\int^{\infty}_0\int_{0}^{\infty}
\sum_{j}
\chi_{E_{1,2}^i}(2^{-j}s)^{\varepsilon}\|Q_s^3g\|^2_{L^2(\mathbb{R}^d)}
\frac{ds}{s}\frac{dt}{t}\Big)^{\frac{1}{2}}\lesssim 2^{-Ni\varepsilon/4}\|g\|_{L^2(\mathbb{R}^d)},
$$
and
$$\Big(\int^{\infty}_0\int_{0}^{\infty}
\sum_{j}
\chi_{E_{1,2}^i}(2^{-j}s)^{\varepsilon}\|Q_t^4f\|^2_{L^2(\mathbb{R}^d)}
\frac{ds}{s}\frac{dt}{t}\Big)^{\frac{1}{2}}\lesssim 2^{-Ni\varepsilon/4}\|f\|_{L^2(\mathbb{R}^d)}.
$$
Therefore, these inequalities, together with the fact that $E_{1,1}^0=\emptyset$ may lead to 
\begin{eqnarray}\label{twoth}
	{\rm II+III}\lesssim \sum_{i=0}^{\infty}2^{i}2^{-Ni\varepsilon/2}
	\|f\|_{L^2(\mathbb{R}^d)}\|g\|_{L^2(\mathbb{R}^d)}\lesssim \|f\|_{L^2(\mathbb{R}^d)}\|g\|_{L^2(\mathbb{R}^d)}.
\end{eqnarray}
Inequlity (\ref{twoth}), together with the estimtate  ( \ref{equation3.II}) for ${\rm I}$, gives that		\begin{eqnarray}\label{wmmequation2.8}\Big|\sum_i\sum_j\int^{\infty}_0\int_{0}^t\int_{\mathbb{R}^d}\chi_{E_1(j,,s,t)}Q_s^4T_{\Omega,\,A;\,j}^iQ_t^4f(x)g(x)dx\frac{ds}{s}\frac{dt}{t}\Big| \lesssim \|f\|_{L^2(\mathbb{R}^d)}\|g\|_{L^2(\mathbb{R}^d)}.
\end{eqnarray}
\noindent{\textbf {Contribution of $E_2(j,,s,t)$.}} 

Let $\alpha\in (\frac{d+1}{d+2},\,1)$, $i\in\mathbb{N}\cup\{0\}$, and write
\begin{eqnarray}\label{E_2}&&\sum_i\sum_{j\in\mathbb{Z}}\int^{\infty}_{2^j}\int^{(2^jt^{\alpha-1})^{\frac{1}{\alpha}}}_0
	\Big|\int_{\mathbb{R}^d}Q_s^4{T}_{\Omega, A,\,j}^iQ_t^4f(x)g(x)dx\Big|\frac{ds}{s}\frac{dt}{t}\\
	&& \leq \sum_{i}\sum_{j\in\mathbb{Z}}\int^{2^j}_{2^{j-Ni}}\int^{(2^js^{-\alpha})^{\frac{1}{1-\alpha}}}_{2^j}
	\Big|\int_{\mathbb{R}^d}Q_s^4 {T}_{\Omega, A,\,j}^iQ_t^4f(x)g(x)dx\Big|\frac{dt}{t}\frac{ds}{s}\nonumber\\
	&&\quad+\sum_{i}\sum_{j\in\mathbb{Z}}\int^{2^{j-Ni}}_0\int^{(2^js^{-\alpha})^{\frac{1}{1-\alpha}}}_{2^j}
	\Big|\int_{\mathbb{R}^d}Q_s^4 {T}_{\Omega, A,\,j}^iQ_t^4f(x)g(x)dx\Big|\frac{dt}{t}\frac{ds}{s}\nonumber:= IV+V.\nonumber
\end{eqnarray}

{Firstly, we consider  the term $ IV$.} {When} $i=0$, the integral $\int^{2^j}_{2^{j-Ni}}\int^{(2^js^{-\alpha})^{1/{(1-\alpha)}}}_{2^j}\frac{dt}{t}\frac{ds}{s}$ vanishes, we only need to consider the case  $i\in\mathbb{N}$.
Since $s>2^{j-Ni}$, then $(2^j s^{-\alpha})^{\frac1 {1-\alpha}}\le 2^j 2^{iN\frac{\alpha}{1-\alpha}}$. Therefore
\begin{eqnarray*}
	IV&&=\sum_{i}\sum_{j\in\mathbb{Z}}\int^{2^j}_{2^{j-Ni}}\int^{(2^js^{-\alpha})^{\frac{1}{\alpha}}}_{2^j}
	\Big|\int_{\mathbb{R}^d}Q_s^4{T}^i_{\Omega, A,\,j}Q_t^4f(x)g(x)dx\Big|\frac{dt}{t}\frac{ds}{s}\\
	&&\leq \sum_{i}\sum_j\int^{2^j}_{2^{j-Ni}}\int^{2^j2^{iN\frac{\alpha}{1-\alpha}}}_{2^j}\Big|\int_{\mathbb{R}^d}
	T_{\Omega,A,j}^iQ_t^4f(x)Q_s^4g(x)dx\Big|\frac{dt}{t}\frac{ds}{s}\\
	&&\leq \sum_{i}\sum_j\int^{2^j}_{2^{j-Ni}}\int^{2^j2^{iN\frac{\alpha}{1-\alpha}}}_{2^j}\Big|\sum_l\int_{\mathbb{R}^d}
	T_{\Omega,A,j}^iQ_t^4f(x)Q_s^3h_{s,l}(x)dx\Big|\frac{dt}{t}\frac{ds}{s},
\end{eqnarray*}
where $h_{s,l}(x)=Q_sg(x)\chi_{I_l}(x)$, and $\{I_l\}_{l}$ be the cubes in (\ref{eq3.decomp}). 

Observe that when $x\in 4dI_l$,
$
T_{\Omega,A,j}^i(Q_t^4f)(x)Q_s^3h_{s,l}(x)=
T_{\Omega,A,j}^i(\zeta_lQ_t^4f)(x)Q_s^3h_{s,l}(x),
$
we rewrite
$$\aligned
{}&{T}_{\Omega,A,j}^i(\zeta_lQ_t^4f)(x)\\&=\Big(A_{I_l}(x)W^i_{\Omega,j}Q_t^4f(x)-W^i_{\Omega,j}(A_{I_l}Q_t^4f)(x)-\sum_{m=1}^dU_{\Omega,m,j}^i(\zeta_{l}[\partial_m A,Q_t]Q_t^3f)(x)\\
	&\quad-\sum_{m=1}^dU_{\Omega,m,j}^i(\zeta_lQ_t[\partial_{m}A,\,Q_t]Q_t^2f)(x)-\sum_{m=1}^dU_{\Omega,\,m,j}^i(\zeta_lQ_tQ_t(\partial_m\widetilde{A_{I_l}}Q_t^2f)(x)\Big)\chi_{4dI_l}(x).
\endaligned$$
Similar to the estimate for ${\rm R}_i^1$ and ${\rm R}_i^2$, we  know that
\begin{eqnarray}\label{eq3.ri1}
	&&\sum_j\int^{2^j}_{2^{j-Ni}}\int^{2^j2^{iN\frac{\alpha}{1-\alpha}}}_{2^j}\Big|\sum_l\int_{\mathbb{R}^d}
	W_{\Omega,j}^i(A_{I_l}Q_t^4f)(x)Q_s^3h_{s,l}(x)dx\Big|\frac{dt}{t}\frac{ds}{s}\\
	&&\quad\lesssim i^2\|\Omega_i\|_{L^1(\mathbb{R}^d)}\|f\|_{L^2(\mathbb{R}^d)}\|g\|_{L^2(\mathbb{R}^d)},\nonumber
\end{eqnarray}
and
\begin{eqnarray}\label{eq3.ri2}
	&&\sum_j\int^{2^j}_{2^{j-Ni}}\int^{2^j2^{iN\frac{\alpha}{1-\alpha}}}_{2^j}\sum_l\Big|\int_{\mathbb{R}^d}
	A_{I_l}(x)W_{\Omega,j}^iQ_t^4f(x)Q_s^3h_{s,l}(x)dx\Big|\frac{dt}{t}\frac{ds}{s}\\
	&&\quad\lesssim i^2\|\Omega_i\|_{L^1(\mathbb{R}^d)}\|f\|_{L^2(\mathbb{R}^d)}\|g\|_{L^2(\mathbb{R}^d)}.\nonumber
\end{eqnarray}
 On the other hand, for each fixed $1\leq m\leq d$, the same reasonging as what we have done for ${\rm R}^1_{i,m}$ and ${\rm R}^2_{i,m}$  yields that 
\begin{eqnarray}\label{eq3.rim1}
	&&\sum_j\sum_l\int^{2^j}_{2^{j-Ni}}\int^{2^j2^{iN\frac{\alpha}{1-\alpha}}}_{2^j}\Big|\int_{\mathbb{R}^d}
	U_{\Omega,m,j}^i(\zeta_l[\partial_mA,Q_t]Q_t^3f)(x)Q_s^3h_{s,l}(x)dx\Big|\frac{dt}{t}\frac{ds}{s}\\
	&&\quad\lesssim i^2\|\Omega_i\|_{L^1(\mathbb{R}^d)}\|f\|_{L^2(\mathbb{R}^d)}\|g\|_{L^2(\mathbb{R}^d)},\nonumber
\end{eqnarray}
and
\begin{eqnarray}\label{eq3.rim2}
	&&\sum_j\sum_l\int^{2^j}_{2^{j-Ni}}\int^{2^j2^{iN\frac{\alpha}{1-\alpha}}}_{2^j}\Big|\int_{\mathbb{R}^d}
	U_{\Omega,m,j}^i(\zeta_lQ_t[\partial_mA,Q_t]Q_t^2f)(x)Q_s^3h_{s,l}(x)dx\Big|\frac{dt}{t}\frac{ds}{s}\\
	&&\quad\lesssim i^2\|\Omega_i\|_{L^1(\mathbb{R}^d)}\|f\|_{L^2(\mathbb{R}^d)}\|g\|_{L^2(\mathbb{R}^d)},\nonumber
\end{eqnarray}
Note that  if $x\in \chi_{4dI_l}(x),$ then
$U_{\Omega,\,m,j}^i(Q_tQ_t(\partial_m\widetilde{A_{I_l}}Q_t^2f))(x)=U_{\Omega,\,m,j}^i(\zeta_lQ_tQ_t(\partial_m\widetilde{A_{I_l}}Q_t^2f))(x).
$ Since the kernel of $Q_t$ is radial and it enjoys the property that $$<U_{\Omega,\,m,j}^i(\zeta_lQ_tf),g>= <U_{\Omega,\,m,j}^i(\zeta_lf), Q_tg>.$$ Hence,
we have
$$\aligned
\int_{\mathbb{R}^d}U_{\Omega,\,m,j}^i(\zeta_lQ_tQ_t(\partial_m\widetilde{A_{I_l}}Q_t^2f))(x)Q_s^3h_{l,s}(x)dx
	&=\int_{\mathbb{R}^d}U_{\Omega,\,m,j}^iQ_s(\partial_m\widetilde{A_{I_l}}Q_sQ_t^2f)(x)
	Q_t^2Q_sh_{s,l}(x)dx\\
	&\quad-\int_{\mathbb{R}^d}U_{\Omega,\,m,j}^iQ_s[\partial_m A,\,Q_s]Q_t^2f(x)Q_t^2Q_sh_{s,l}(x)dx.
\endaligned
$$
A trivial argument then yields that
\begin{eqnarray}
	&&\Big|\sum_{j}\sum_l\int^{2^j}_{2^{j-iN}}\int^{2^j2^{iN\frac{\alpha}{1-\alpha}}}_{2^j}\int_{\mathbb{R}^d}U_{\Omega,\,m,j}^iQ_s[\partial_m A,\,Q_s]Q_t^2f(x)Q_t^2Q_sh_{l,s}(x)dx\frac{dt}{t}\frac{ds}{s}\Big|\\
	%&&\quad\lesssim %\Big(\sum_{j}\sum_l\int^{2^j}_{2^{j-iN}}\int^{2^j2^{iN\frac{\alpha}{1-\alpha}}}_{2^j}\|U_{\Omega,\,m,j}^iQ_s[\partial_m %A,\,Q_s]Q_t^2f\|_{L^2(\mathbb{R}^d)}^2\frac{ds}{s}\frac{dt}{t}\Big)^{1/2}\\
	%&&\qquad\times \Big(\sum_{j}\sum_l\int^{2^j}_{2^{j-iN}}\int^{2^j2^{iN\frac{\alpha}{1-\alpha}}}_{2^j}
	%\big\|Q_t^2Q_s^2\big(\sum_{l}g_l\big)\big\|_{L^2(\mathbb{R}^d)}^2\frac{ds}{s}\frac{dt}{t}\Big)^{1/2}\\
	&&\quad\lesssim i^2\|\Omega_i\|_{L^1(S^{d-1})}\|f\|_{L^2(\mathbb{R}^d)}\|g\|_{L^2(\mathbb{R}^d)}.\nonumber
\end{eqnarray}
Now we write
$$\aligned
{}&\int_{\mathbb{R}^d}U_{\Omega,\,m,j}^iQ_s(\partial_m\widetilde{A_{I_l}}Q_sQ_t^2f)(x)
	Q_t^2Q_sh_{s,l}(x)dx\\
	&=\int_{\mathbb{R}^d}Q_sQ_t^2f(x)
	[\partial_mA,\,Q_s]U^i_{\Omega,m,j}Q_t^2Q_sh_{s,l}(x)dx+\int_{\mathbb{R}^d}Q_sQ_t^2f(x)
	Q_s[\partial_mA,\,U^i_{\Omega,m,j}]Q_t^2Q_sh_{s,l}(x)dx
	\endaligned
	$$ 
	$$\aligned
	{}&\qquad+\int_{\mathbb{R}^d}Q_sQ_t^2f(x)
	Q_sU^i_{\Omega,m,j}[\partial_mA,Q_t^2]Q_sh_{s,l}(x)dx+\int_{\mathbb{R}^d}Q_sQ_t^2f(x)Q_sU^i_{\Omega,m,j}Q_t^2[\partial_mA,Q_s]h_{s,l}(x)dx\\
	&\qquad+\int_{\mathbb{R}^d}Q_sQ_t^2f(x)Q_sU_{\Omega,m,j}^iQ_t^2Q_s(\partial_m\widetilde{A}_{I_l}h_{s,l})(x)dx:=
	\sum_{k=1}^5{\rm S}_{i,m,l}^k.\endaligned
$$ A standard argument involving H\"older's inequality leads to that
\begin{eqnarray}\label{eq3.viml1}
	&&\sum_j\int^{2^j}_{2^{j-Ni}}\int^{2^j2^{iN\frac{\alpha}{1-\alpha}}}_{2^j}\big|\sum_l{\rm S}_{i,m,l}^1\big|
	\frac{dt}{t}\frac{ds}{s}\\
	&&\quad\lesssim\sum_j\int^{2^j}_{2^{j-Ni}}\int^{2^j2^{iN\frac{\alpha}{1-\alpha}}}_{2^j}
	\|Q_sQ_t^2f\|_{L^2(\mathbb{R}^d)}
	\Big\|[\partial_mA,\,Q_s]U^i_{\Omega,m,j}Q_t^2Q_s^2g\Big\|_{L^2(\mathbb{R}^d)}
	\frac{dt}{t}\frac{ds}{s}\nonumber\\
	&&\quad\lesssim \|\Omega_i\|_{L^1(S^{d-1})}\Big(\sum_j\int^{2^j}_{2^{j-Ni}}\int^{2^j2^{iN\frac{\alpha}{1-\alpha}}}_{2^j}
	\|Q_sQ_t^2f\|_{L^2(\mathbb{R}^d)}^2\frac{dt}{t}\frac{ds}{s}\Big)^{1/2}\nonumber\\
	&&\qquad\times\Big(\sum_j\int^{2^j}_{2^{j-Ni}}\int^{2^j2^{iN\frac{\alpha}{1-\alpha}}}_{2^j}
	\Big\|Q_t^2Q_s^2g\Big\|_{L^2(\mathbb{R}^d)}^2\frac{dt}{t}\frac{ds}{s}\Big)^{1/2}\nonumber\\
	&&\quad\lesssim i^2|\Omega_i\|_{L^1(S^{d-1})}\|f\|_{L^2(\mathbb{R}^d)}\|g\|_{L^2(\mathbb{R}^d)}.\nonumber
\end{eqnarray}
Similarly, one can verify that
\begin{eqnarray}\label{eq3.viml3}
	&&\sum_j\int^{2^j}_{2^{j-Ni}}\int^{2^j2^{iN\frac{\alpha}{1-\alpha}}}_{2^j}\big|\sum_l{\rm S}_{i,m,l}^3\big|
	\frac{dt}{t}\frac{ds}{s}\lesssim i^2\|\Omega_i\|_{L^1(S^{d-1})}\|f\|_{L^2(\mathbb{R}^d)}\|g\|_{L^2(\mathbb{R}^d)}.
\end{eqnarray}
and
\begin{eqnarray}\label{eq3.viml4}
	&&\sum_j\int^{2^j}_{2^{j-Ni}}\int^{2^j2^{iN\frac{\alpha}{1-\alpha}}}_{2^j}\big|\sum_l{\rm S}_{i,m,l}^4\big|
	\frac{dt}{t}\frac{ds}{s}\lesssim i^2\|\Omega_i\|_{L^1(S^{d-1})}\|f\|_{L^2(\mathbb{R}^d)}\|g\|_{L^2(\mathbb{R}^d)}.
\end{eqnarray}
On the other hand, the fact (see \cite[Lemma 4 and Lemma 3]{hu}) 
$$\|[\partial_mA,\,U^i_{\Omega,m,j}]h\|_{L^2(\mathbb{R}^d)}\lesssim \big(2^{-i}+i\|\Omega_i\|_{L^1(S^{d-1})}\big)\|h\|_{L^2(\mathbb{R}^d)},$$ implies that
\begin{eqnarray}\label{eq3.viml2}
	&&\sum_j\int^{2^j}_{2^{j-Ni}}\int^{2^j2^{iN\frac{\alpha}{1-\alpha}}}_{2^j}\big|\sum_l{\rm V}_{i,m,l}^2\big|
	\frac{dt}{t}\frac{ds}{s}\lesssim\big(i2^{-i}+i^2\|\Omega_i\|_{L^1(S^{d-1})}\big)\|f\|_{L^2(\mathbb{R}^d)}
	\|g\|_{L^2(\mathbb{R}^d)}.
\end{eqnarray}
Applying H\"older's inequality and inequality  (\ref{eq3.qs}) in the case $s\in (2^{j-1}, 2^j)$, we obtain
\begin{eqnarray}\label{eq3.viml5}
	&&\sum_j\int^{2^j}_{2^{j-Ni}}\int^{2^j2^{iN\frac{\alpha}{1-\alpha}}}_{2^j}\big|\sum_l{\rm V}_{i,m,l}^5\big|
	\frac{dt}{t}\frac{ds}{s}\\
	&&\quad\lesssim \|\Omega_i\|_{L^1(S^{d-1})}\Big(\sum_j\int^{2^j}_{2^{j-Ni}}\int^{2^j2^{iN\frac{\alpha}{1-\alpha}}}_{2^j}
	\|Q_sQ_t^2f\|_{L^2(\mathbb{R}^d)}^2\frac{dt}{t}\frac{ds}{s}\Big)^{1/2}\nonumber\\
	&&\qquad\times\Big(\sum_j\int^{2^j}_{2^{j-Ni}}\int^{2^j2^{iN\frac{\alpha}{1-\alpha}}}_{2^j}
	\Big\|Q_t^2Q_s\big(\sum_l\partial_m\widetilde{A}_{I_l}h_{s,l}\big)\Big\|_{L^2(\mathbb{R}^d)}^2
	\frac{dt}{t}\frac{ds}{s}\Big)^{1/2}
	\nonumber
\end{eqnarray}
\begin{eqnarray}
	&&\quad\lesssim i^{\frac{3}{2}}\|\Omega_i\|_{L^1(S^{d-1})}\|f\|_{L^2(\mathbb{R}^d)}
	\Big(\sum_j\int^{2^j}_{2^{j-Ni}}
	\Big\|\sum_lQ_s\big(\partial_m\widetilde{A}_{I_l}h_{s,l}\big)\Big\|_{L^2(\mathbb{R}^d)}^2
	\frac{ds}{s}\Big)^{1/2}\nonumber\\
	&&\quad\lesssim i^{2}\|\Omega_i\|_{L^1(S^{d-1})}\|f\|_{L^2(\mathbb{R}^d)}
	\Big(\sum_j\int^{2^j}_{2^{j-1}}
	\|M_qh\|_{L^2(\mathbb{R}^d)}^2\frac{ds}{s}\Big)^{1/2}\nonumber\\
	&&\quad\lesssim i^{2}\|\Omega_i\|_{L^1(S^{d-1})}\|f\|_{L^2(\mathbb{R}^d)}\|g\|_{L^2(\mathbb{R}^d)}.\nonumber
\end{eqnarray}
Collecting the estimates from (\ref{eq3.ri1}) to (\ref{eq3.viml5}) in all, we deduce that
\begin{eqnarray}\label{wmmequation2.9}
	IV	&&=\sum_i\sum_j\int^{2^j}_{2^{j-Ni}}\int^{(2^js^{-\alpha})^{\frac{1}{\alpha}}}_{2^j}\Big|\int_{\mathbb{R}^d}
	Q_t^4f(x)T_{\Omega,A,j}^iQ_s^4g(x)dx\Big|\frac{dt}{t}\frac{ds}{s}\\
	&&\nonumber\lesssim \Big(\sum_{i}i2^{-i}+\sum_ii^2\|\Omega_i\|_{L^1(S^{d-1})}\Big)\|f\|_{L^4(\mathbb{R}^d)}\|g\|_{L^2(\mathbb{R}^d)}\\&&\lesssim \|f\|_{L^2(\mathbb{R}^d)}\|g\|_{L^2(\mathbb{R}^d)}.\nonumber
\end{eqnarray}

{To show the estimate for $\rm V$, note that for each fixed $j$}, it holds that
\begin{eqnarray*}&&\big\{(s,\,t):\,0\leq s\leq 2^{j-Ni},\,2^j\leq t<(2^js^{-\alpha})^{\frac{1}{1-\alpha}}\big\}\\
	&&\quad\subset\big\{(s,\,t):\,2^j\leq t <\infty, 0<s\leq \min\{2^{j-Ni},\,(2^jt^{1-\alpha})^{\frac{1}{\alpha}}\}\big\}.
\end{eqnarray*}
It then follows that
\begin{eqnarray*}&&\sum_j\int^{2^{j-Ni}}_0\int^{(2^js^{-\alpha})^{\frac{1}{1-\alpha}}}_{2^j}
	\|Q_t^4f\|^2_{L^2(\mathbb{R}^d)} \frac{dt}{t}(2^{-j}s)^{\varepsilon}\frac{ds}{s}\\
	&&\quad\leq  2^{-Ni\varepsilon/2}\int^{\infty}_0\sum_{j:\,2^j\leq t}\int_0^{(2^jt^{\alpha-1})^{\frac{1}{\alpha}}}(2^{-j}s)^{\frac{\varepsilon}{2}}
	\frac{ds}{s}\|Q_t^4f\|^2_{L^2(\mathbb{R}^d)} \frac{dt}{t}\\
	&&\quad\lesssim 2^{-Ni\varepsilon/2}\|f\|_{L^2(\mathbb{R}^d)}^2,
\end{eqnarray*}
and
\begin{eqnarray*}&&\sum_j\int^{2^{j-Ni}}_0\int^{(2^js^{-\alpha})^{\frac{1}{1-\alpha}}}_{2^j}
	\frac{dt}{t}\|Q_s^3g\|^2_{L^2(\mathbb{R}^d)} (2^{-j}s)^{\varepsilon}\frac{ds}{s}\\
	&&\quad\leq   2^{-Ni\varepsilon/2}\int^{\infty}_0\Big(\sum_{j:\,2^j\geq s2^{Ni}}\int_{2^j}^{(2^js^{-\alpha})^{\frac{1}{1-\alpha}}}
	\frac{dt}{t}(2^{-j}s)^{\frac{\varepsilon}{2}}\Big)\|Q_s^3g\|^2_{L^2(\mathbb{R}^d)} \frac{ds}{s}\\
	&&\quad\lesssim 2^{-Ni\varepsilon/2}\|g\|_{L^2(\mathbb{R}^d)}^2,
\end{eqnarray*}
Thus, by Lemma \ref{wmmwmmlem3.4kkk},  we obtain 
\begin{eqnarray}\label{wmmequation2.10}
	V&&\le\sum_i\sum_j\int^{2^{j-Ni}}_0\int^{(2^js^{-\alpha})^{\frac{1}{1-\alpha}}}_{2^j}
	\|Q_sT^i_{\Omega,\,A;\,j}Q_t^4f\|_{L^2(\mathbb{R}^d)}\|Q_s^3g\|_{L^2(\mathbb{R}^d)}\frac{ds}{s}\frac{dt}{t} \\
	&&\le \sum_i2^i\Big(\sum_j\int^{2^{j-Ni}}_0\int^{(2^js^{-\alpha})^{\frac{1}{1-\alpha}}}_{2^j}
	\|Q_t^4f\|^2_{L^2(\mathbb{R}^d)} \frac{dt}{t}(2^{-j}s)^{\varepsilon}\frac{ds}{s}\Big)^{\frac{1}{2}}\nonumber\end{eqnarray}
\begin{eqnarray}	&&\qquad\times \Big(\sum_j\int^{2^{j-Ni}}_0\int^{(2^js^{-\alpha})^{\frac{1}{1-\alpha}}}_{2^j}
	\|Q_s^3g\|^2_{L^2(\mathbb{R}^d)} (2^{-j}s)^{\varepsilon}\frac{ds}{s}\frac{dt}{t}\Big)^{1/2}\nonumber\\&&
\quad\lesssim \sum_i2^i2^{-Ni\varepsilon/2} \|f\|_{L^2(\mathbb{R}^d)}\|g\|_{L^2(\mathbb{R}^d)}\nonumber.
\end{eqnarray}
Combining  estimates (\ref{wmmequation2.9})-(\ref{wmmequation2.10}) yields
\begin{eqnarray}\label{wmmequation2.11}
	&&\Big|\sum_i\sum_j\int^{\infty}_{2^j}\int_{0}^{(2^jt^{\alpha-1})^{1/\alpha}}\int_{\mathbb{R}^d}Q_s^4
	T_{\Omega,A,j}^iQ_t^4f(x)g(x)dx
	\frac{ds}{s}\frac{dt}{t}\Big|
	\lesssim \|f\|_{L^2(\mathbb{R}^d)}\|g\|_{L^2(\mathbb{R}^d)}.
\end{eqnarray}
Therefore,  by (\ref{E_2}), (\ref {wmmequation2.9}) and (\ref{wmmequation2.11}), it holds that
	\begin{eqnarray*}\label{wmmequation2.8}\Big|\sum_i\sum_j\int^{\infty}_0\int_{0}^t\int_{\mathbb{R}^d}\chi_{E_2(j,,s,t)}Q_s^4T_{\Omega,\,A;\,j}^iQ_t^4f(x)g(x)dx\frac{ds}{s}\frac{dt}{t}\Big| \lesssim IV+V \lesssim \|f\|_{L^2(\mathbb{R}^d)}\|g\|_{L^2(\mathbb{R}^d)},
\end{eqnarray*}
which gives the contribution of $E_2(j,s,t)$.

To finish the proof of \eqref{wmmlittlewood1}, it remains to show the contribution of the term $E_3^i(j,s,t)$. 

\noindent{\textbf {Contribution of $E_3(j,,s,t)$.}} 

Our aim is to prove 
\begin{eqnarray}\label{wmmequation2.12}
	&&\Big|\sum_i\sum_j\int^{\infty}_0\int_{0}^t\int_{\mathbb{R}^d}E_{3}(j,s,t)Q_s^4T_{\Omega,\,A;\,j}^iQ_t^4f(x)g(x)dx\frac{ds}{s}\frac{dt}{t}\Big|
	\lesssim \|f\|_{L^2(\mathbb{R}^d)}\|g\|_{L^2(\mathbb{R}^d)},\end{eqnarray}
where \begin{eqnarray} \label{2.29}
T_{\Omega _i,\,A;\,j}f(x)=\int_{\mathbb{R}^d}\frac{\Omega_i(x-y)}{|x-y|^{d+1}}(A(x)-A(y)-\nabla A(y)(x-y)\big) \phi_j(|x-y|)f(y)dy,\label{wmmhtg}
\end{eqnarray}
Since the sum of $i$ and the sum of $j$ are independent and the sum of $j$ depends only on the functions $\phi_j$ and $E_{3}(j,s,t), $ one may put  $\phi_j  \cdot E_{3}(j,s,t)$ together in the place of  $\phi_j$ in (\ref{2.29}), and temporary moves the summation over $j$ before $\phi_j  \cdot E_{3}(j,s,t)$, which indicates that it is possible to move the summation over $i$ inside the integral again before $\Omega_i$ to obtain $\Omega$. After that, one may move the sum of $j$ outside the integral. Therefore, to prove (\ref{wmmequation2.12}), it suffices to show that
\begin{eqnarray}\label{wmmequation2.13}
&&\sum_j\int^{\infty}_{2^j}\int^t_{(2^jt^{\alpha-1})^{1/\alpha}}\Big|\int_{\mathbb{R}^d}Q_s^4
{T}_{\Omega,A,j}Q_t^4f(x)g(x)dx\Big|
\frac{ds}{s}\frac{dt}{t}
\lesssim \|f\|_{L^2(\mathbb{R}^d)}\|g\|_{L^2(\mathbb{R}^d)}.\end{eqnarray}
To this purpose, we set
$$h^{(1)}(x,\,y)=\int\int\psi_s(x-y)\sum_{j:2^j\leq s^{\alpha}t^{1-\alpha}}K_{A,\,j}(z,\,u)\big[\psi_t(u-y)-\psi_t(x-\,y)\big]dudz.$$
Let $H^{(1)}$  be the integral operator corresponding to kernel $h^{(1)}$.  It then follows that
\begin{eqnarray}\label{2.227}&&\Big|\sum_j\int^{\infty}_{2^j}\int^t_{(2^jt^{\alpha-1})^{1/\alpha}}
	\int_{\mathbb{R}^d}Q_s^4T_{\Omega,A,j}Q_t^4f(x)g(x)dx
	\frac{ds}{s}\frac{dt}{t}\Big|\\
	&&\leq \int^{\infty}_0\int^t_0\|H^{(1)}Q_t^3f\|_{L^2(\mathbb{R}^d)}
	\|Q_s^3g\|_{L^2(\mathbb{R}^d)}\frac{ds}{s}\frac{dt}{t}\nonumber\end{eqnarray}
\begin{eqnarray}	&&\quad+\sum_j\int^{\infty}_{2^j}\int^t_{(2^jt^{\alpha-1})^{1/\alpha}}
	\int_{\mathbb{R}^d}(Q_sT_{\Omega,A;\,j}1)(x)Q_t^4f(x)Q_s^3g(x)dx\frac{ds}{s}\frac{dt}{t}.\nonumber
\end{eqnarray}
Applying Lemma \ref{wmmlem2.7} and reasoning as the same argument as in \cite[p. 1282]{hof1} give us that
$$h^{(1)}(x,\,y)\lesssim \big(\frac{s}{t}\big)^{\gamma}t^{-d}\chi_{\{(x,\,y):\,|x-y|\leq Ct\}}(x,\,y),$$
where $\gamma=(d+2)\alpha-d-1$. This in turn indicates that
$|H^{(1)}Q_tf(x)|\lesssim \big(\frac{s}{t}\big)^{\gamma}M(Q_tf)(x).
$
Therefore
\begin{eqnarray}\label{wmmequation2.13} \int^{\infty}_0\int^t_0\|H^{(1)}Q_t^3f\|_{L^2(\mathbb{R}^d)}\|Q_s^3g\|_{L^2(\mathbb{R}^d)}
	\frac{ds}{s}\frac{dt}{t}
&&\lesssim \Big( \int^{\infty}_0\int^t_0\big(\frac{s}{t}\big)^{\gamma}\|M(Q_t^3f)\|_{L^2(\mathbb{R}^d)}^2
	\frac{ds}{s}\frac{dt}{t}\Big)^{\frac{1}{2}}\\
	&&\quad\times\Big( \int^{\infty}_0\|Q_s^3g\|_{L^2(\mathbb{R}^d)}^2\int_{s}^{\infty}\big(\frac{s}{t}\big)^{\gamma}
	\frac{dt}{t}\frac{ds}{s}\Big)^{\frac{1}{2}}\nonumber\\ && \lesssim \|f\|_{L^2(\mathbb{R}^d)}\|g\|_{L^2(\mathbb{R}^d)}.\nonumber
\end{eqnarray}
It remains to show the coresponding estimate for the second term on the rightside of (\ref{2.227}).

Let $F_x^j(s,t)=(Q_sT_{\Omega,A;\,j}1)(x)Q_t^4f(x)Q_s^3g(x)$. Then
 \begin{align} \label{split} {}&
\int^{\infty}_{2^j}\int^t_{(2^jt^{\alpha-1})^{1/\alpha}}F_x^j(s,t)\frac{dsdt}{st}	\\&=\int^{\infty}_0\int_0^tF_x^j(s,t)\frac{dsdt}{st}-\int^{2^j}_0\int_0^t~F_x^j(s,t)\frac{dsdt}{st}-\int^{\infty}_{2^j}\int_0^{(2^jt^{\alpha-1})^{1/\alpha}}~F_x^j(s,t)\frac{dsdt}{st}\nonumber
	\end{align}
Therefore, it is sufficient to consider the contributions of each terms in equation  (\ref{split}) to the second term in (\ref{2.227}).

Consider the first term in (\ref{split}).
Let
$P_s=\int^{\infty}_sQ_t^4\frac{dt}{t}.$
Han and Sawyer \cite{hansaw} observed that the kernel $\Phi$ of the convolution operator
$P_s$ is a radial bounded function with bound $cs^{-d}$, supported on a ball of radius $Cs$ and has
integral zero. Therefore, it is easy to see that $\Phi$ is a schwartz function. Since $P_sg=\Phi_s * f$, it then follows from the Littlewood-Paley theory that
$$\int^{\infty}_0\|P_sg\|_{L^2(\mathbb{R}^d)}^2\frac{ds}{s}\lesssim \|g\|_{L^2(\mathbb{R}^d)}^2.$$
On the other hand, whenever $\Omega\in L^1(\mathbb{R}^d)$, it was shown in \cite[p.121, Lemma 4.1] {hof} that  $T_{\Omega,A}1\equiv b\in {\rm BMO}(\mathbb{R}^d)$. Therefore, by \cite[p.114, (3.1)] {hof}, $\int^{\infty}_0Q_s^3(Q_sbP_s)\frac{ds}{s}$ defines an operator 
which is bounded on $L^2(\mathbb{R}^d)$. However, we can't use this boundedness directly in our case, since once using  H\"older's inequality, we have to put the absolute value inside the integral and the {$L^2(\mathbb{R}^d)$} boundedness may fail in this case. To overcome this obstacle, we consider to use the property of Carleson measure, lemma \ref{graf} from the book of Grafakos.

Note that $|Q_s
T_{\Omega,\,A}1(x)|^2\frac{dxds}{s}$ is a Carleson measure since $T_{\Omega,\,A}1\in {\rm BMO}(\mathbb{R}^d)$. By   H\"older's inequality, Lemma \ref{graf}, it yields that
	\begin{eqnarray}&&\Big|\int^{\infty}_0\int_0^t\int_{\mathbb{R}^d}\sum_jQ_s
	T_{\Omega,\,A;j}1(x)Q_t^4f(x)Q^3_sg(x)dx\frac{ds}{s}\frac{dt}{t}\Big|\\
	&&\quad\lesssim {\bigg(\int^{\infty}_0\int_{\mathbb{R}^d}|Q_s^3g(x)|^2dx\frac{ds}{s}\bigg)^{\frac{1}{2}}\bigg(\int_{\mathbb {R}_+^{n+1}}|P_sf(x)|^2|Q_s
	T_{\Omega,\,A}1(x)|^2\frac{dxds}{s}\bigg)^{\frac{1}{2}}}\nonumber\\
	&&
	\quad\lesssim \|f\|_{L^2(\mathbb{R}^d)} \|g\|_{L^2(\mathbb{R}^d)}\nonumber
	\end{eqnarray}

On the other hand, by Lemma \ref{wmmwmmlem3.3}, one gets $$\|Q_sT_{\Omega,\,A;\,j}1\|_{L^{\infty}(\mathbb{R}^d)}\lesssim \|\Omega\|_{L^1(\mathbb{S}^{d-1})}(2^{-j}s)^{\epsilon}.$$
Denote by
$D^1_{j,s,t}=\{(j,s,t): s\leq t\leq 2^j\}$, $D^2_{j,s,t}=\{(j,s,\,t):\,s\leq t,s^{\alpha}t^{1-\alpha}\leq 2^j\leq t\}$.
It then follows that
\begin{eqnarray}\label{wmmequation2.15}&&\sum_j\int^{2^j}_0\int_0^t\int_{\mathbb{R}^d}
	\Big|Q_sT_{\Omega,\,A;j}1(x)Q_t^4f(x)Q_s^3g(x)\Big|dx\frac{ds}{s}\frac{dt}{t}\\
	&&\qquad+\sum_j\int^{\infty}_{2^j}\int_0^{(2^jt^{\alpha-1})^{1/\alpha}}\int_{\mathbb{R}^d}
	\Big|Q_sT_{\Omega,\,A;j}1(x)Q_t^4f(x)Q^3_sg(x)\Big|dx\frac{ds}{s}\frac{dt}{t}\nonumber\\
	&&\quad\lesssim\sum_{i=1}^2\bigg\{\Big(\int_0^\infty\int_0^\infty \sum_j(2^{-j}s)^{\epsilon}\chi_{D_{j,s,t}^i}(j,s,t)\|Q_tf\|_{L^2(\mathbb{R}^d)}^2\frac{ds}{s}\frac{dt}{t}\Big)^{1/2}\nonumber\\
	&&\quad\quad\times
	\Big(\int_0^\infty\int_0^\infty \sum_j(2^{-j}s)^{\epsilon}\chi_{D_{j,s,t}^i}(j,s,t)\|Q_sgf\|_{L^2(\mathbb{R}^d)}^2\frac{ds}{s}\frac{dt}{t}\Big)^{1/2}\bigg\}\nonumber\\
	&&\quad\lesssim \|f\|_{L^2(\mathbb{R}^d)}\|g\|_{L^2(\mathbb{R}^d)},\nonumber
\end{eqnarray}
where in the last inequality, we used the property (\ref{lpp}).

Combining (\ref{split})-(\ref{wmmequation2.15}), we have		$$\Big|\sum_j\int^{\infty}_{2^j}\int^t_{(2^jt^{\alpha-1})^{1/\alpha}}\int_{\mathbb{R}^d}(Q_sT_{\Omega,A;\,j}1)(x)Q_t^4f(x)Q^3_sg(x)dx\frac{ds}{s}\frac{dt}{t}
\Big|\lesssim \|f\|_{L^2(\mathbb{R}^d)}\|g\|_{L^2(\mathbb{R}^d)},
$$
which, together with  (\ref{wmmequation2.13}), leads to  (\ref{wmmequation2.12}). This finishes the proof of $E_3(j,s,t)$, and also completes the proof of inequality (\ref{wmmlittlewood1}).
\vspace{0.1cm}

To finish the proof of Theorem 1.1, it remains to  show the estimate (\ref{wmmlittlewood2}). Oberve that 
\begin{eqnarray*} &&\int^{\infty}_0\int_{t}^{\infty}\int_{\mathbb{R}^d}Q_s^4{T}_{\Omega,\,A}Q_t^4f(x)g(x)dx
	\frac{ds}{s}\frac{dt}{t}=-\int^{\infty}_0\int_{0}^{s}\int_{\mathbb{R}^d}Q_t^4\widetilde{T}_{\widetilde{\Omega},\,A}Q_s^4g(x)f(x)dx	\frac{dt}{t}
	\frac{ds}{s},
\end{eqnarray*}
where $\widetilde{\Omega}(x)=\Omega(-x)$ and $\widetilde{T}_{\widetilde{\Omega},\,A}$ is   the operator defined by (\ref{equation1.dualoperator}), with $\Omega$ replaced by $\widetilde{\Omega}$.  Let $T_{\widetilde{\Omega},\,m}$ be the operator defined by
$$T_{\widetilde{\Omega},\,m}h(x)={\rm p.\,v.}\int_{\mathbb{R}^d}\frac{\widetilde{\Omega}(x-y)(x_m-y_m)}{|x-y|^{d+1}}h(y)dy.
$$
It then follows that		
$$\widetilde{T}_{\widetilde{\Omega},\,A}h(x)=T_{\widetilde{\Omega},\,A}h(x)+\sum_{m=1}^d[\partial_mA,\,
T_{\widetilde{\Omega},\,m}]h(x).
$$
Inequality (\ref{wmmlittlewood1}) tells us that		\begin{eqnarray}\label{wmmequation2.16}\Big|\int^{\infty}_0\int_{0}^t\int_{\mathbb{R}^d}Q_s^4
	{T}_{\widetilde{\Omega},\,A}Q_t^4g(x)f(x)dx\frac{dt}{t}\frac{ds}{s}\Big|\lesssim \|f\|_{L^2(\mathbb{R}^d)}\|g\|_{L^2(\mathbb{R}^d)}.
\end{eqnarray}
For each fixed $m$ with $1\leq m\leq d$, by duality, involving Lemma \ref{wmmlem2.6} and H\"older's inequality may lead to
$$\aligned {}&\Big|\int^{\infty}_0\int_{0}^t\int_{\mathbb{R}^d}Q_s^4[\partial_mA,\,T_{\widetilde{\Omega},\,m}]Q_t^4g(x)f(x)dx
	\frac{ds}{s}\frac{dt}{t}\Big|\\	&\lesssim\int^{\infty}_0\|[\partial_mA,\,T_{\widetilde{\Omega},\,m}]Q_s^4f\|_{L^2(\mathbb{R}^d)}
	\Big\|\int_{s}^{\infty}Q_t^4
	g\frac{dt}{t}\Big\|_{L^2(\mathbb{R}^d)}
	\frac{ds}{s}\\
	&\quad\lesssim \Big(\int^{\infty}_0\|Q_s^4f\|^2_{L^2(\mathbb{R}^d)}\frac{ds}{s}\Big)^{1/2}
	\Big(\int^{\infty}_0\|P_sg\|^2_{L^2(\mathbb{R}^d)}
	\frac{ds}{s}\Big)^{1/2}\\
&\quad\lesssim \|f\|_{L^2(\mathbb{R}^d)}\|g\|_{L^2(\mathbb{R}^d)}.
\endaligned
$$
This estimate, together with (\ref{wmmequation2.16}), leads to (\ref{wmmlittlewood2}) and then completes the proof of Theorem \ref{wmmthm1.1}.

\vspace{0.3cm}

	\section{proof of Theorem \ref{thm1.4}}\label{Sec3}
	
	This section is devoted to prove Theorem \ref{thm1.4}, the weak type endpoint estimates for $T_{\Omega,\,A}$ and $\widetilde{T}_{\Omega,A}$.
	To this end, we first introduce the definition of standard dyadic grid. Recall that  the standard dyadic grid in $\mathbb{R}^d$, denoted  by $\mathcal{D}$, consists of all cubes of the form $$2^{-k}([0,\,1)^d+j),\,k\in  \mathbb{Z},\,\,j\in\mathbb{Z}^d.$$
	For each fixed $j\in\mathbb{Z}$, set
	$\mathcal{D}_j=\{Q\in \mathcal{D}:\, \ell(Q)=2^j\}$.
	
	The following lemma plays an important role in our analysis.
		\begin{lemma}{\rm (}\cite{cohen}{\rm )}\label{lem2.3}
		Let $A$ be a function on $\mathbb{R}^d$ with derivatives of order one in $L^q(\mathbb{R}^d)$ for some $q\in (d,\,\infty]$. Then
		$$|A(x)-A(y)| \lesssim|x-y|\Big(\frac{1}{|I_{(x,|x-y|)}|}\int_{I_{(x,|x-y|)}}|\nabla A(z)|^q{\rm d}z\Big)^{\frac{1}{q}},$$
		where
		$I_{(x,|x-y|)}$ is a cube which is centered at $x$  with length $2|x-y|.$
	\end{lemma}

	\subsection{ Proof of \eqref{equation2.2} in Theorem \ref{thm1.4}} \quad
	
	The key ingredient of our proof  lies in  the step of dealing with the bad part of the Calder\'on-Zygmund decomposition of $f$. By homogeneity, it suffices to prove  \eqref{equation2.2} for the case $\lambda=1$.
	Applying the Calder\'on-Zygmund decomposition to $|f|\log ({\rm e}+|f|)$ at level $1$, we can obtain a collection
	of non-overlapping closed dyadic cubes $\mathcal{S}=\{\mathbb {L}\}$, such that 
	\vspace{0.2cm}
	
	\begin{enumerate}
		\item [(i)]$\|f\|_{L^{\infty}(\mathbb{R}^d\backslash \cup_{\mathbb {L}\in\mathcal{S}}\mathbb {L})}\lesssim 1;$
		\item [(ii)]$\int_{\mathbb {L}}|f(x)|\log({\rm e}+|f(x)|)dx \lesssim |\mathbb {L}|; $ \item [(iii)]$\sum_{\mathbb {L}\in\mathcal{S}}|\mathbb {L}|\lesssim \int_{\mathbb{R}^d}|f(x)|\log ({\rm e}+|f(x)|)dx$.\end{enumerate}
		\vspace{0.2cm}
	Let $g$ be the good part and $b$ be the bad part of the decomposition of $f$, which are defined by $$g(x)=f(x)\chi_{\mathbb{R}^d\backslash \cup_{\mathbb{L}\in\mathcal{S}}\mathbb{L}}(x)+\sum_{\mathbb{L}\in\mathcal{S}}\langle f\rangle_{\mathbb{L}}\chi_{\mathbb{L}}(x) \quad \hbox{\ and \ } \quad b(x)=\sum_{\mathbb{L}\in\mathcal{S}}(f-\langle f\rangle_{\mathbb{L}})\chi_{\mathbb{L}}(x)=\sum_{\mathbb{L}\in \mathcal{S}}b_{\mathbb{L}}(x).$$
	It is easy to see that $\|g\|_{L^{\infty}(\mathbb{R}^d)}\lesssim 1$, and for $E=\cup_{\mathbb{L}\in\mathcal{S}}100d\mathbb{L}$, it holds that
	$$|E|\lesssim \int_{\mathbb{R}^d}|f(x)|\log ({\rm e}+|f(x)|)dx.$$
	The $L^2(\mathbb{R}^d)$ boundedness of $T_{\Omega,\,A}$ then yields that
	\begin{equation}
		\big|\{x\in\mathbb{R}^d:\, |T_{\Omega,\,A}g(x)|\geq 1/2\}\big|\lesssim \|T_{\Omega,\,A}g\|_{L^2(\mathbb{R}^d)}^2\lesssim \|g\|^2_{L^2(\mathbb{R}^d)}\lesssim\|f\|_{L^1(\mathbb{R}^d)}.\label{good}
	\end{equation}
	Therefore, it is sufficient to show that 
	\begin{equation}
		\big|\{x\in\mathbb{R}^d:\, |T_{\Omega,\,A}b(x)|\geq 1/2\}\big| \lesssim\|f\|_{L^1(\mathbb{R}^d)}.\label{bad}
	\end{equation}
	
	To prove (\ref{bad}), smooth truncation method will play important roles and will be used several times.
	Let $\phi$ be a smooth radial nonnegative function on $\mathbb{R}^d$ with ${\rm supp}\, \phi\subset\{x:\frac{1}{4}\leq |x|\leq 1\}$ and
	$\sum_s\phi_s(x)=1$ with $\phi_s(x)=\phi(2^{-s}x)$ for all $x\in \mathbb{R}^d\backslash\{0\}$. Set $\mathcal{S}_{j}=\{\mathbb{L}\in\mathcal{S}:\, \ell(\mathbb{L})=2^j\}$.
	
	\medskip
	Then, we have $$\aligned
	\int_{\mathbb{R}^d}\frac{\Omega(x-y)}{|x-y|^{d+1}}(A(x)-A(y))b(y)dy&=\int_{\mathbb{R}^d}\frac{\Omega(x-y)}{|x-y|^{d+1}}(A(x)-A(y))\sum_s\phi_{s}(x-y)\sum_j\sum_{\mathbb{L}\in \mathcal{S}_{s-j}}b_\mathbb{L}(y)dy\\&=
	\int_{\mathbb{R}^d}\frac{\Omega(x-y)}{|x-y|^{d+1}}(A(x)-A(y))\sum_j\sum_s\phi_{s}(x-y)\sum_{\mathbb{L}\in \mathcal{S}_{s-j}}b_\mathbb{L}(y)dy\\&=
	\sum_j\sum_s\sum_{\mathbb{L}\in \mathcal{S}_{s-j}}
	T_{\Omega,\,A;\,s,j}b_\mathbb{L}(x),
	\endaligned
	$$
	where \begin{eqnarray}\label{equation2.4}T_{\Omega,\,A;\,s,j}b_\mathbb{L}(x)=\int_{\mathbb{R}^d}\phi_{s}(x-y)\frac{\Omega(x-y)}{|x-y|^{d+1}}(A(x)-A(y))b_\mathbb{L}(y)dy\end{eqnarray}
	
	Let
	$A_{\mathbb{L}}(y)=A(y)-\sum_{n=1}^d\langle\partial _n A\rangle_\mathbb{L}y_n.$
	A trivial computation {leads to} the fact that
	$$A_{\mathbb{L}}(x)-A_\mathbb{L}(y)-\nabla A_\mathbb{L}(y)\cdotp(x-y)=A(x)-A(y)-\nabla A(y)\cdotp(x-y).
	$$
	Now we use the method of smoothness truncation for the first time,
$T_{\Omega,A}$ can be written as
	$$T_{\Omega,A}b(x)=\sum_{j}\sum_{s}\sum_{\mathbb{L}\in \mathcal{S}_{s-j}}T_{\Omega,\,A_\mathbb{L};s,j}b_{\mathbb{L}}(x)-\sum_{n=1}^dT_{\Omega}^n\Big(\sum_{\mathbb{L}\in\mathcal{S}}b_\mathbb{L}\partial_nA_{\mathbb{L}}\Big)(x),$$
	where
	$$T_{\Omega}^nh(x)={\rm p.\, v.}\int_{\mathbb{R}^d}\frac{\Omega(x-y)}{|x-y|^{d+1}}(x_n-y_n)h(y)dy, \quad \hbox{	for $1\leq n\leq d$.}
	$$
	Fixed $1\leq n\leq d$, since the kernel $ {\Omega(x)x_n}{|x|^{-1}}$ is still in $L\log L {(\mathbb {S}^{d-1})}$, homogenous of degree zero and satisfies the vanishing condition on the unit sphere, by the weak endpoint estimate of the  operators $T_{\Omega}^n$ (see \cite{se} or \cite{dinglai}), it follows that
	\begin{eqnarray}\label{2.3}
		\qquad\Big|\Big\{x\in\mathbb{R}^d\backslash E:\Big| T_{\Omega}^n\Big(\sum_{\mathbb{L}\in\mathcal{S}}b_\mathbb{L}\partial_nA_\mathbb{L}\Big)(x)\Big|>\frac{1}{4d}\Big\}\Big|&\lesssim&
		\Big\| \sum_{\mathbb{L}\in\mathcal{S}}b_\mathbb{L}\partial_nA_\mathbb{L}\Big\|_{L^1(\mathbb{R}^d)}\label{errorterm}\\
		&\lesssim&\sum_{\mathbb{L}\in\mathcal{S}}|\mathbb{L}|\|b_\mathbb{L}||_{L\log L,\,\mathbb{L}}\nonumber\\
		&\lesssim&\int_{\mathbb{R}^d}|f(x)|\log ({\rm e}+|f(x)|)dx,\nonumber
	\end{eqnarray}
	where in the last inequality, we have used the fact that $\|b_\mathbb{L}\|_{L\log L,\,\mathbb{L}}\lesssim 1$ for each cube $\mathbb{L}\in\mathcal{S}$.
	
		\smallskip
	Therefore, to prove inequality (\ref{equation2.2}), by (\ref{good}), (\ref{bad}) and (\ref{2.3}), it is sufficient to show that
	\begin{eqnarray}\label{equation2.6}&&\Big|\Big\{x\in\mathbb{R}^d\backslash E:\, \Big|\sum_{j}\sum_{s}\sum_{\mathbb{L}\in \mathcal{S}_{s-j}}T_{\Omega,A_\mathbb{L};s,j}b_\mathbb{L}(x)\Big|>1/4\Big\}\Big|\lesssim {\|f\|}_{L^1(\mathbb{R}^d)}.\end{eqnarray}
	
	In order to prove inequality (\ref{equation2.6}), we first give some estimate for $\sum_s\sum_{\mathbb{L}\in\mathcal{S}_{s-j}}T_{\Omega,\,A_\mathbb{L};s,j}b_{\mathbb{L}}$.
	For this purpose, we need to introduce some notations. 

	For $\mathbb{L}\in\mathcal{S}_{s-j}$, $s$, $j\,\in\mathbb{Z}$ with $j\geq \log_2(100d/2)=:j_0$. Let $L_{j,1}=2^{j+2}d\mathbb{L}$, $L_{j,2}=2^{j+4}d\mathbb{L}$, $L_{j,3}=2^{j+6}d\mathbb{L}$, and  $y^j_{\mathbb{L}}$ be a point on the boundary of $L_{j,3}$. Set
	$$A_{\varphi_{\mathbb{L}}}(y)=\varphi_{\mathbb{L}}(y)\big(A_{\mathbb{L}}(y)-A_{\mathbb{L}}(y^j_{\mathbb{L}})),$$
	where $\varphi_{\mathbb{L}}\in C^{\infty}_c(\mathbb{R}^d)$, ${\rm supp}\,\varphi_{\mathbb{L}}\subset L_{j,1}$, $\varphi_{\mathbb{L}}\equiv 1$ on $3\cdot2^{j}d\mathbb{L}$, and
	$\|\nabla \varphi_{{\mathbb{L}}}\|_{L^{\infty}(\mathbb{R}^d)}\lesssim 2^{-s}.
	$
	\vspace{0.1cm
	}
	Let $y_0$ be the center point of $\mathbb{L}$. Observe that for $x\in\mathbb{R}^d\backslash E$, $j\leq j_0$, $y\in \mathbb{L}$, we have $|x-y|\geq |x-y_0|-|y-y_0|>2^s$. The support condition of $\phi$ then implies that $T_{\Omega,A_{\mathbb{L}};s,j}b_\mathbb{L}(x)=0$ if $j\leq j_0$.
	For $y\in \mathbb{L}\in \mathcal{S}_{s-j}$, $s,\,j\in\mathbb{Z}$ with $j> j_0$, we have $\varphi_{\mathbb{L}}(y)= 1$.	By the support condition of $\phi$, it follows that $|x-y_0|\leq |x-y|+|y-y_0|\leq 1.5d2^s $. Hence $x\in 3\cdot2^{j}d\mathbb{L}$ and $\varphi_{\mathbb{L}}(x)= 1$. Collecting these facts in all, it follows that $$\phi_{s}(x-y)(A_{{\mathbb{L}}}(x)-A_{{\mathbb{L}}}(y))=\phi_{s}(x-y)(A_{\varphi_{\mathbb{L}}}(x)-A_{\varphi_{\mathbb{L}}}(y)).$$
	
	The kernel $\Omega$ will be decomposed into disjoint forms as in Section {\ref{Sect 2}} as follows:
	$$\Omega_0(\theta )=\Omega(\theta )\chi_{E_0}(\theta ),\quad \,\Omega_k(\theta )=\Omega(\theta )\chi_{E_k}(\theta )\,\,(k\in\mathbb{N}),$$
	where 	$E_0=\{\theta \in \mathbb{S}^{d-1}:\, |\Omega(\theta )|\leq 1\} \ \hbox{and\ } E_k=\{\theta \in \mathbb{S}^{d-1}:\,2^{k-1}<|\Omega(\theta )|\leq 2^k\}$ for $ k\in\mathbb{N}.$
	
	Let the operator $ T^i_{\Omega,A_{\mathbb{L}};\,s,j}b_\mathbb{L}$ be defined in the same form as   $ T_{\Omega,A_{\mathbb{L}};\,s,j}b_\mathbb{L}$, with $\Omega$ replaced by $\Omega_i$.
	Then we can  divide the summation of $T_{\Omega,A_{\mathbb{L}};\,s}b_\mathbb{L}$ into two terms as follows
	\begin{eqnarray*}
	\sum_{ j> j_0}\sum_{s}\sum_{\mathbb{L}\in \mathcal{S}_{s-j}} T_{\Omega,A_{\mathbb{L}};\,s,j}b_\mathbb{L}(x) 
		&&=\sum\limits_{i=0}^{\infty}  \sum_{ j> j_0}\sum_{s}\sum_{\mathbb{L}\in \mathcal{S}_{s-j}} T^i_{\Omega,A_{\mathbb{L}};\,s,j}b_\mathbb{L}(x) 	\\
		&&=\sum\limits_{i=0}^{\infty}  \sum_{j_0< j\leq Ni} \sum_{s}\sum_{\mathbb{L}\in\mathcal{S}_{s-j}}  T^i_{\Omega,A_{\mathbb{L}};\,s,j}b_\mathbb{L}(x)\\&& \quad+\sum\limits_{i=0}^{\infty}  \sum_{ j> Ni}\sum_{s}\sum_{\mathbb{L}\in\mathcal{S}_{s-j}}  T^i_{\Omega,A_{\mathbb{L}};\,s,j}b_\mathbb{L}(x) \\
		&&:={\rm D}_1(x)+{\rm D}_2(x),
	\end{eqnarray*}
	where $N$ is some constant which will be used at the end of this subsection.
	Thus, to prove inequality (\ref{equation2.2}), it is enough to consider the contributions of $D_1$ and $D_2$, respectively. 
	
	\vspace{0.3cm
	}
	\noindent{\bf {Contribution of ${\rm D_1}$}}.
	
	\vspace{0.3cm
	}
	
	We first claim that
	if $\mathbb{L}\in \mathcal{S}_{s-j}$, then
	$$\big|T^i_{\Omega,A_{\mathbb{L}};\,s,j}b_\mathbb{L}(x)\big|\lesssim j
	\int_{\{2^{s-2}\leq |y|\leq 2^{s+2}\}}\frac{|\Omega_i(y')|}{|y|^{d}}|b_\mathbb{L}(x-y)|dy.$$
	{This claim is a consequence of the following lemma, which will also be used several times later.}
	
	\begin{lemma}\label{lem2.x} Let $A$ be a function in $\mathbb{R}^d$ with derivatives of order one in  ${\rm BMO}(\mathbb{R}^d)$. Let $s,\,j\in\mathbb{Z}$ and $			\mathbb{L}\in \mathcal{S}_{s-j}$ with $j>j_0$ and let $R_{s,\mathbb{L};j}(x,y)$ be the function on $\mathbb{R}^d\times \mathbb{R}^d$ defined by
		\begin{eqnarray*}R_{s,\mathbb{L};j}(x,y)=\phi_{s}(x-y)\frac{A_{\varphi_{\mathbb{L}}}(x)-A_{\varphi_{\mathbb{L}}}(y)}{|x-y|^{d+1}}.
		\end{eqnarray*} Then,  $R_{s,\mathbb{L};j}$ enjoys the properties that
		\begin{itemize}
			\item[\rm (i)] For any $x,\, y\in\mathbb{R}^d$, $$\big|R_{s,\mathbb{L};j}(x,y)\big|\lesssim \frac{j}{|x-y|^{d}}\chi_{\{2^{s-2}\leq |x-y|\leq 2^{s+2}\}}(x,\,y);$$
			\item[\rm (ii)] For any $x,\,x'\in \mathbb{R}^d$ and $y\in \mathbb{L}$ with $|x-y|>2|x-x'|$,
			$$\big|R_{s,\mathbb{L};j}(x,y)-R_{s,\mathbb{L};j}(x',y)\big|\lesssim\frac{|x-x'|}{|x-y|^{d+1}}\Big(j+\Big|\log \big({2^{s-j}}{|x-x'|^{-1}}\big)\Big|\Big);$$
			\item [\rm (iii)] For any $x,\,y'\in\mathbb{R}^d$ and $y\in \mathbb{L}$ with $|x-y|>2|y-y'|$,
			$$\big|R_{s,\mathbb{L};j}(x,y)-R_{s,\mathbb{L};j}(x,\,y')\big|\lesssim \frac{|y-y'|}{|x-y|^{d+1}}\Big(j+\Big|\log \big({2^{s-j}}{|y-y'|^{-1}}\big)\Big|\Big).$$
		\end{itemize}
	\end{lemma}
	\begin{proof} We first prove \rm (i).  It is obvious that ${\rm supp}\,R_{s,\mathbb{L};j} \subset L_{j,2}\times L_{j,1}$. Fixed $x\in L_{j,1}$, we know
		that 
		{$2^{s-j}<|x-y_\mathbb{L}^j|$} and
		$$\aligned
		\big|\langle \nabla A\rangle_{\mathbb{L}}-\langle \nabla A\rangle_{I_{(x,|x-y_\mathbb{L}^j|)}}\big|&\le\big|\langle \nabla A\rangle_{\mathbb{L}}-\langle \nabla A\rangle_{I_{(x, 2^{s-j})}}\big|+{\big |\langle \nabla A\rangle_{I_{(x, 2^{s-j})}}-\langle \nabla A\rangle_{I_{(x,|x-y_\mathbb{L}^j|)}}\big|}.\endaligned			
		$$ Note that if $x\in 4\mathbb{L}$, then ${I_{(x, 2^{s-j})}}\subset 8\mathbb{L}$ and it holds that $$\aligned
		\big|\langle \nabla A\rangle_{\mathbb{L}}-\langle \nabla A\rangle_{I_{(x, 2^{s-j})}}\big|&\le
		\big|\langle \nabla A\rangle_{\mathbb{L}}-\langle \nabla A\rangle_{8\mathbb{L}}\big|+\big|\langle \nabla A\rangle_{8\mathbb{L}}-\langle \nabla A\rangle_{I_{(x, 2^{s-j})}}\big| \lesssim	 1.	\endaligned
		$$
		
		If $x\in L_{j,1}\backslash  4\mathbb{L}$, then the center of $\mathbb{L}$ and the center of ${I_{(x, 2^{s-j})}}$ are at a distance of $a2^{s-j}$ with $a>1$. Hence, the results in \cite[Proposition 3.1.5, p. 158 and 3.1.5-3.1.6, p.166.]{gra} gives that
		
		$$\big|\langle \nabla A\rangle_{\mathbb{L}}-\langle \nabla A\rangle_{I_{(x,2^{s-j})}}\big| 
		\lesssim j\quad 
		\hbox{and} \quad \big|\langle \nabla A\rangle_{I_{(x, 2^{s-j})}}-\langle \nabla A\rangle_{I_{(x,|x-y_\mathbb{L}^j|)}}\big| 
		\lesssim j,$$
		since $2^s<|x-y_\mathbb{L}^j|<2^{s+5+d^2}$.
		
		Therefore, for $x\in \mathbb{L}_{j,1}$, it holds that
		\begin{align}
			\big|\langle \nabla A\rangle_{\mathbb{L}}-\langle \nabla A\rangle_{I_{(x,|x-y_\mathbb{L}^j|)}}\big|\lesssim j. \label{john}
		\end{align} 
		Lemma \ref{lem2.3}, together with John-Nirenberg inequality then gives that \begin{align}\label{jocz}
			|A_{\varphi_\mathbb{L}}(x)|&\lesssim |x-y_{\mathbb{L}}^j|\Big(\frac{1}{|I_{ (x,|x-y_\mathbb{L}^j|)}|}\int_{I_{ (x,|x-y_\mathbb{L}^j|)}}|\nabla A(z)-\langle \nabla A\rangle_{\mathbb{L}}|^qdz\Big)^{1/q}\lesssim j2^s,
		\end{align}
		which finishes the proof of (i).
		
		Now we give the proof of \rm (ii).
		For any $x,\,x'\in \mathbb{R}^d$ and $y\in \mathbb{L}$ with $|x-y|>2|x-x'|$, it is easy to see that
		%	$|R_{s,\mathbb{L};j}(x,y)-R_{s,\mathbb{L};j}(x',y)|\not =0$  if  $x,\,x'\in \mathbb{L}_{j,1}$. Since
		\begin{itemize}
			\item[\rm (1)] if $x\notin L_{j,1}$ and $x'\notin L_{j,1}$, then	$R_{s,\mathbb{L};j}(x,y)=R_{s,\mathbb{L};j}(x',y) =0$;
			\item[\rm (2)]  if $x\notin L_{j,1}$, then $x'\notin 3\cdot2^{j}d\mathbb{L}$, hence $R_{s,\mathbb{L};j}(x,y)=R_{s,\mathbb{L};j}(x',y) =0$;
			\item[\rm (3)]  if $x'\notin L_{j,1}$, then $x\notin 3\cdot2^{j}d\mathbb{L}$, hence  	$R_{s,\mathbb{L};j}(x,y)=R_{s,\mathbb{L};j}(x',y) =0$.
		\end{itemize}

		If $z\in I_{(x,|x-x'|)}$, another  application of Lemma \ref{lem2.3} and John-Nirenberg inequality {indicates}
		\begin{eqnarray}
			|\nabla A_{\varphi_\mathbb{L}}(z)|&\lesssim &2^{-s}|A_{\mathbb{L}}(z)-A_{\mathbb{L}}(y^j_{L})|+|\nabla A(z)-\langle \nabla A\rangle_{\mathbb{L}}|\label{ns1}\\
			&\lesssim&j+|\nabla A(z)-\langle \nabla A\rangle_{\mathbb{L}}|,\nonumber
		\end{eqnarray}
		and the similar method as what was used in the proof of \eqref{john} further implies that \begin{eqnarray}
			|\langle \nabla A\rangle_{\mathbb{L}}-\langle \nabla A\rangle_{I_{(x,|x-x'|)}}|\label{ns2}&&\leq |\langle \nabla A\rangle_{\mathbb{L}}-\langle \nabla A\rangle_{I_{(x,2^{s-j})}}|+|\langle \nabla A\rangle_{I_{(x,|x-x'|)}} - \langle \nabla A\rangle_{I_{(x,2^{s-j})}}|\\
			&&\lesssim  \log 2^j+\Big|\log \big({2^{s-j}}{|x-x'|^{-1}}\big)\Big|.\nonumber
		\end{eqnarray}
		By Lemma \ref{lem2.3}, \eqref{ns1} and  \eqref{ns2}, we have 
		\begin{eqnarray}\label{czsz}
		|A_{\varphi_{\mathbb{L}}}(x)-A_{\varphi_{\mathbb{L}}}(x')|\label{wtw}
			&&\lesssim  |x-x'|\Big(\frac{1}{|I_{(x,|x-x'|)}|}\int_{I_{(x,|x-x'|)}}|\nabla A_{\varphi_{\mathbb{L}}}(z)|^qdz\Big)^{\frac{1}{q}}\nonumber\\
			&&\lesssim  |x-x'|\Big(j+\frac{1}{|I_{(x,|x-x'|)}|}\int_{I_{(x,|x-x'|)}}|\nabla A(z)-\langle \nabla \nonumber
			\end{eqnarray}
		Similarly as in	\eqref{jocz}, we obtain 
		\begin{eqnarray}
|A_{\varphi_{\mathbb{L}}}(x)-A_{\varphi_{\mathbb{L}}}(x')|\label{wtw}&&\lesssim 	|x-x'|\big(j+|\langle \nabla A\rangle_{\mathbb{L}}-\langle \nabla A\rangle_{I_{(x,|x-x'|)}}|\big)\\&&\lesssim|x-x'|\Big[j+\Big|\log \big({2^{s-j}}{|x-x'|^{-1}}\big)\Big|\Big].\nonumber
		\end{eqnarray}
		
	In a similar way, we have
		\begin{eqnarray*}
			|A_{\varphi_{\mathbb{L}}}(x)-A_{\varphi_{\mathbb{L}}}(y)|\lesssim |x-y|\Big[j+\Big|\log \big({2^{s-j}}{|x-y|^{-1}}\big)\Big|\Big];\\
			|A_{\varphi_{\mathbb{L}}}(x')-A_{\varphi_{\mathbb{L}}}(y)|\lesssim |x'-y|\Big[j+\Big|\log \big({2^{s-j}}{|x'-y|^{-1}}\big)\Big|\Big].
		\end{eqnarray*}
	{	Therefore,}
		$$	\aligned
		|R_{s,\mathbb{L};j}(x,\,y)-R_{s,\mathbb{L};j}(x',\,y)|
		&\quad\leq |\phi_s(x-y)|\Big|\frac{A_{\varphi_{\mathbb{L}}}(x)-A_{\varphi_{\mathbb{L}}}(y)}{|x-y|^{d+1}}-
		\frac{A_{\varphi_{\mathbb{L}}}(x')-A_{\varphi_{\mathbb{L}}}(y)}{|x'-y|^{d+1}}\Big|\\
		&\qquad+\frac{|A_{\varphi_{\mathbb{L}}}(x')-A_{\varphi_\mathbb{L}}(y)|}{|x'-y|^{d+1}}\big|\phi_{s}(x-y)-\phi_{s}(x'-y)\big|\\
		&\quad\lesssim \frac{|x-x'|}{|x-y|^{d+1}}\Big(j+\Big|\log \big({2^{s-j}}{|x-x'|^{-1}}\big)\Big|\Big).\endaligned
		$$
	{	This completes the proof of (ii) in Lemma \ref{lem2.x}. (iii) can be proved in the same way as (ii).}
	\end{proof}}

	Let us turn back to the contribution of $D_1.$ It follows from the method of rotation of Calder\'on-Zygmund that
	\begin{eqnarray}
		\|{\rm D}_{1}\|_{L^1(\mathbb{R}^d)}&=&\sum\limits_{i=0}^{\infty}  \sum_{j_0< j\leq Ni} \sum_{s}\sum_{\mathbb{L}\in\mathcal{S}_{s-j}}  \|T^i_{\Omega,A_{\mathbb{L}};\,s,j}b_\mathbb{L}(x)\|_{L^1(\mathbb{R}^d)}\nonumber\\
		&\lesssim &\sum\limits_{i=0}^{\infty}  \sum_{j_0< j\leq Ni} \sum_{s}\sum_{\mathbb{L}\in\mathcal{S}_{s-j}}  	j 	\int_{\mathbb{R}^{d} }\int_{2^{s-2}}^{ 2^{s+2}} \int_{\mathbb{S}^{d-1}}\frac{|\Omega_i(y')|}{|r|}|b_\mathbb{L}(x-ry')|\,dy' \,dr\,dx\nonumber\\
		&\lesssim &\sum_{i=0}^{\infty}\|\Omega_i\|_{L^1(\mathbb{S}^{d-1})}\sum_{j_0< j\leq Ni}j\sum_{s}\sum_{\mathbb{L}\in \mathcal{S}_{s-j}}\|b_\mathbb{L}\|_{L^1(\mathbb{R}^d)}\label{wmm}\\
		&\lesssim &\|\Omega\|_{L(\log L)^2(\mathbb{S}^{d-1})}\|f\|_{L^1(\mathbb{R}^d)}.\nonumber
	\end{eqnarray}
	{\bf {Contribution of ${\rm D_2}$}}.
	
	\vspace{0.3cm
	}
	
	The estimate of ${\rm D_2}$ is not short, we split the proof into three steps.
	
	{\bf {Step 1. Boil down the estimate of $D_2$ to ${\rm D}_{2}^*$}}.
	
	Let  $l_{\tau}(j)=\tau j+3$, where $0<\tau<1$ will be chosen later. Let $\omega$ be a nonnegative, radial $C^{\infty}_c(\mathbb{R}^d)$ function which is supported in $\{x\in \mathbb{R}^d:|x|\leq 1\}$ and has integral $1$. Set
	$\omega_{t}(x)=2^{-td}\omega(2^{-t}x)$.
	For $s\in\mathbb{N}$ and a cube $\mathbb{L}$, we define $R^j_{s,\mathbb{L}}$ as
	\begin{align}\label{rcon}
		R^j_{s,\mathbb{L}}(x,y)=\int_{\mathbb{R}^d}\omega_{s-l_{\tau}( j)}(x-z)\frac{1}{|z-y|^{d+1}}\phi_{s}(z-y)
		\big(A_{\varphi_\mathbb{L}}(z)-A_{\varphi_\mathbb{L}}(y)\big)dz.
	\end{align}
	It is obvious that
	${\rm supp}R_{s,\mathbb{L}}^{j}(x,y)\subset \{(x,\,y):\, 2^{s-3}\leq |x-y|\leq 2^{s+3}\}
	$.  Moreover, if $y\in \mathbb{L}$ with ${\mathbb{L}\in\mathcal{S}_{s-j}} $, then (i) of Lemma \ref{lem2.x} implies that 
	\begin{eqnarray}\label{eq:kernelsize}|R_{s,\mathbb{L}}^{j}(x,y)|\leq j2^{-sd}\chi_{\{2^{s-3}\leq |x-y|\leq 2^{s+3}\}}(x,\,y).
	\end{eqnarray}

	We define the operator $T^{i,j}_{\Omega,\mathbb{L};\,s}$ by
	$$T^{i,j}_{\Omega,\mathbb{L};\,s}h(x)=
	\int_{\mathbb{R}^d}\Omega_i(x-y)R_{s,\mathbb{L}}^{j}(x,y)h(y)dy,$$
	and let ${\rm D}_{2}^*$ be the operator as follows
	$${\rm D}_{2}^*(x)=\sum\limits_{i=0}^{\infty}  \sum_{ j> Ni}\sum_{s}\sum_{\mathbb{L}\in\mathcal{S}_{s-j}} T^{i,j}_{\Omega,\mathbb{L};\,s}b_\mathbb{L}(x).$$
	
	The following lemma indicates the intrinsically close relationship in each  subtract terms between  ${\rm D}_{2}$ and ${\rm D}_{2}^*$. Thus, the corresponding proof is transferred to verify it for each term of ${\rm D}_{2}^*$.
	\begin{lemma}\label{lem2.4}
		Let $\Omega$ be homogeneous of degree zero, $A$ be a function  on $\mathbb{R}^d$ with  derivatives of order one in ${\rm BMO}(\mathbb{R}^d)$.    For $j>j_0$ and $i\ge 0$, it holds that
		$$\|T^i_{\Omega,A_{\mathbb{L}};\,s,j}b_\mathbb{L}- T^{i,j}_{\Omega,\mathbb{L};\,s}b_\mathbb{L}\|_{L^1(\mathbb{R}^d)}\lesssim j 2^{-\tau j}  \|\Omega_i\|_{L^{1}(\mathbb{S}^{d-1})}\|b_\mathbb{L}\|_{L^1(\mathbb{R}^d)}.$$
	\end{lemma}
	\begin{proof}
		For each $y\in \mathbb{L}$ and $z\in {\rm supp}\,\omega_{s-l_{\tau}( j)}$, notice that $R_{s,\mathbb{L};j}(x,\,y)-R_{s,\mathbb{L};j{\tiny }}(x-z,\,y)=0$ if $x\in L_{j,1}\backslash {3\cdot2^{j}d\mathbb{L}}$. In fact, since $|z|\leq 2^{s-\tau j-3}$, then we have 
		$2^{s+1}<|x-y|<3\cdot 2^s$ and $2^s<|x-y-z|<2^{s+2}$.
		
		By Lemma \ref{lem2.x}, we have
		$${\big|R_{s,\mathbb{L};j}(x,\,y)-R_{s,\mathbb{L};j}(x-z,\,y)\big|\lesssim \frac{|z|}{2^{s(d+1)}}\big[j+\log \big(\frac{2^{s-j}}{|z|}\big)\big]\chi_{\{2^{s-2}\leq |x-y|\leq 2^{s+2}\}}(x,y).}
		$$
		Therefore		
		\begin{align*}
			&\|T^i_{\Omega,A_{\mathbb{L}};\,s,j}b_\mathbb{L}-T^{i,j}_{\Omega,{\mathbb{L}};\,s}b_\mathbb{L}\|_{L^1(\mathbb{R}^d)}\\
			&\leq  \int_{\mathbb{R}^{d}}\int_{\mathbb{R}^{d}}|\Omega_i(x-y)|\Big|\int_{\mathbb{R}^d}\omega_{s-l_{\tau}( j)}(z)\Big(R_{s,\mathbb{L};j}(x,\,y)-R_{s,\mathbb{L};j}(x-z,\,y)\Big)dz\Big||b_\mathbb{L}(y)|dydx\\
			&\lesssim 2^{-sd-\tau j}\int_{\mathbb{R}^{d}}\int_{\mathbb{R}^{d}}|\Omega_i(y)|\int_{|z|\leq 1}|z|\bigg(j+\big|\log \frac{2^{\tau j-j+3}}{|z|}\big|\bigg)dz\chi_{\{2^{s-2}\leq |y|\leq 2^{s+2}\}}(y)|b_\mathbb{L}(x-y)|dydx\\
			& {\lesssim j 2^{-\tau } \|\Omega_i\|_{L^{1}(\mathbb{S}^{d-1})}\|b_{\mathbb{L}}\|_{L^1(\mathbb{R}^d)}}.
		\end{align*}

		This leads to the desired conclusion of Lemma \ref{lem2.4}.
	\end{proof}
	
	With Lemma \ref{lem2.4} in hand, we only need to estimate ${\rm D}_{2}^*$. This is the content of the second step.
	
	\medskip
	{\bf {Step 2. Estimate for each term of ${\rm D}_{2}^*$}}.
	
	Define $P_tf(x)=\omega _t*f(x)$.	Now  we split $$T^{i,j}_{\Omega,{\mathbb{L}};\,s}=P_{s-j\kappa}T^{i,j}_{\Omega,{\mathbb{L}};\,s}+(I-P_{s-j\kappa})T^{i,j}_{\Omega,{\mathbb{L}};\,s},$$ where $\kappa \in (0,1)$ will be chosen later. In the following, we will estimate this two terms one by one.
	We have the following norm inequality for  $P_{s-j\kappa}T^{i,j}_{\Omega,{\mathbb{L}};\,s}$.
	\begin{lemma}\label{ljj}Let $\Omega$ be homogeneous of degree zero, $A$ be a function  in $\mathbb{R}^d$ with  derivatives of order one in ${\rm BMO}(\mathbb{R}^d)${\color{red},} $b_\mathbb{L}$ satisfies the vanishing moment with $\ell(\mathbb{L})=2^{s-j}$.  For each $j\in\mathbb{N}$ with $j>j_0$,  we have 
		$$\big \| P_{s-j\kappa} T^{i,j}_{\Omega,{\mathbb{L}};\,s}b_\mathbb{L}\big\|_{L^1(\mathbb{R}^d)}\lesssim j\big(2^{-(1-\kappa)j} +2^{-(1-\tau)j} \big) \|\Omega_i\|_{L^{\infty}(\mathbb{S}^{d-1})}\|b_\mathbb{L}\|_{L^1(\mathbb{R}^d)}.$$
	\end{lemma}
	
	Before proving   Lemma \ref{ljj}, we need the following lemma for $R_{s,\mathbb{L}}^{j}$.
	\begin{lemma}\label{lemmaR}
		Let $R^j_{s,\mathbb{L}}$ be defined as \eqref{rcon}, $\theta\in \mathbb{S}^{d-1}$,  $y,\,y'\in \mathbb{L}$ with $\ell(\mathbb{L})=2^{s-j}$. Then
		\begin{eqnarray*}&&\int_{\mathbb{L}}\int_{\mathbb{L}}|R_{s,\mathbb{L}}^{j}(y+r\theta,\,y)-R_{s,\mathbb{L}}^{j}(y'+r\theta,\,y')||b_{\mathbb{L}}(y)|dydy'\lesssim j 2^{-sd}2^{\tau j}2^{-j}|\mathbb{L}|\int_{\mathbb{L}}|b_\mathbb{L}(y)|dy.
		\end{eqnarray*}
	\end{lemma}
	\begin{proof}
		By the triangle inequality, the mean value theorem and the support condition of $\phi$,	we get
				\begin{eqnarray*}&&|R^j_{s,\mathbb{L}}(y'+r\theta,\,y)-R^j_{s,\mathbb{L}}(y'+r\theta,\,y')|\\
			&&\quad \lesssim \int_{\mathbb{R}^d}|\omega_{s-l_{\tau}( j)}(y'+r\theta-z)||\phi_s(z-y')|
			\frac{|A_{\varphi_{\mathbb{L}}}(y)-A_{\varphi_{\mathbb{L}}}(y')|}{|z-y'|^{d+1}}dz\\
			&&\quad+\int_{\mathbb{R}^d}|\omega_{s-l_{\tau}( j)}(y'+r\theta-z)|
			\frac{|A_{\varphi_{\mathbb{L}}}(z)-A_{\varphi_{\mathbb{L}}}(y)|}{|z-y|^{d+1}}|\phi_{s}(z-y)-\phi_{s}(z-y')|dz\\
			&&\quad+\int_{\mathbb{R}^d}|\omega_{s-l_{\tau}( j)}(y'+r\theta-z)||\phi_s(z-y')|
			\frac{|A_{\varphi_\mathbb{L}}(z)-A_{\varphi_{\mathbb{L}}}(y)||y-y'|}{|z-y|^{d+2}}dz\\
			&&\quad	:= I+II+III.
		\end{eqnarray*}
		If $r\notin [2^{s-4}, 2^{s+4}]$, by the support of $R^j_{s,\mathbb{L}}$, it gives that $|R^j_{s,\mathbb{L}}(y'+r\theta,\,y)-R^j_{s,\mathbb{L}}(y'+r\theta,\,y')|=0$. 
		
		\smallskip
		For $y,\,y'\in \mathbb{L}$, \eqref{wtw}  gives us that
		$$|A_{\varphi_{\mathbb{L}}}(y')-A_{\varphi_{\mathbb{L}}}(y)|\lesssim |y-y'|\Big[j+|\log \big(\frac{2^{s-j}}{|y-y'|}\big)|\Big].
		$$
		For $I$, since $|z-y'|\geq 2^{s-2}$, $y,\,y'\in \mathbb{L}$, then \eqref{wtw} gives us that
		\begin{eqnarray*}
			I\lesssim & |y-y'|\Big[j+|\log \big(\frac{2^{s-j}}{|y-y'|}\big)|\Big] 2^{-s(d+1)}.
		\end{eqnarray*}
	Consider now the other two terms. If $y,\,y'\in \mathbb{L}$ and $|z-y'|\leq 2^{s}$, \rm{(i)} of  Lemma \ref{lem2.x} gives us that
		$$|A_{\varphi_{\mathbb{L}}}(y)|\lesssim j2^s,\,\,\,|A_{\varphi_{\mathbb{L}}}(z)|\lesssim j2^s.$$
		On the other hand,  for $j>j_0$, when $y,\,y'\in \mathbb{L}$ and $|z-y'|\geq 2^{s-2}$, it holds that
		$$|z-y|\geq |z-y'|-|y-y'|\geq 2^{s-2}-\sqrt{d}2^{s-j}>2^{s-2}-\sqrt{d}2^{s-\log_2(100d/2)} >2^{s-3}.$$
		{Therefore}
		\begin{align*}
			II+III&=\int_{\mathbb{R}^d}|\omega_{s-l_{\tau}( j)}(y'+r\theta-z)|
			\frac{|A_{\varphi_{\mathbb{L}}}(z)-A_{\varphi_{\mathbb{L}}}(y)|}{|z-y|^{d+1}}|\phi_{s}(z-y)-\phi_{s}(z-y')|dz\\
			&\quad+\int_{\mathbb{R}^d}|\omega_{s-l_{\tau}( j)}(y'+r\theta-z) \phi_s(z-y')|
			\frac{|A_{\varphi_{\mathbb{L}}}(z)-A_{\varphi_{\mathbb{L}}}(y)||y-y'|}{|z-y|^{d+2}}dz\\
			&\lesssim j2^{-s(d+1)}|y-y'|,
		\end{align*}
		which implies that
		\begin{align}
			|R^j_{s,\mathbb{L}}(y'+r\theta,\,y)-R^j_{s,\mathbb{L}}(y'+r\theta,\,y')|\lesssim \frac{ j|y-y'|}{2^{s(d+1)}}\Big[1+|\log \big(\frac{2^{s-j}}{|y-y'|}\big)|\Big].\label{cz1}
		\end{align}
		Similarly, we have
		\begin{align}&|R^j_{s,\mathbb{L}}(y+r\theta,\,y)-R^j_{s,\mathbb{L}}(y'+r\theta,\,y)|\label{cz2}\\
			&\, \leq \int_{\mathbb{R}^d}|\omega_{s-l_{\tau}( j)}(y+r\theta-z)-\omega_{s-l_{\tau}( j)}(y'+r\theta-z)||\phi_{s}(z-y)|\frac{|A_{\varphi_{\mathbb{L}}}(z)-
				A_{\varphi_{\mathbb{L}}}(y)|}{|z-y|^{d+1}}dz\nonumber\\
			&\, \leq j2^{-sd} 2^{-s+l_{\tau}( j)}|y-y'|\int_{\mathbb{R}^d}|\nabla\omega_{s-l_{\tau}( j)}(z)|dz \nonumber\\
			&\,\lesssim j |y-y'|2^{-s+l_{\tau}( j)}2^{-sd}.\nonumber
		\end{align}
		Notice that 
		\begin{align*}
			&	\int_{\mathbb{L}}\int_{\mathbb{L}} |y-y' |\Big[1+|\log \big(\frac{2^{s-j}}{|y-y'|}\big)|\Big]dy' |b_{\mathbb{L}}(y) |dy\leq 2^{s-j}|\mathbb{L}|\int_{\mathbb{L}}|b_\mathbb{L}(y)|dy.
		\end{align*}
		Combining \eqref{cz1} with \eqref{cz2}, it gives that
		
		\begin{eqnarray*}&&\int_{\mathbb{L}}\int_{\mathbb{L}}|R_{s,\mathbb{L}}^{j}(y+r\theta,\,y)-R_{s,\mathbb{L}}^j(y'+r\theta,\,y') | |b_{\mathbb{L}}(y) |dydy'\\
			&&\quad\lesssim j 2^{-s(d+1)}2^{l_{\tau}( j)}\int_{\mathbb{L}}\int_{\mathbb{L}}|y-y'|dy'|b_{\mathbb{L}}(y) |dy\\
			&&\qquad+j 2^{-s(d+1)}\int_{\mathbb{L}}\int_{\mathbb{L}} |y-y' |\Big[1+|\log \big(\frac{2^{s-j}}{|y-y'|}\big)|\Big]dy' |b_{\mathbb{L}}(y) |dy\\
			&&\quad \lesssim j 2^{-sd}2^{l_{\tau}( j)}2^{-j}|\mathbb{L}|\int_{\mathbb{L}}|b_\mathbb{L}(y)|dy.
		\end{eqnarray*}
		{This finishes the proof of Lemma \ref{lemmaR}.}
	\end{proof}
	With Lemma \ref{lemmaR}, we are ready to prove Lemma \ref{ljj} now.
	\begin{proof}[Proof of Lemma  \ref{ljj}]
		Write
		$$	P_{s-j\kappa} T^{i,j}_{\Omega,{\mathbb{L}};\,s}b_\mathbb{L} (x)= \int_{\mathbb{R}^d}\Big(\int_{\mathbb{R}^d} \omega_{s-j\kappa}(x-z)\Omega_i(z-y) R_{s,\mathbb{L}}^{j}(z,y)dz \Big)b_\mathbb{L}(y)dy.$$
		Let $z-y=r\theta$. By Fubini's theorem, $	P_{s-j\kappa} T^{i,j}_{\Omega,{\mathbb{L}};\,s}b_\mathbb{L} (x)$ can be written as
		$$\int_{\mathbb{S}^{d-1}} \int_{\mathbb{R}^d}\int_0^{\infty}\Omega_i(\theta)\omega_{s-j\kappa}(x-y-r\theta) R_{s,\mathbb{L}}^{j}(y+r\theta ,y)r^{d-1} b_\mathbb{L}(y)drdyd\sigma_{\theta}.$$
		Let $y' \in  \mathbb{L}$. By the vanishing moment of $b_\mathbb{L}$, we have
		\begin{align*}
			|P_{s-j\kappa} T^{i,j}_{\Omega_i,{\mathbb{L}};\,s}b_L (x)|&\leq \inf\limits_{y'\in \mathbb{L}} \int_{\mathbb{S}^{d-1}}| \Omega_i(\theta)|\Big|\int_{\mathbb{R}^d}\int_0^{\infty} \Big(\omega_{s-j\kappa}(x-y-r\theta) R_{s,\mathbb{L}}^{j}(y+r\theta ,y)\\&\quad-\omega_{s-j\kappa}(x-y'-r\theta) R_{s,\mathbb{L}}^{j}(y'+r\theta ,y')\Big) r^{d-1}dr b_\mathbb{L}(y)dy\Big|d\sigma_{\theta}\\
			&\leq  \int_{\mathbb{S}^{d-1}}| \Omega_i(\theta)|\frac{1}{|\mathbb{L}|}\int_{\mathbb{L}}\Big|\int_{\mathbb{R}^d}\int_0^{\infty} \Big(\omega_{s-j\kappa}(x-y-r\theta) R_{s,\mathbb{L}}^{j}(y+r\theta ,y)\\&\quad-\omega_{s-j\kappa}(x-y'-r\theta) R_{s,\mathbb{L}}^{j}(y'+r\theta ,y')\Big) r^{d-1}dr b_\mathbb{L}(y)dy\Big|d{y'}d\sigma_{\theta}   \\
			&\lesssim I+II,
		\end{align*}
		where
		\begin{align*}
			I=&\frac{1}{|\mathbb{L}|}\int_{\mathbb{S}^{d-1}}  |\Omega_i(\theta)|\int_{\mathbb{L}} \Big|\int_{\mathbb{R}^d}\int_0^{\infty}\Big(\omega_{s-j\kappa}(x-y-r\theta) -\omega_{s-j\kappa}(x-y'-r\theta)   \Big)\\ &\times R_{s,\mathbb{L}}^{j}(y+r\theta ,y)r^{d-1}dr b_\mathbb{L}(y)dy \Big | d{y'}d\sigma_{\theta},
		\end{align*}
		and
		\begin{align*}
			II=&\frac{1}{|\mathbb{L}|}\int_{\mathbb{S}^{d-1}}  |\Omega_i(\theta)|\int_{\mathbb{L}} \Big | \int_{\mathbb{R}^d}\int_0^{\infty}\omega_{s-j\kappa}(x-y'-r\theta)  \big (R_{s,\mathbb{L}}^{j}(y+r\theta ,y) \\ &-R_{s,\mathbb{L}}^{j}(y'+r\theta ,y')\big)r^{d-1}dr b_\mathbb{L}(y)dy\Big|d{y'}\,d\sigma_{\theta}.
		\end{align*}
		 {Note that $|y-y'|\lesssim 2^{s-j}$, when $y$, $y'\in \mathbb{L}$} . By \eqref{eq:kernelsize} and the mean value formula, it follows that
		\begin{align*}
			\|I\|_{L^1(\mathbb{R}^d)}\lesssim &j\int_{\mathbb{S}^{d-1}} |\Omega_i(\theta)|\int_{\mathbb{R}^d}\int_{2^{s-3}}^{2^{s+3}} 2^{-s+j\kappa} \|\triangledown\omega\|_{L^1(\mathbb{R}^d)} 2^{s-j}  2^{-sd}r^{d-1}dr |b_\mathbb{L}(y)|dyd\sigma({\theta})\\
			& \lesssim j 2^{-(1-\kappa)j} \|\Omega_i\|_{L^{\infty}(\mathbb{S}^{d-1})} \|b_{\mathbb{L}}\|_{L^1(\mathbb{R}^d)}.
		\end{align*}
		By Lemma \ref{lemmaR} and the Fubini's theorem one can get
		
		\begin{align*}
			\|II\|_{L^1(\mathbb{R}^d)}&\lesssim\int_{\mathbb{S}^{d-1}} \int_{2^{s-3}}^{2^{s+3}} |\Omega_i(\theta)|\frac{1}{|\mathbb{L}|}\int_{\mathbb{L}}  \int_{\mathbb{L}}\|\omega_{s-j\kappa}(\cdot-y'-r\theta)\|_{ L^1(\mathbb{R}^d) }  \times |\big (R_{s,\mathbb{L}}^{j}(y+r\theta ,y) \\ &\quad -R_{s,\mathbb{L}}^{j}(y'+r\theta ,y')\big)| |b_\mathbb{L}(y)|dyd{y'}\,r^{d-1}dr d\sigma_{\theta}\\
			&\lesssim j2^{-{(1-\tau)}j} \|\Omega_i\|_{L^{\infty}(\mathbb{S}^{d-1})} \|b_{\mathbb{L}}\|_{L^1(\mathbb{R}^d)}.\end{align*}
		This finishes the proof of Lemma \ref{ljj}.
	\end{proof}
	
	To estimate the term $(I-P_{s-j\kappa})T^{i,j}_{\Omega,{\mathbb{L}};\,s}$, we  introduce a partition of unity on the unit surface $\mathbb{S}^{d-1}$.
	For $j>j_0$, let  $\mathfrak{E}^j=\{e_{\nu}^j\}$ be a collection of unit vectors on $\mathbb{S}^{d-1}$ such that 	\vspace{0.2cm}
	
	\begin{itemize}
		\item[\rm (a)] for $\nu\not =\nu'$, $|e_{\nu}^j-e_{\nu'}^j|>2^{-j\gamma-4}$;
		\item[\rm (b)] for each $\theta\in \mathbb{S}^{d-1}$, there exists an $e^j_{\nu}$ satisfying that $|e^j_{\nu}-\theta|\leq 2^{-j\gamma-4},$ 	where $\gamma\in (0,\,1)$ is a constant. 
	\end{itemize}

		\vspace{0.2cm}
		The set $\mathfrak{E}^j$ can be constructed as in \cite{se}.  Observe that ${\rm card}(\mathfrak{E}^j)\lesssim 2^{j\gamma(d-1)}$.
	
	Below, we will construct an associated partition of unity on the unit surface $\mathbb{S}^{d-1}$.
	Let $\zeta$ be a smooth, nonnegative, radial function with $\zeta(u)\equiv 1$ when $|u|\leq 1/2$ and ${\rm supp}\, \zeta\subset \{|x|\leq 1\}$. Set
	$$\widetilde{\Gamma}_\nu^j(\xi)=\zeta\Big(2^{j\gamma}\big(\frac{\xi}{|\xi|}-e^j_{\nu}\big)\Big),  \ \hbox {and\ \  } \Gamma_{\nu}^j(\xi)=\widetilde{\Gamma}^j_{\nu}(\xi)\Big(\sum_{e^j_{\nu}\in\mathfrak{E}^j}\widetilde{\Gamma}^j_{\nu}(\xi)\Big)^{-1}.$$ It is easy to verify that $\Gamma^j_{\nu}$ is homogeneous of degree zero, and for all $j$ and $\xi\in \mathbb{S}^{d-1}$,
	$\sum_{\nu}\Gamma^j_{\nu}(\xi)=1.
	$
	Let $\psi\in C^{\infty}_c(\mathbb{R})$ such that $0\leq \psi\leq 1$,  ${\rm supp}\, \psi\subset [-4,\,4]$ and $\psi(t)\equiv 1$ when $t\in [-2,\,2]$. Define the multiplier operator $G_{\nu}^j$ by
	$$\widehat{G_{\nu}^jf}(\xi)=\psi\big(2^{j\gamma}\langle \xi/|\xi|, e_{\nu}^j\rangle\big)\widehat{f}(\xi).
	$$
	Denote the operator $T^{i,j}_{\Omega,{\mathbb{L}};\,s,\nu}$ by
	\begin{align}\label{deftlq}
		T^{i,j}_{\Omega,{\mathbb{L}};\,s,\nu}h(x)=
		\int_{\mathbb{R}^d}\Omega_i(x-y)\Gamma^j_{\nu}(x-y)R_{s,\mathbb{L}}^{j}(x,y)h(y)dy.
	\end{align}
	It is obvious that
	$T^{i,j}_{\Omega,\mathbb{L};\,s}h(x)=\sum_{\nu}T^{i,j}_{\Omega,{\mathbb{L}};\,s,\nu}h(x).
	$
	For each fixed $i,\,s,\,j,\,\mathbb{L}$ and $\nu$,  $(I-P_{s-j\kappa})T^{i,j}_{\Omega,{\mathbb{L}};\,s,\nu}$ can be decomposed in the following way
	$$(I-P_{s-j\kappa})T^{i,j}_{\Omega,{\mathbb{L}};\,s,\nu}=G_{\nu}^j(I-P_{s-j\kappa})T^{i,j}_{\Omega,{\mathbb{L}};\,s,\nu}+(1-G_{\nu}^j)(I-P_{s-j\kappa})T^{i,j}_{\Omega,{\mathbb{L}};\,s,\nu}.
	$$
	
	{\bf {Estimate for the term $G_{\nu}^j(I-P_{s-j\kappa})T^{i,j}_{\Omega,{\mathbb{L}};\,s,\nu}$}}.
	
	For the term $G_{\nu}^j(I-P_{s-j\kappa})T^{i,j}_{\Omega,{\mathbb{L}};\,s,\nu}$, we have  the following lemma.
	\begin{lemma}\label{lem2.5}
		Let $\Omega$ be homogeneous of degree zero,  $A$ be a function  in $\mathbb{R}^d$ with  derivatives of order one in ${\rm BMO}(\mathbb{R}^d)$. For each $j\in\mathbb{N}$ with $j> j_0$, we have that,
		\begin{align*}
			&\Big\|\sum_{\nu} \sum_s\sum_{\mathbb{L}\in\mathcal{S}_{s-j}}G_{\nu}^j(I-P_{s-j\kappa})T^{i,j}_{\Omega,{\mathbb{L}};\,s,\nu}b_\mathbb{L}\Big\|_{L^2(\mathbb{R}^d)}^2	\lesssim j^22^{-j\gamma}\|\Omega_i\|_{L^{\infty}(\mathbb{S}^{d-1})}^2\Big\|\sum_s\sum_{\mathbb{L}\in\mathcal{S}_{s-j}}b_\mathbb{L}\Big\|_{L^1(\mathbb{R}^d)}.
		\end{align*}
	\end{lemma}
	\begin{proof}
		The proof is similar to the proof   of Lemma 2.3 in \cite{dinglai}. For the sake of self-ciontained, we present the proof here. Observe that
		$$\sup_{\xi\not=0}\sum_{\nu}|\psi(2^{j\gamma}\langle e^j_{\nu},\xi/|\xi|\rangle)|^2\lesssim 2^{j\gamma(d-2)}.
		$$
		This, together with Plancherel's theorem and  Cauchy-Schwartz inequality, leads to that
		\begin{eqnarray*}
			&&\Big\|\sum_{\nu} \sum_s\sum_{\mathbb{L}\in\mathcal{S}_{s-j}} G_{\nu}^j(I-P_{s-j\kappa})T^{i,j}_{\Omega,{\mathbb{L}};\,s,\nu}b_\mathbb{L}\Big\|_{L^2(\mathbb{R}^d)}^2\\
			&&\quad=\Big\|
			\sum_{\nu}\psi\Big(2^{j\gamma}\langle e^j_{\nu},\,\xi/|\xi|\rangle\Big)\mathcal{F}\Big(\sum_s\sum_{\mathbb{L}\in\mathcal{S}_{s-j}}(I-P_{s-j\kappa})T^{i,j}_{\Omega,{\mathbb{L}};\,s,\nu}b_\mathbb{L}\Big)(\xi)\Big\|^2_{L^2(\mathbb{R}^d)}\\
			&&\quad\lesssim2^{j\gamma(d-2)}\sum_{\nu}\Big\|\sum_s\sum_{\mathbb{L}\in\mathcal{S}_{s-j}}(I-P_{s-j\kappa})T^{i,j}_{\Omega,{\mathbb{L}};\,s,\nu}b_\mathbb{L}\Big\|_{L^2(\mathbb{R}^d)}^2.
		\end{eqnarray*}
		Applying (\ref{eq:kernelsize}), we see that for each fixed $s$, $\mathbb{L}$, and $x\in\mathbb{R}^d$,
		\begin{align}
			&\big|	(I-P_{s-j\kappa})T^{i,j}_{\Omega,{\mathbb{L}};\,s,\nu}b_\mathbb{L}(x)\big|\label{hjk} \\
			%	&\lesssim | T^{j}_{s,\,l(s),\nu_k}(B_{s-j,l(s)})(x)|+| P_{s-j\kappa} T^{j}_{s,\,l(s),\nu_k}(B_{s-j,l(s)})(x)|\\
			&\lesssim  \int_{\mathbb{R}^d}|\Omega_i(x-y)||\Gamma^j_{\nu}(x-y)||R_{s,\mathbb{L}}^{j}(x,y)||b_\mathbb{L}(y)|dy\nonumber\\
			&\quad+ \int_{\mathbb{R}^{d}}\int_{\mathbb{R}^{d}} |\Omega_i(z-y)||\omega_{s-j\kappa}(x-z)| |\Gamma^j_{\nu}(z-y)||R_{s,\mathbb{L}}^{j}(z,y)|dz|b_\mathbb{L}(y)|dy\nonumber\\
			&\lesssim j \|\Omega_i\|_{L^{\infty}(\mathbb{S}^{d-1})}H_{s,\nu}^j*|b_\mathbb{L}|(x),\nonumber
		\end{align}
		where $ H_{s,\nu}^j(x)=2^{-sd}\chi_{\mathcal{R}_{s\nu}^j}(x)$, and
		$\mathcal{R}_{s\nu}^j=\{x\in\mathbb{R}^d:\,|\langle x,e_{\nu}^j \rangle|\leq 2^{s+3},\,|x-\langle x,e_{\nu}^j\rangle e_{\nu}^j|\leq 2^{s+3-j\gamma}\}.
		$
		This means that $\mathcal{R}_{s\nu}^j$ is contained
		in a box  having one long side of length $\lesssim 2^s$
		and $(d-1)$ short sides of
		length $\lesssim 2^{s-j\gamma}$.
		Therefore, we have
		\begin{eqnarray*}
			&&\Big\|\sum_s\sum_{\mathbb{L}\in\mathcal{S}_{s-j}}(I-P_{s-j\kappa})T^{i,j}_{\Omega,{\mathbb{L}};\,s,\nu}b_\mathbb{L}\Big\|_{L^2(\mathbb{R}^d)}^2\\
			&&\lesssim j^2\|\Omega_i\|^2_{L^{\infty}(\mathbb{S}^{d-1})}\sum_s\sum_{\mathbb{L}\in\mathcal{S}_{s-j}}\sum_{I\in\mathcal{S}_{s-j}}
			\int_{\mathbb{R}^d}\big(H_{s,\nu}^j*H_{s,\nu}^j*|b_I|\big)(x)|b_\mathbb{L}(x)|dx\\
			&&\quad+2j^2\|\Omega_i\|^2_{L^{\infty}(\mathbb{S}^{d-1})}\sum_s\sum_{\mathbb{L}\in\mathcal{S}_{s-j}}\sum_{i<s}\sum_{I\in \mathcal{S}_{i-j}}\int_{\mathbb{R}^d}\big(H_{s,\nu}^j*H_{i,\nu}^j*
			|b_I|\big)(x)|b_\mathbb{L}(x)|dx.\\
		\end{eqnarray*}
		Let $\widetilde {\mathcal{R}}_{s\nu}^j={\mathcal{R}}_{s\nu}^j+{\mathcal{R}}_{s\nu}^j$.
		As in the proof of Lemma 2.3 in \cite{dinglai}, for each fixed $\mathbb{L}\in \mathcal{S}_{s-j}$, $x\in \mathbb{L}$, $\nu$ and $s$, we obtain
	$$\sum_{i\leq s}\sum_{I\in \mathcal{S}_{i-j}}H_{s,\nu}^j*H_{i,\nu}^j*|b_I|(x)\lesssim{ 2^{-j\gamma(d-1)}2^{-sd}\sum_{i\leq s}\sum_{I\in \mathcal{S}_{i-j}}\int_{x+\widetilde{\mathcal{R}}_{j\nu}^s}|b_I(y)|dy} \lesssim 2^{-2j\gamma(d-1)},
	$$
		where we have used the fact that $\int_{\mathbb{R}^d} |b_I(y)|dy\lesssim |I|$ and the cubes $I\in \mathcal {S}$   are pairwise disjoint.
		
		This, in turn, implies further that
		$$\Big\|\sum_s\sum_{\mathbb{L}\in\mathcal{S}_{s-j}}(I-P_{s-j\kappa})T^{i,j}_{\Omega,{\mathbb{L}};\,s,\nu}b_\mathbb{L}\Big\|_{L^2(\mathbb{R}^d)}^2\lesssim j^2\|\Omega_i\|_{L^{\infty}(\mathbb{S}^{d-1})}^22^{-2j\gamma(d-1)}\Big\|\sum_s\sum_{\mathbb{L}\in\mathcal{S}_{s-j}} b_\mathbb{L}\Big\|_{L^1(\mathbb{R}^d)},
		$$
	which finishes the proof of Lemma \ref{lem2.5}.
	\end{proof}
	
	{\bf {Estimate for the term $(1-G_{\nu}^j)(I-P_{s-j\kappa})T^{i,j}_{\Omega,{\mathbb{L}};\,s,\nu}.$}}
	
	\medskip
	We need to present a lemma for $(1-G_{\nu}^j)(I-P_{s-j\kappa})T^{i,j}_{\Omega,{\mathbb{L}};\,s,\nu}$.
	\begin{lemma}\label{lem2.6wzd}Let $\Omega$ be homogeneous of degree zero,  $A$ be a function  in $\mathbb{R}^d$ with  derivatives of order one in ${\rm BMO}(\mathbb{R}^d)$.     For each $j\in\mathbb{N}$ with $j> j_0$, $\ell(\mathbb{L})=2^{s-j}$, some  $s_0>0$, we have that
		$$\sum_{\nu}  \| (I-{G_{\nu}^j)}(I-P_{s-j\kappa} )T^{i,j}_{\Omega,{\mathbb{L}};\,s,\nu}b_\mathbb{L}\|_{L^1(\mathbb{R}^d)}\lesssim j 2^{-s_0 j}  \|\Omega_i\|_{L^{\infty}(\mathbb{S}^{d-1})} \|b_\mathbb{L}\|_{L^1(\mathbb{R}^d)}.$$
		
	\end{lemma}
	Next we give the estimate of ${\rm D}_{2}^*$ and  	postpone the proof of  Lemma \ref{lem2.6wzd} later.
	
	Let $\varepsilon=\min\{  (1-\kappa), (1-\tau), s_0,\gamma \}$.  With Lemma \ref{ljj}, Lemma \ref{lem2.5} and Lemma \ref{lem2.6wzd}, we have
	\begin{eqnarray}
		&&\Big|\Big\{x\in\mathbb{R}^d:\, |{\rm D}_{2}^*|>1/16\Big\}\Big|\label{ddd}\lesssim  \sum\limits_{i=0}^{\infty}  \sum_{ j> Ni} j^22^{-j\varepsilon}\|\Omega_i\|^2_{L^{\infty}(\mathbb{S}^{d-1})}\Big\|\sum_s\sum_{\mathbb{L}\in\mathcal{S}_{s-j}}b_\mathbb{L}\Big\|_{L^1(\mathbb{R}^d)}\lesssim  \|f\|_{L^1(\mathbb{R}^d)}.
	\end{eqnarray}
	The proof of Lemma \ref{lem2.6wzd}
	is similar to the proof   of Lemma 2.4 in \cite{dinglai}.
	For the completeness of this paper, we give the proof for the remaining term $(1-G_{\nu}^j)(I-P_{s-j\kappa})T^{i,j}_{\Omega,{\mathbb{L}};\,s,\nu}$ here.
	Let's introduce the Littlewood-Paley decomposition first. Let $\alpha$ be a radial $C^{\infty}$ function such that $\alpha(\xi)=1$ for $|\xi| \leq 1$, $\alpha(\xi)=0$ for $|\xi| \geq 2$ and $0\leq \alpha(\xi)\leq 1$  for all $\xi \in \mathbb{R}^d$. Define $\beta _k(\xi)=\alpha(2^k\xi)-\alpha(2^{k+1}\xi)$.
	Choose  $\tilde{\beta}$ be a radial $C^{\infty}$ function such that $\tilde{\beta}(\xi)=1$ for $1/2\leq |\xi| \leq 2$, $\rm{supp} \,\tilde{\beta} \in [1/4,4] $ and
	$0\leq \tilde{\beta}\leq 1$ for all $\xi\in \mathbb{R}^d$. Set $ \tilde{\beta}_k(\xi)=\tilde{\beta}(2^k\xi) $, then it is easy to see $\beta_k=\tilde{\beta}_k{\beta}_k$. Define the convolution operators $\Lambda _k$ and $\tilde{\Lambda} _k$ with Fourier multipliers $\beta _k$ and $\tilde{\beta _k}$, respectively.
	\begin{align*}
		\widehat{\Lambda _k f}(\xi)= \beta _k(\xi)\widehat{f}(\xi),\qquad \widehat{ \tilde{ \Lambda} _k f}(\xi)=\tilde{ \beta} _k(\xi)\widehat{f}(\xi).
	\end{align*}
	It is easy to have $\Lambda _k=\tilde{ \Lambda} _k \Lambda _k$.
	
	{\bf  Proof of  Lemma \ref{lem2.6wzd}.}
	We first write $(I-{G_{\nu}^j)}T^{i,j}_{\Omega,{\mathbb{L}};\,s,\nu}= \sum\limits_{k}(I-{G_{\nu}^j)} \Lambda _k T^{i,j}_{\Omega,{\mathbb{L}};\,s,\nu}$. Then
	\begin{align*}
		\|  (I-{G_{\nu}^j)}(I-P_{s-j\kappa} )\Lambda _kT^{i,j}_{\Omega,{\mathbb{L}};\,s,\nu}b_\mathbb{L}\|_{L^1(\mathbb{R}^d)}
		&\leq     \| (I-P_{s-j\kappa} )\tilde{\Lambda} _k  (I-{G_{\nu}^j)}\Lambda _kT^{i,j}_{\Omega,{\mathbb{L}};\,s,\nu}b_\mathbb{L}\|_{L^1(\mathbb{R}^d)}\\
		&\leq  \| (I-P_{s-j\kappa} )\tilde{\Lambda} _k\|_{L^1(\mathbb{R}^d)\rightarrow L^1(\mathbb{R}^d)} \| (I-{G_{\nu}^j)}\Lambda _kT^{i,j}_{\Omega,{\mathbb{L}};\,s,\nu}b_\mathbb{L}\|_{L^1(\mathbb{R}^d)}.
	\end{align*}
	We can write
	\begin{align*}
		(I-{G_{\nu}^j)}\Lambda _kT^{i,j}_{\Omega,{\mathbb{L}};\,s,\nu}b_\mathbb{L}(x)&=\int_\mathbb{L}  (I-{G_{\nu}^j)}\Lambda _k   \Omega_i(x-y)\Gamma^j_{\nu}(x-y)R_{s,\mathbb{L}}^{j}(x,y) b_\mathbb{L}(y)dy:= \int_\mathbb{L} M_k(x,y)b_\mathbb{L}(y)dy,
	\end{align*}
	where $M_k$ is the kernel of the operator $ (I-{G_{\nu}^j)}\Lambda _kT^{i,j}_{\Omega,{\mathbb{L}};\,s,\nu} $.
	Then
	$$\| (I-{G_{\nu}^j)}\Lambda _kT^{i,j}_{\Omega,{\mathbb{L}};\,s,\nu}b_\mathbb{L}\|_{L^1(\mathbb{R}^d)}\leq \int_\mathbb{L} \|M_k(\cdot,y)\|_{L^1(\mathbb{R}^d)} b_\mathbb{L}(y)dy.$$
	
	Applying the method of Lemma 4.2 in \cite{dinglai}, there esists $M>0$ such that
	$$\|M_k(\cdot,y)\|_{L^1(\mathbb{R}^d)}\lesssim j 2^{\tau j-j\gamma(d-1)-s+k+j\gamma(1+2M)} .$$
	Hence, note that $ \| (I-P_{s-j\kappa} )\tilde{\Lambda} _k\|_{L^1(\mathbb{R}^d)\rightarrow L^1(\mathbb{R}^d)} \leq \|\mathcal{F}^{-1}( \tilde{\beta}_k)-\omega_{s-j\kappa }*\mathcal{F}^{-1}( \tilde{\beta}_k) \|_{L^1(\mathbb{R}^d)}\lesssim 1$, we have
	\begin{align}
		\| (I-{G_{\nu}^j)}(I-P_{s-j\kappa} )\Lambda _kT^{i,j}_{\Omega,{\mathbb{L}};\,s,\nu}b_\mathbb{L}\|_{L^1(\mathbb{R}^d)}
		\lesssim j2^{\tau j-j\gamma(d-1)-s+k+j\gamma(1+2M)} \|b_\mathbb{L}\|_{L^1(\mathbb{R}^d)}.\label{last2}
	\end{align}
	On the other hand, we can write
	\begin{align*}
		&	\|(I-{G_{\nu}^j)}(I-P_{s-j\kappa} )\Lambda _kT^{i,j}_{\Omega,{\mathbb{L}};\,s,\nu}b_\mathbb{L}\|_{L^1(\mathbb{R}^d)}\\
		&	\leq  \| (I-P_{s-j\kappa} )\tilde{\Lambda} _k\|_{L^1(\mathbb{R}^d)\rightarrow L^1(\mathbb{R}^d)} \| (I-{G_{\nu}^j)}\Lambda _k\|_{L^1(\mathbb{R}^d)\rightarrow L^1(\mathbb{R}^d)}\|T^{i,j}_{\Omega,{\mathbb{L}};\,s,\nu}b_\mathbb{L}\|_{L^1(\mathbb{R}^d)}.
	\end{align*}
	By \eqref{deftlq},	it is easy to show that
	$$\|T^{i,j}_{\Omega,{\mathbb{L}};\,s,\nu}b_\mathbb{L}\|_{L^1(\mathbb{R}^d)}\lesssim j2^{-j\gamma(d-1)}\|\Omega_i\|_{L^{\infty}(\mathbb{S}^{d-1})}\|b_\mathbb{L}\|_{L^1(\mathbb{R}^d)}.$$
	Let $W_{k,s,\,\kappa}^j$ be the kernel of $(I-P_{s-j\kappa} )\tilde{\Lambda} _k$, then by the mean value formula, we obtain
	\begin{eqnarray}\label{equation2.wks}\int_{\mathbb{R}^d}|W_{k,s,\kappa}^j(y)|dy&\leq & \int_{\mathbb{R}^{d}} \int_{\mathbb{R}^{d}} \big|\mathcal{F}^{-1}{\tilde{\beta}}_k(y) -\mathcal{F}^{-1}{\tilde{\beta}}_k(y-z) \big|\omega_{s-j\kappa}(z)dzdy 
		\lesssim 2^{s-j\kappa-k}. 
	\end{eqnarray}
	By  the proof of \cite[Lemma 3.2]{lai2}, it holds that $\| (I-{G_{\nu}^j)}\Lambda _k\|_{L^1(\mathbb{R}^d)}\lesssim 1 $.
	Hence
	\begin{align}
		\| (I-{G_{\nu}^j)}(I-P_{s-j\kappa} )\Lambda _kT^{i,j}_{\Omega,{\mathbb{L}};\,s,\nu}b_\mathbb{L} \|_{L^1(\mathbb{R}^d)}
		\lesssim j 2^{-j\gamma(d-1)+s-k-j\kappa} \|b_\mathbb{L}\|_{L^1(\mathbb{R}^d)}.\label{last1}
	\end{align}
	Let  $m=s-[j\varepsilon_0]$, with $0<\varepsilon_0<1$. Since ${\rm card}(\mathfrak{E}^j)\lesssim 2^{j\gamma(d-1)}$. Then  \eqref{last2} and  \eqref{last1} lead to
	\begin{align*}
		& \sum_{\nu}  \| (I-{G_{\nu}^j)}(I-P_{s-j\kappa} )T^{i,j}_{\Omega,{\mathbb{L}};\,s,\nu}(b_\mathbb{L})\|_{L^1(\mathbb{R}^d)}\\
		&\leq  \Big (\sum_{\nu}  \sum_{k<m} + \sum_{\nu}  \sum_{k\geq m} \Big)
		\| (I-P_{s-j\kappa} ) (I-{G_{\nu}^j)}\Lambda _kT^{i,j}_{\Omega,{\mathbb{L}};\,s,\nu}b_\mathbb{L}\|_{L^1(\mathbb{R}^d)} \\
		&\lesssim (2^{{s_1}j}+2^{{s_2}j})j\|\Omega_i\|_{L^{\infty}(\mathbb{S}^{d-1})}  \|b_\mathbb{L}\|_{L^1(\mathbb{R}^d)},
	\end{align*}
	where $s_1=\big(\tau-\varepsilon_0+\gamma(1+2M)\big)$ and $s_2=-\kappa+\varepsilon_0$.
	
	We can now choose $0\ll \tau \ll \gamma \ll \varepsilon_0< \kappa < 1$ such that $\max\{s_1,s_2\}<0$. Let $s_0=-\max\{s_1,s_2\}$, then the proof of Lemma \ref{lem2.6wzd} is finished now.
	\vspace{0.6cm}
	
	{\bf {Step 3. Final conclusion.}}
	
	{	We  are ready to combine  \eqref{wmm}, Lemm \ref{lem2.4} and \eqref{ddd} to conclude the estimate for \eqref{equation2.6}.
		\begin{eqnarray*}
			&&\Big|\Big\{x\in\mathbb{R}^d\backslash E:\, \Big|\sum_j\sum_s\sum_{\mathbb{L}\in\mathcal{S}}T_{\Omega,A_\mathbb{L};s,j}b_\mathbb{L}(x)\Big|>1/4\Big\}\Big|\\
			&&\leq 8\|{\rm D}_{1}\|_{L^1(\mathbb{R}^d)}|
			+16\|{\rm D}_{2}-{\rm D}_{2}^*\|_{L^1(\mathbb{R}^d)}+\Big|\Big\{x\in\mathbb{R}^d:\, |{\rm D}_{2}^*|>1/16\Big\}\Big|\\
			&&\lesssim\|f\|_{L^1(\mathbb{R}^d)}.
		\end{eqnarray*}
		
		Then, this estimate, together with \eqref{good} and  \eqref{errorterm}, leads to inequality \eqref{equation2.2} in Theorem \ref{thm1.4}.}

	\subsection{Proof of  \eqref{equation2.3} in Theorem \ref{thm1.4}}
	It suffices to prove (\ref{equation2.3}) for $\lambda=1$. For a bounded function $f$ with compact support, we employ the Calder\'on-Zygmund decomposition to $|f|$ at level $1$ then obtain a collection
	of non-overlapping dyadic cubes $\mathcal{S}=\{Q\}$, such that
	$$\|f\|_{L^{\infty}(\mathbb{R}^d\backslash \cup_{Q\in\mathcal{S}}Q)}\lesssim 1,\quad \int_{Q}|f(x)|dx \lesssim |Q|,\qquad \hbox{and\ } \sum_{Q\in\mathcal{S}}|Q|\lesssim \int_{\mathbb{R}^d}|f(x)|dx.$$
	Let $E=\cup_{Q\in\mathcal{S}}100dQ$.
	With the same notations as in the proof of \eqref{equation2.2}, for $x\in\mathbb{R}^d\backslash E$, we write
	$$\widetilde{T}_{\Omega,A}b(x)=\sum_j\sum_{Q\in\mathcal{S}}T_{\Omega,\,A_Q,j}b_{Q}(x)-\sum_{Q\in\mathcal{S}}
	\sum_{n=1}^d\partial_nA_{Q}(x)T_{\Omega}^nb_Q(x).$$
	By estimate (\ref{equation2.6}), the proof of (\ref{equation2.3}) can be reduced to show that for each $n$ with $1\leq n\leq d$,
	$$\Big|\Big\{x\in\mathbb{R}^d\backslash E:\, \Big|\sum_{Q\in\mathcal{S}}\partial_nA_{Q}(x)T_nb_Q(x)\Big|>1/4\Big\}\Big|\lesssim \int_{\mathbb{R}^d}|f(x)|dx.$$
But this inequality has already been proved in \cite[Section 3]{hutao}. Then the proof of  \eqref{equation2.3} is finished.\qed
	
	\subsection{Proof of Corollary \ref{cor2.1}}
	
	The proof of Corollary \ref{cor2.1} is now standard. We present the proof here mainly to make the constant of the norm inequality clearly.  Consider the case $p\in (1,2]$.		
	Let 	$$f_{\lambda}(x)= \Big \{\begin{array}{ll}
		f(x),\,&|f(x)|>\lambda\\
		0,\,&|f(x)|\leq \lambda;\end{array}
	$$
	and
	$$f^{\lambda}(x)= \Big \{\begin{array}{ll}
		0,\,&|f(x)|>\lambda\\
		f(x),\,&|f(x)|\leq \lambda \end{array}
	$$
	By \eqref{equation2.2},  we have
	\begin{align*}
		p\int_{0}^{\infty}\lambda^{p-1}\{x\in\mathbb{R}^d:|{T}_{\Omega,A}f_{\lambda}(x)|>\lambda/2\}\,d\lambda&\leq p\int_{0}^{\infty}\lambda^{p-1}\int_{\mathbb{R}^d}\frac{|f_{\lambda}(x)|}{\lambda}\log\left(e+\frac{|f_{\lambda}(x)|}{\lambda}\right)dx\,d\lambda\\
				&\leq \left(\frac{p}{p-1}\right)^2\|f\|_{L^p(\mathbb{R}^d)}^p.
	\end{align*}
	$L^2(\mathbb{R}^d)$ boundedness of $T_{\Omega,\,A}$   implies that
	\begin{align*}
		p\int_{0}^{\infty}\lambda^{p-1}\{x\in\mathbb{R}^d:|{T}_{\Omega,A}f^{\lambda}(x)|>\lambda/2\}\,d\lambda\leq p\int_{|f(x)|}^{\infty}\lambda^{p-1}\lambda ^{-2}\|f^{\lambda}\|_{L^2(\mathbb{R}^d)}\,d\lambda
		\leq\frac{p}{2-p}\|f\|_{L^p(\mathbb{R}^d)}^p.
	\end{align*}
	Since $p\in (1,2]$, we have
		$$\|T_{\Omega,\,A}f\|_{L^p(\mathbb{R}^d)}=\bigg(p\int_{0}^{\infty}\lambda^{p-1}\{x\in\mathbb{R}^d:|{T}_{\Omega,A}f(x)|>\lambda\}\,d\lambda\bigg)^{1/p}\leq (p')^2\|f\|_{L^p(\mathbb{R}^d)}.$$
	When  $p\in (2,\infty)$,   by \eqref{equation2.3}, we know $\widetilde{T}_{\Omega,A}f(x)$ is of weak type $(1,1)$. Combining the $L^2(\mathbb{R}^d)$ boundedness of $\widetilde{T}_{\Omega,A}f(x)$ and the Marcinkiewica interpolation theorem, we have
	$$\|\widetilde{T}_{\Omega,\,A}f\|_{L^{p'}(\mathbb{R}^d)}\leq p'\|f\|_{L^{p'}(\mathbb{R}^d)}.$$
	By duality, it holds that
	$$\|{T}_{\Omega,\,A}f\|_{L^{p}(\mathbb{R}^d)}\leq p\|f\|_{L^{p}(\mathbb{R}^d)}.$$
	This completes the proof of  Corollary \ref{cor2.1}.\qed

	\section{$(L(\log L)^{\beta},\,L^{r})$ bilinear sparse domination of $T_{\Omega,\,A}$ }\label{ssp1}
	In this section, we will establish a bilinear sparse domination of $T_{\Omega,\,A}$. It is worth pointing out that we do not need to prove the self-adjoint property as was shown in the previous articles (see \cite[Appendix A]{ccdo} or \cite{Barron2017}). Throughout this section, we always assume that $\Omega\in L^{\infty}(\mathbb{S}^{d-1})$ and satisfies the vanishing condition \eqref{equa:1.1}.
	
	Suppose that $\mathcal{S}$
	is a sparse family of cubes. Associated with  $\mathcal{S}$ and constants $\beta,\,r\in (0,\,\infty)$, we define the bilinear sparse operator
	$\mathcal{A}_{\mathcal{S},\,L(\log L)^\beta,L^r}$  by
	$$\mathcal{A}_{\mathcal{S},\,L(\log L)^\beta,\,L^r}(f_1,\,f_2)=\sum_{Q\in\mathcal{S}}\langle |f_1|\rangle_{L(\log L)^\beta,\,Q}\langle |f_2|\rangle_{Q,\,r}|Q|.$$
	Let us introduce the definition of $(L(\log L)^{\beta},\,L^{r})$ bilinear sparse domination.
	\begin{definition}[\bf{$(L(\log L)^{\beta},\,L^{r})$ bilinear sparse domination}]\label{defn3.1}
		Let $T$ be a sublinear operator acting on $\cup_{p\geq 1}L^p(\mathbb{R}^d)$ and $\beta,\,r\in (0,\,\infty)$. We say that $T$ enjoys a
		$(L(\log L)^{\beta},\,L^{r})$ bilinear sparse domination with a bound $A$, if
		\begin{eqnarray*}\bigg|\int_{\mathbb{R}^d}f_2(x)Tf_1(x)dx\bigg|\le A \sup_{\mathcal{S}}\mathcal{A}_{\mathcal{S},\,L(\log L)^{\beta},\,L^{r}}(f_1,\,f_2)
		\end{eqnarray*}
		holds for  bounded functions {$f_1$ and $f_2$ with compact supports},
		where the supremum is taken over all sparse families of cubes.
	\end{definition}
	The main purpose of this section is to demonstrate the following Theorem.
	\begin{theorem}\label{thm3.1}
	Under the hypothesis of Theorem \ref{thm1.2}, for any $r\in (1,\,2]$, the operator $T_{\Omega,\,A}$ enjoys a bilinear
	$(L\log L,\,L^r)$ sparse domination with bound $Cr'\|\Omega\|_{L^{\infty}(\mathbb{S}^{d-1})}$.
\end{theorem}
In order to show Theorem \ref{thm3.1}, we first give some notions and several lemmas.
	For $s\in \mathbb{Z}$, let
	$$K_s(x,\,y)=\frac{\Omega(x-y)}{|x-y|^{d+1}}(A(x)-A(y)-\nabla A(y)(x-y))\phi_{s}(x-y)$$
	and define the bilinear operator $\Lambda ^A$ by
	$$\Lambda ^A(f_1,f_2)=\sum_{s}\int_{\mathbb{R}^{d}} \int_{\mathbb{R}^{d}}  K_s(x,\,y)f_1(y)f_2(x)dydx.$$
	Then {H\"{o}lder's inequality}, together with the boundedness of $T_{\Omega,\,A}$,  Corollary \ref{cor2.1}, tells us that for $r\in (1,\,2]$, the following inequality holds
	$${\Big|\int_{\mathbb{R}^d}f_2(x)T_{\Omega,\,A}f_1(x)dx\Big|\lesssim r'\|f_1\|_{L^{r'}(\mathbb{R}^d)}\|f_2\|_{L^{r}(\mathbb{R}^d)}.}
	$$
	Therefore, if $\|f_1\|_{L^{r'}}\|f_2\|_{L^{r}}\neq 0,$ we have
	\begin{eqnarray}\label{equation3.1}
		C_T=\frac{\big|\Lambda ^A(f_1,\,f_2)\big|}{\| f_1\|_{L^{r'}(\mathbb{R}^d)} \| f_2\|_{ L^{r}(\mathbb{R}^d)}}\lesssim r'  ,\,\,\,r\in (1,\,2].
	\end{eqnarray}
	
	\begin{definition}[\bf{Localized spaces over stopping collections}, \cite{ccdo}]\label{defn3.2}
		
		Let $Q\in\mathcal{D}$ be a fixed dyadic cube in $\mathbb{R}^d$. A collection $\mathcal{Q}\subset\mathcal{D}$ of dyadic cubes is a stopping collection with top $Q$ if the elements of $\mathcal{Q}$ are pairwise disjoint, contained in $3Q$ and enjoy the following properties:
		\begin{eqnarray*}
			&&{\rm(1).} \quad  \mathbb{L}_1,\,\mathbb{L}_2\in\mathcal{Q},\,\,\mathbb{L}_1\cap \mathbb{L}_2\not=\emptyset\Rightarrow \mathbb{L}_1=\mathbb{L}_2,\,\,\,\mathbb{L}\in \mathcal{Q}\Rightarrow \mathbb{L}\in 3Q;\\
			&&{\rm(2).} \quad \mathbb{L}_1,\,\mathbb{L}_2\in\mathcal{Q},\,|s_{\mathbb{L}_1}-s_{\mathbb{L}_2}|\geq 8\Rightarrow 7\mathbb{L}_1\cap 7\mathbb{L}_2=\emptyset,
		 \bigcup_{\mathbb{L}\in \mathcal{Q},\,3\mathbb{L}\cap 2Q\not =\emptyset}9\mathbb{L}\subset \bigcup_{\mathbb{L}\in \mathcal{Q}}\mathbb{L}:={\rm sh}\mathcal{Q}.\nonumber
		\end{eqnarray*}
	\end{definition}
	\begin{definition}[\bf{Norm $\mathcal{Y}_p(\mathcal{Q})$}, \cite{ccdo}]
		For $p\in [1,\,\infty]$, define $\mathcal{Y}_p(\mathcal{Q})$ as the subspace of $L^p(\mathbb{R}^d)$ of functions satisfying ${\rm supp}\,h\subset 3Q$ and $\|h\|_{\mathcal{Y}_p(\mathcal{Q})}<\infty$, where
		$$ \|h\|_{\mathcal{Y}_p(\mathcal{Q})}=\Bigg\{\begin{array}{ll}
			\max\big\{\|h\chi_{\mathbb{R}^d\backslash {\rm sh}\mathcal{Q}}\|_{L^{\infty}(\mathbb{R}^d)},\,\,\sup_{\mathbb{L}\in\mathcal{Q}}\inf_{x\in \mathbb{L}^*}M_ph(x)\big\},\,\,&p\in [1,\,\infty)\\
			\|h\|_{L^{\infty}(\mathbb{R}^d)},\,&p=\infty,\end{array}$$
		and ${\mathbb{L}}^*=2^5\mathbb{L}$. 
		
		We also define
		$$ \|h\|_{\mathcal{Y}_{L\log L}(\mathcal{Q})}=
		\max\big\{\|h\chi_{\mathbb{R}^d\backslash {\rm sh}\mathcal{Q}}\|_{L^{\infty}(\mathbb{R}^d)},\,\,\sup_{\mathbb{L}\in\mathcal{Q}}\inf_{x\in \mathbb{L}^*}M_{L\log L}h(x)\big\}.$$
	Moreover, $\mathcal{X}_p(\mathcal{Q})$ ($\mathcal{X}_{L\log L}(\mathcal{Q})$) is denoted to be the subspace of $\mathcal{Y}_p(\mathcal{Q})$ ($\mathcal{Y}_{L\log L}(\mathcal{Q})$) of functions satisfying
		$$b(x)=\sum_{\mathbb{L}\in\mathcal{Q}}b_\mathbb{L}(x),\,\,\,{\rm supp}\,b_\mathbb{L}\subset \mathbb{L}.$$
		Furthermore, we write $b\in \dot{\mathcal{X}}_p(\mathcal Q)$ if $b=\sum_{\mathbb{L}\in\mathcal{Q}}b_\mathbb{L}(x)\in \mathcal{X}_p(\mathcal{Q})$ and for all $\mathbb{L}\in\mathcal{Q}$, $\int_{\mathbb{L}}b_\mathbb{L}(x)dx=0$.
		For $b\in\mathcal{X}_p(\mathcal{Q})$ ($\mathcal{X}_{L\log L}(\mathcal{Q})$), we use $\|b\|_{\mathcal{X}_p(\mathcal{Q})}$ ($\mathcal{X}_{L\log L}(\mathcal{Q})$) to denote $\|b\|_{\mathcal{Y}_p(\mathcal{Q})}$ ($\mathcal{Y}_{L\log L}(\mathcal{Q})$), and
		similarly for $b\in\dot{\mathcal{X}}_p(\mathcal{Q})$ ($\dot{\mathcal{X}}_{L\log L}(\mathcal{Q})$).
	\end{definition}
	
	For a dyadic cube $Q$ with $s_Q=\log_2\ell(Q)$, define the operators $T^Q_{\Omega,\,A}f$ and $\Lambda^{A,s_Q}_Q$ by
	$$T^Q_{\Omega,\,A}f(x)=\sum_{s\leq s_Q}\int_{\mathbb{R}^d}K_s(x,y)f_1(y)dy;
	$$
	$$\Lambda ^{A,s_Q}_Q(f_1,f_2)=\sum_{s\leq s_Q}\int_{\mathbb{R}^d}\int_{\mathbb{R}^d}K_s(x,y)f_1(y)dy\, f_2(x)\,dx.
	$$
Then we have $$\Lambda^{A,s_Q}_Q(f_1\chi_Q, f_2)=\Lambda^{A,s_Q}_Q(f_1\chi_Q, f_2\chi_{3Q}):=\Lambda_Q^A(f_1,f_2).$$

	Given a stopping collection $\mathcal{Q}$ with top $Q$, we define
	$$\Lambda_{\mathcal{Q}}^A(f_1,f_2):=\Lambda_Q^A(f_1,f_2)-\sum\limits_{\mathbb{L}\in \mathcal{Q},\mathbb{L}\subset Q}\Lambda_\mathbb{L}^A(f_1,f_2).$$
	We need the following lemma.
	\begin{lemma} \label{lem41}Let $\Omega$ be homogeneous of degree zero, satisfy the vanishing condition (\ref{equa:1.1}) and $\Omega\in L(\log L)^{2}(\mathbb{S}^{d-1})$. Let $A$ be a function  in $\mathbb{R}^d$ with  derivatives of order one in ${\rm BMO}(\mathbb{R}^d)$.  Then
		$$\|T^{Q}_{\Omega,\,A}f\|_{L^p(\mathbb{R}^d)}\lesssim \Big \{\begin{array}{ll}
		p'^2\|f\|_{L^p(\mathbb{R}^d)},\,&p\in (1,\,2];\\
		p\|f\|_{L^p(\mathbb{R}^d)},\,&p\in (2,\,\infty).\end{array}
		$$
	\end{lemma}
	The proof of Lemma \ref{lem41} can be infered from \cite{chenlu} with small modifications. We omit the details here.
	Then, the application of Lemma \ref{lem41} leads to that
	$$|\Lambda_{{Q}}^A(f_1,f_2)|\lesssim r'\| f_1\|_{L^{r'}(\mathbb{R}^d)} \| f_2\|_{ L^{r}(\mathbb{R}^d)}\lesssim r'|Q|
	\| f_1\|_{\mathcal{Y}_{r'}(\mathcal{Q})}\| f_2\|_{ \mathcal{Y}_{r}(\mathcal{Q})},
	$$
	and
	\begin{eqnarray*}
		\sum\limits_{\mathbb{L}\in \mathcal{Q},\mathbb{L}\subset Q}\Lambda_\mathbb{L}^A(f_1,f_2)
		\lesssim r'|Q|
		\| f_1\|_{\mathcal{Y}_{r'}(\mathcal{Q})}\| f_2\|_{ \mathcal{Y}_{r}(\mathcal{Q})}.
	\end{eqnarray*}
	Therefore, it is easy to see that
	\begin{eqnarray*}
		|\Lambda_{\mathcal{\mathcal{Q}}}^A(f_1,f_2)|\lesssim   r'|Q|
		\| f_1\|_{\mathcal{Y}_{r'}(\mathcal{Q})}\| f_2\|_{ \mathcal{Y}_{r}(\mathcal{Q})}.
	\end{eqnarray*}

	\begin{remark}
		Given a stopping collection $\mathcal{Q}$ with top $Q$. Suppose $b=\sum_{\mathbb{L}\in\mathcal{Q}}b_\mathbb{L}$ with $b_\mathbb{L}$ supported in $\mathbb{L}$  and $g$ is supported in $3Q$. Fixing $\mathbb{L}\in \mathcal{Q}$ with scale $s_\mathbb{L}=\log_2\ell(\mathbb{L})$, we can get
		$$\Lambda_{\mathcal{Q}}^A(b_\mathbb{L},g)=\Lambda_Q^{A,s_Q}(b_\mathbb{L},g)-\Lambda_\mathbb{L}^{A,s_\mathbb{L}}(b_\mathbb{L},g{\chi_{3\mathbb{L}}});$$
		$$\Lambda_{\mathcal{Q}}^A(b,g)= \int_{\mathbb{R}^{2d}} \sum\limits_{s} \sum\limits_{j\geq 1}K_{s}(x,y)b_{s-j}(y)g(x) dy\,dx,$$
		where $b_{s-j}=\sum\limits_{\mathbb{L}\in \mathcal{Q},s_\mathbb{L}=s-j}b_\mathbb{L}$.
	\end{remark}

	It is easy to have that
	$$\Lambda_{\mathcal{Q} }^A(b_\mathbb{L},\,h)= 	\Lambda_{\mathcal{Q} }^{A_\mathbb{L}} (b_\mathbb{L},\,h) -\sum_{n=1}^d	\Lambda_{n,\mathcal{Q} } (\partial_nA_{\mathbb{L}}b_\mathbb{L},\,h) ,$$
	where
	$$ \Lambda_{\mathcal{Q} }^{A_\mathbb{L}} (b_\mathbb{L},\,h)=\int_{\mathbb{R}^d}\int_{\mathbb{R}^d}\sum_{s\in\mathbb{Z}} \frac{\Omega(x-y)}{|x-y|^{d}}\frac{\big(A_\mathbb{L}(x)-A_\mathbb{L}(y)\big)}{|x-y|}\phi_s(x-y)b_\mathbb{L}(y) h(x)dydx;$$
	$$	\Lambda_{n,\mathcal{Q} } (\partial_nA_{\mathbb{L}}b_\mathbb{L},\,h) =\int_{\mathbb{R}^d}\int_{\mathbb{R}^{d}}\sum\limits_{s\in\mathbb{Z}}\frac{\Omega(x-y)}{|x-y|^{d}}\frac{\partial_nA_{\mathbb{L}}(y)(x_n-y_n)}{|x-y|}\phi_s(x-y)b_\mathbb{L}(y)h(x)dydx.$$
	The following lemma is also needed in our proof.
	\begin{lemma} \label{easy}
		Let $r\in (1,\,2]$, $Q\subset \mathcal{D}$ and $\mathcal{Q}$ be a stopping collection with top cube $Q$. Then for $b\in\dot{\mathcal{X}}_1(\mathcal{Q})$, $f_2$ with ${\rm supp}\, f_2 \subset 3Q$, we have
		\begin{eqnarray}\label{equation4.5lem}
			\Big	|\Lambda_{n,\mathcal{Q}}(\sum\limits_{\mathbb{L}\in  \mathcal{Q}   }\partial_nA_{\mathbb{L}}b_\mathbb{L},\,f_2) \Big| \leq r'|Q|\|\Omega\|_{L^{\infty}(\mathbb{S}^{d-1})}\|b\|_{\dot{\mathcal{X}}_{L\log L}(\mathcal{Q})}\|f_2\|_{\mathcal{Y}_{r}(\mathcal{Q})}.
		\end{eqnarray}
	\end{lemma}

	\begin{proof}We may assume $\|\Omega\|_{L^{\infty}(\mathbb{S}^{d-1})}=1$ and claim that for each $n$ $(1\leq n\leq d)$,
		the inequality 
		$$|\Lambda_{n,\mathcal{Q}}(f_1,\,f_2)|\lesssim r'|Q|\|f_2\|_{\mathcal{Y}_{r}(\mathcal{Q})}$$
	holds provided that ${\rm supp} f_1\subset Q$ and $\langle |f_1|\rangle_{Q}\leq 1$. 
	
	In fact, applying the Calder\'{o}n-Zygmund decomposition of $f_1$ at level $1$
		with respect to $\mathcal{Q}$, we may decompose $f_1$ as $f_1=b_1+g_1$,
		with the properties that		
		$$\|g_1\|_{L^{\infty}(\mathbb{R}^d)} \lesssim 1, \qquad  {\rm supp}\,g_1\subset Q, \qquad \|b_1\|_{\dot{\mathcal{X}}_{1}(\mathcal{Q})}\lesssim 1.$$
		An application of \cite[Lemma 2.7]{ccdo} may lead to that,
		\begin{eqnarray}
			|\Lambda_{n,\mathcal{Q}}(g_1,\,f_2)|
			&\lesssim &r'\|g_1\|_{L^{r'}(\mathbb{R}^{d})} \|f_2\|_{L^{r}(\mathbb{R}^{d})}\nonumber	\lesssim r'|Q| \|f_2\|_{\mathcal{Y}_{r}(\mathcal{Q})}.\nonumber
		\end{eqnarray}
		On the other hand, by \cite[Lemma 4.1]{ccdo} and the fact that  $\|b_1\|_{\dot{\mathcal{X}}_{1}(\mathcal{Q})}\lesssim 1$, we obtain 
		\begin{align}
		|	\Lambda_{n, \mathcal{Q}}(b_1,\,f_2)|\lesssim  r'|Q| \|{f_2}\|_{\mathcal{Y}_{r}(\mathcal{Q})}, \label{inzd2}
		\end{align}
	which verifies this claim.
				
		Now we return to the proof of (\ref{equation4.5lem}). By homogeneity, we may assume $\|b\|_{\dot{\mathcal{X}}_{L\log L}(\mathcal{Q})}=1$, which means that for each  $\mathbb{L}\in\mathcal{Q}$,
		$\|b_\mathbb{L}\|_{L\log L,\,L}\leq 1.$
		Observe that ${\rm supp}\,\sum\limits_{\mathbb{L}\in  \mathcal{Q}   }b_\mathbb{L}\partial_nA_{\mathbb{L}}\subset Q$, and by (\ref{holderandjohnniren}),
		$$\Big\|\sum\limits_{\mathbb{L}\in  \mathcal{Q}   }b_\mathbb{L}\partial_nA_{\mathbb{L}}\Big\|_{ L^{1}(\mathbb{R}^{d})}\lesssim \sum\limits_{\mathbb{L}\in  \mathcal{Q}   }|\mathbb{L}|\|b_\mathbb{L}\|_{ {L \log L, L}}\lesssim |Q|.$$
		Our claim now leads to
		$$\Big	|\Lambda_{n,\mathcal{Q}}(\sum\limits_{\mathbb{L}\in  \mathcal{Q}   }\partial_nA_{\mathbb{L}}b_\mathbb{L},\,f_2) \Big| \leq r'|Q|\|{f_2}\|_{\mathcal{Y}_{r}(\mathcal{Q})},$$
which completes the proof of Lemma \ref{easy}.			
	\end{proof}
	
	We need the following lemma, which is  a variation of Lemma 3.1 in \cite{ccdo}.
	\begin{lemma}\label{lemma t}
		Let $\Omega \in L^{\infty}(\mathbb{S}^{d-1})$.    Then for all $j\geq 1$, it holds that
		\begin{align} \label{4.11}
			\sum\limits_s \sum\limits_{\mathbb{L}\in \mathcal{Q},s_\mathbb{L}=s-j}	&\int_{\mathbb{R}^d}\int_{\mathbb{R}^d}\phi_s(x-y)\frac{|\Omega(x-y)|}{|x-y|^{d+1}}|A_\mathbb{L}(x)-A_\mathbb{L}(y)||b_{\mathbb{L}}(y)
			|\,dy|\,h(x)|\,dx\\
			&\leq j \|\Omega \|_{ L^{\infty}(\mathbb{S}^{d-1})} |Q|\|b\|_{  \dot{\mathcal X} _{L\log L}(\mathcal{Q})} \|h\|_{\mathcal{Y}_{1}(\mathcal{Q})}.\nonumber
		\end{align}
		
	\end{lemma}
	\begin{proof}
		Since  {$\|b_\mathbb{L}\|_{L\log L, \mathbb{L}}\lesssim \|b\|_{\dot{\mathcal X} _{L\log L}(\mathcal{Q})}$}, it suffices to show that
		if $\mathbb{L}\in \mathcal{Q}$ and $s=s_\mathbb{L}+j$, then
		\begin{align*}
			&\int_{\mathbb{R}^{d}}\int_{\mathbb{R}^{d}}\phi_s(x-y)\frac{|\Omega(x-y)|}{|x-y|^{d+1}}|A_\mathbb{L}(x)-A_\mathbb{L}(y)| |b_{\mathbb{L}}(y)|\,|h(x)|\,dydx\\
			&\qquad\leq j \|\Omega \|_{ L^{\infty}(\mathbb{S}^{d-1})}  |\mathbb{L}|\|b_\mathbb{L}\|_{L \log L, \mathbb{L}} \|h\|_{\mathcal{Y}_{1}(\mathcal{Q})}.
		\end{align*}
		A straightforward computation involving	Lemma \ref{lem2.3} yields that
		\begin{align*}
			|A_{\varphi_{\mathbb{L}}}(x)-A_{\varphi_{\mathbb{L}}}(y)|
						\lesssim j|x-y|.
		\end{align*}
		Then  it follows that
		\begin{align*} &\int_{\mathbb{R}^{d}}\int_{\mathbb{R}^d}\phi_s(x-y)\frac{|\Omega(x-y)|}{|x-y|^{d+1}}|A_\mathbb{L}(x)-A_\mathbb{L}(y)
			||b_{\mathbb{L}}(y)\,dy|h(x)|\,dx\\
			&\lesssim\| \Omega\|_{ L^{\infty}(\mathbb{S}^{d-1})}\int_{\mathbb{R}^d}\int_{\mathbb{R}^d}\frac{|A_{\varphi_{\mathbb{L}}}(x)-A_{\varphi_{\mathbb{L}}}(y)|}{|x-y|^{d+1}} |\phi_s(x-y) b_{\mathbb{L}}(y)|dy h(x)|\,dx\\
			&\leq \| \Omega\|_{ L^{\infty}(\mathbb{S}^{d-1})} j\|b_{\mathbb{L}}\|_{L^1(\mathbb{R}^d)} \inf\limits_{z\in \mathbb{L}^*}Mh(z)\\
			&\lesssim j | \Omega\|_{ L^{\infty}(\mathbb{S}^{d-1})}  |\mathbb{L}|\|b\|_{\dot{\mathcal{X}}_{L\log L}(\mathcal{Q})}\|h\|_{\mathcal{Y}_1(\mathcal{Q})}.
		\end{align*}
	Note that the sums on the left side  of (\ref{4.11})  are  equivlent with the sum over $\mathbb{L}$ which is contained in $Q$.	This completes the proof of Lemma \ref{lemma t}.	
	\end{proof}
Let $T^{j}_{\Omega,{\mathbb{L}};\,s,\nu}b_{s-j}$ be the same as in \eqref{deftlq} with $\Omega_i$ replaced by $\Omega$, we give a stronger version of	 Lemma \ref{lem2.5}.
	\begin{lemma}\label{lemma2.5}
		Let $\beta>1$, $\gamma\in (0,1)$, then
		$$\Big|\langle  \sum_{\nu} \sum_sG_{\nu}^j(I-P_{s-j\kappa})T^{j}_{\Omega,{\mathbb{L}};\,s,\nu}b_{s-j},  h\rangle\Big|\lesssim  j 2^{-j\gamma/2 }  \|\Omega\|_{L^{\infty}(\mathbb{S}^{d-1})}|Q| \|b\|_{\dot{\mathcal{X}}_1(\mathcal{Q})} \|h\|_{\mathcal{Y}_{\beta}(\mathcal{Q})}.$$
	\end{lemma}

	\begin{proof}

		We may assume $\|\Omega\|_{L^{\infty}(\mathbb{S}^{d-1})}=1$ and  denote by $t$  the dual exponent of $\beta$.
		It is enough to show for each $r\geq 1$ and $t=2r$,
		\begin{align}
			\frac{1}{|Q|^{\frac{1}{t}}}	\Big\| \sum_{\nu} \sum_sG_{\nu}^j(I-P_{s-j\kappa})T^{j}_{\Omega,{\mathbb{L}};\,s,\nu}b_{s-j} \Big\|_{L^{t}(\mathbb{R}^d)}
			\lesssim 2^{-j\gamma/2}  \|b\|_{\dot{\mathcal{X}}_1(\mathcal{Q})}.\label{hG}
		\end{align}	
		Denote $M_\nu=  \sum\limits_s {G_{\nu}^j}(I-P_{s-j\kappa} )T^{j}_{\Omega,{\mathbb{L}};\,s,\nu}b_{s-j}  $. Applying  Plancherel's theorem, we get
		\begin{align} \label{strong}
		\Big\|\sum\limits_{\nu_1,\cdots,\nu_r} \prod\limits_{k=1}^{r} M_{\nu_k}\Big\|^2_{L^{2}(\mathbb{R}^d)} 
			&\lesssim  2^{j\gamma(d-2)r}\sum\limits_{\nu_1,\cdots,\nu_r} \Big \| \prod\limits_{k=1}^{r}  \sum\limits_{s}(I-P_{s-j\kappa} )T^{j}_{\Omega,{L};\,s,\nu}b_{s-j}  \Big\|_{L^{2}(\mathbb{R}^d)} ^2\\
			&\lesssim  2^{j\gamma(d-2)r} 2^{j\gamma(d-1)r}\Big\| \prod\limits_{k=1}^{r}  \sum\limits_{s}(I-P_{s-j\kappa} )T^{j}_{\Omega,{\mathbb{L}};\,s,\nu}b_{s-j} \Big\|_{L^{2}(\mathbb{R}^d)} ^2\nonumber\\
			&	\lesssim 2^{-j\gamma r}2^{j\gamma(d-1)t} \sup\limits_{\nu} \Big\|  \sum\limits_{s}(I-P_{s-j\kappa} )T^{j}_{\Omega,{\mathbb{L}};\,s,\nu}b_{s-j}  \Big\|_{L^{t}(\mathbb{R}^d)} ^t.\nonumber 
		\end{align}	
		On the other hand, inequality \eqref{hjk} tells us that
		\begin{align*}
			\big|	(I-P_{s-j\kappa} )T^{j}_{\Omega,{\mathbb{L}};\,s,\nu}b_{s-j} (x)\big|
			\lesssim jH_{s,\nu}^j*|b_{s-j}|(x),
		\end{align*}
		where $ H_{s,\nu}^j(x)=2^{-sd}\chi_{\mathcal{R}_{s\nu}^j}(x)$ is defined in Section 3.
		Following the proof of (4.6) in \cite{ccdo}, for each fixed $\nu$, one may easily verify that
		\begin{align}
			\sum\limits_{s_1\geq\cdots\geq s_t}\int \bigg(\prod_{k=1}^{t}H_{s_k,\nu}^j (x-y_k) |b_{s_k-j}(y_k) |\bigg) dy_1\cdots dy_t dx
			\lesssim 2^{-j\gamma(d-1)t} |Q| \|b\|^t_{\dot{\mathcal{X}}_1(\mathcal{Q})}. \label{strong2}
		\end{align}

		Combining the above estimate  \eqref{strong2} with \eqref{strong}, one may obtain
		\begin{align*}
			&	\Big\| \sum_{\nu} \sum_sG_{\nu}^j(I-P_{s-j\kappa})T^{j}_{\Omega,{\mathbb{L}};\,s,\nu}b_{s-j} \Big \|^t_{L^{t}(\mathbb{R}^d)} \lesssim 2^{-j\gamma r}j |Q| \|b\|^t_{\dot{\mathcal{X}}_1(\mathcal{Q})}.
		\end{align*}
		
		This finishes the proof of Lemma \ref{lemma2.5}.
	\end{proof}
	
	The following lemma concerns on a decay estimate, which will be used later.
	\begin{lemma}\label{e}
		Under the hypothesis of Theorem \ref{thm1.2},  for all $j\geq 1$ and some $\varepsilon\in(0,1)$, we have
		\begin{align*}
			\Big |\langle {\sum_s\sum_{\mathbb{L}\in\mathcal{Q},s_\mathbb{L}=s-j} T_{\Omega,A_{\mathbb{L}};\,s,j}b_\mathbb{L}},  h\rangle\Big| \lesssim   j2^{-\varepsilon j} \|\Omega\|_{L^{\infty}(\mathbb{S}^{d-1})}|Q|\|b\|_{\dot{\mathcal{X}}_{1}(\mathcal{Q})}\|{h}\|_{\mathcal{Y}_{\infty}(\mathcal{Q})}.
		\end{align*}
	\end{lemma}
	\begin{proof}
		For $1\leq j\leq j_0$, as in the estimate for ${D_1}$ in 	\eqref{wmm}, it is easy to deduce that
		\begin{align*}
			\Big	|\langle \sum_s\sum_{\mathbb{L}\in\mathcal{Q},s_\mathbb{L}=s-j}  T_{\Omega,A_{\mathbb{L}};\,s,j}b_\mathbb{L}(x),  h\rangle\Big| \lesssim   j  |Q|\|\Omega\|_{L^{\infty} (\mathbb{S}^{d-1})} \|b\|_{\dot{\mathcal{X}}_1(\mathcal{Q})}\|h\|_{\mathcal{Y}_{\infty}(\mathcal{Q})}.
		\end{align*}
		
		For $ j> j_0$, let $\varepsilon=\min\{ \tau, (1-\kappa), (1-\tau), s_0,\gamma/2 \}$. By Lemma \ref{lem2.4},  Lemma \ref{ljj},  Lemma \ref{lem2.6wzd}, and Lemma \ref{lemma2.5},
		we have
		\begin{align*}
			&	\Big|\big\langle \sum_s\sum_{\mathbb{L}\in\mathcal{Q},s_\mathbb{L}=s-j}  T_{\Omega,A_{\mathbb{L}};\,s,j}b_\mathbb{L},  h\big\rangle\Big| \\
			&\lesssim \Big |\big\langle \sum_s\sum_{\mathbb{L}\in\mathcal{Q},s_\mathbb{L}=s-j} T_{\Omega,A_{\mathbb{L}};\,s,j}b_\mathbb{L}(x)- T^{j}_{\Omega,A_\mathbb{L};\,s}b_\mathbb{L},  h\big\rangle\Big| +\Big |\big\langle \sum_s\sum_{\mathbb{L}\in\mathcal{Q},s_\mathbb{L}=s-j} T^{j}_{\Omega,A_\mathbb{L};\,s}b_\mathbb{L},  h\big\rangle\Big|\\
			&\lesssim  j2^{-j\varepsilon}\|\Omega\|_{L^{\infty}(\mathbb{S}^{d-1})}\sum_s\sum_{\mathbb{L}\in\mathcal{Q},s_\mathbb{L}=s-j}\|b_\mathbb{L}\|_{L^1(\mathbb{R}^d)}\|h\|_{\mathcal{Y}_{\infty}(\mathcal{Q})}
			\\
			&	\lesssim j2^{-\varepsilon j}  \|\Omega\|_{L^{\infty}(\mathbb{S}^{d-1})}|Q| \|b\|_{\dot{\mathcal{X}}_1(\mathcal{Q})} \|h\|_{\mathcal{Y}_{\infty}(\mathcal{Q})}.
		\end{align*}
		
	\end{proof}

	We are now in a position to consider the proof of the bad part.
	\begin{lemma}\label{diff}
		For all dyadic lattices $\mathcal{D}$, all cube $Q\in\mathcal{D}$  and all stopping collections $\mathcal{Q}\subset \mathcal{D}$, it holds that
		\begin{eqnarray*}
			\Big|\langle 		\sum_{j\geq 1}\sum_s\sum_{\mathbb{L}\in\mathcal{Q},s_\mathbb{L}=s-j} T_{\Omega,A_{\mathbb{L}};\,s,j}b_\mathbb{L},  h\rangle\Big| \lesssim r'|Q|\|\Omega\|_{L^{\infty}(\mathbb{S}^{d-1})}\|b\|_{\dot{\mathcal{X}}_{L\log L}(\mathcal{Q})}\|h\|_{\mathcal{Y}_r(\mathcal{Q})}.
		\end{eqnarray*}
		
	\end{lemma}
	\begin{proof}
		For some $0<\varepsilon <1$ and any $r>1$, interpolating between the estimates in Lemma \ref{lemma t} and Lemma \ref{e} gives that
		\begin{align*}
			\Big	|\big\langle \sum_s\sum_{\mathbb{L}\in\mathcal{Q},s_\mathbb{L}=s-j}  T_{\Omega,\,A_\mathbb{L};\,s,j}b_\mathbb{L},  h\big\rangle\Big|& \lesssim   j 2^{-\varepsilon j\frac{r-1}{r}} \|\Omega\|_{L^{\infty}(\mathbb{S}^{d-1})}|Q|\|b\|_{\dot{\mathcal{X}}_{L\log L}(\mathcal{Q})}\|{h}\|_{\mathcal{Y}_{r}(\mathcal{Q})}\\
			& \lesssim 2^{-\varepsilon j\frac{r-1}{2r}} \|\Omega\|_{L^{\infty}(\mathbb{S}^{d-1})}|Q|\|b\|_{\dot{\mathcal{X}}_{L\log L}(\mathcal{Q})}\|{h}\|_{\mathcal{Y}_{r}(\mathcal{Q})}.
		\end{align*}
		Summing over the last inequality for $j\geq 1$ yields that
		\begin{align*}
			\Big |\big\langle \sum_{j\geq 1}\sum_s\sum_{\mathbb{L}\in\mathcal{Q},s_\mathbb{L}=s-j} T_{\Omega,A_{\mathbb{L}};\,s,j}b_\mathbb{L},  h\big\rangle\Big|\lesssim  \frac{r}{r-1}\|\Omega\|_{L^{\infty}(\mathbb{S}^{d-1})}|Q|\|b\|_{\dot{\mathcal{X}}_{L\log L}(\mathcal{Q})}\|{h}\|_{\mathcal{Y}_{r}(\mathcal{Q})},
		\end{align*}
	which finishes the proof of Lemma \ref{diff}.
	\end{proof}
	We are ready to prove Theorem \ref{thm3.1}, the  sparse domination of $T_{\Omega,\,A}$.

	\begin{proof}[Proof of Theorem \ref{thm3.1}]
		We may assume $\|\Omega\|_{L^{\infty}(\mathbb{S}^{d-1})}=1$.
		Obviously, it suffices  to prove  the estimate
		\begin{align}
			|\Lambda^{A}(f_1,\,f_2)|\lesssim r'\sup_{\mathcal{S}}\mathcal{A}_{\mathcal{S},\,L\log L,\,L^r}(f_1,\,f_2)\label{wmmmm}
		\end{align}
		holds for all bounded functions $f_1$ and $f_2$ with compact supports.
		
		To prove (\ref{wmmmm}), we follow the main steps of Theorem C in \cite{ccdo}. But we don't need to  prove the self-adjoint property as in \cite[Appendix A]{ccdo}.
		In addition, we only  need to decompose $ f_1$. These factors make our proof  more elegant and clear. We divide the proof into two steps.
		
		{\bf	Step 1. Auxiliary estimate.}
		Using the Calder\'{o}n-Zygmund decomposition
		with respect to $\mathcal{Q}$. We may decompose $f_1=b_1+g_1$
		with $b_1=\sum\limits_{\mathbb{L}\in \mathcal{Q}}(f_1-\langle f_1\rangle _\mathbb{L}){\chi}_\mathbb{L}$ such that  for any $p\geq1$,
		
		$$\|g_1\|_{\mathcal{Y}_{\infty}(\mathcal{Q})}\lesssim \|f_1\|_{\mathcal{Y}_{p}(\mathcal{Q})}, \qquad \|b_1\|_{\dot{\mathcal{X}}_{p}(\mathcal{Q})}\lesssim \|f_1\|_{\mathcal{Y}_{p}(\mathcal{Q})}.$$ Then			
		\begin{align*}
			\Lambda^A_{\mathcal{Q}}(f_1,\,f_2)
			=	\Lambda^A_{\mathcal{Q}}(g_1,\,f_2)+		\Lambda^A_{\mathcal{Q}}(b_1,\,f_2).
		\end{align*}
		By inequality \eqref{equation3.1}, it is obvious that
		\begin{align*}
			|	\Lambda^A_{\mathcal{Q}}(g_1,\,f_2)| \leq r'|Q|\|g_1\|_{\mathcal{Y}_{r'}(\mathcal{Q})} \|f_2\|_{\mathcal{Y}_{r}(\mathcal{Q})}
			&\leq r' |Q|\|f_1\|_{\mathcal{Y}_{1}(\mathcal{Q})} \|f_2\|_{\mathcal{Y}_{r}(\mathcal{Q})}\\
			&	\leq r'|Q| \|f_1\|_{\mathcal{Y}_{L\log L}(\mathcal{Q})} \|f_2\|_{\mathcal{Y}_{r}(\mathcal{Q})}.
		\end{align*}
		On the other hand, the fact that {$ \|b_1\|_{\dot {\mathcal{X}}_{L\log L}(\mathcal{Q})}\lesssim \|f_1\|_{\mathcal{Y}_{L\log L}(\mathcal{Q})}$},  Lemma \ref{easy}, together with Lemma \ref{diff} yield that
		\begin{align*}
			&|\Lambda^A_{\mathcal{Q}}(b_1,\,f_2)|\leq r'|Q|\|b_1\|_{\dot{\mathcal{X}}_{L\log L}(\mathcal{Q})}\|f_2\|_{\mathcal{Y}_{r}(\mathcal{Q})}\lesssim r'|Q| \|f_1\|_{\mathcal{Y}_{L\log L}(\mathcal{Q})} \|f_2\|_{\mathcal{Y}_{r}(\mathcal{Q})}.
		\end{align*}
	Therefore
		\begin{align*}
			| \Lambda_Q^A(f_1,f_2)|& \leq |\Lambda_{\mathcal{Q}}^A(f_1,f_2)|+\sum\limits_{\mathbb{L}\in \mathcal{Q},\mathbb{L}\subset Q}|\Lambda_L^A(f_1,f_2)|\\
			&\leq  r'|Q| \langle f_1\rangle _{{L\log L},3Q} \langle f_2\rangle_{{r},3Q}  +\sum\limits_{\mathbb{L}\in \mathcal{Q},\mathbb{L}\subset Q}|\Lambda_\mathbb{L}^A(f_1\chi_\mathbb{L},f_2)|.
		\end{align*}
		{\bf Step 2. Iterative process.}
		Fix $f_j\in L^{r_j}(\mathbb{R}^d)$, $j=1,2$ with compact support, we choose the dyadic lattice $\mathcal{D}$ such that we can  find a large enough dyadic cube  $Q_0\in \mathcal{D}$ with supp $f_1\in Q_0$,  supp $f_2\in 3Q_0$.
		
		Set $\mathcal{S}_0=\{Q_0\}$ and let
		$\mathcal{S}_{k+1}$  be the maximal cubes $\mathbb{L}\in \mathcal{Q}$ such that $9 \mathbb{L}\subset \cup_{Q\in\mathcal{S}_k} E_{Q}$
		$E_Q:=\{x\in3Q: \max\{ \frac{M_{L(\log L)}f_1 {\chi}_{3Q}(x)}{\|f_1\|_{L(\log L),\,3Q}} , \frac{M_{p}{(f_2 {\chi}_{3Q})}{(x)}}{{\langle f_2\rangle}_{r,3Q}}\}>C_d\}$ for some
		chosen constant $C_d$.
		For all $Q\in \mathcal{S}_0$,  one gets
		\begin{align*}
			|\Lambda^{A}(f_1,\,f_2)|=	| \Lambda_{Q_0}^A(f_1{\chi}_{Q_0},f_2{\chi}_{3Q_0})|
			\leq  r'|Q_0| \|f_1\|_{\mathcal{Y}_{L\log L}(\mathcal{Q }_0)} \|f_2\|_{\mathcal{Y}_{r}(\mathcal{Q}_0)}  +\sum\limits_{\mathbb{L}\in \mathcal{S}_{1}({Q_0}),\mathbb{L}\subset Q_0}|\Lambda_\mathbb{L}^A(f_1,f_2)|.
		\end{align*}
		For all $k>0$ and $Q\in \mathcal{S}_k$, we have
		\begin{align*}
			|\Lambda^{A}(f_1,\,f_2)|
			\leq  r'|Q| \|f_1\|_{\mathcal{Y}_{L\log L}(\mathcal{Q})} \|f_2\|_{\mathcal{Y}_{r}(\mathcal{Q})}  +\sum\limits_{\mathbb{L}\in \mathcal{S}_{k+1}({Q}),\mathbb{L}\subset Q}|\Lambda_\mathbb{L}^A(f_1,f_2)|.
		\end{align*}
		When $\inf\{s_Q:Q\in \mathcal{S}_k\}<s_{Q_0}$,  the iteration stops and there holds that
		$$|\Lambda^{A}(f_1,\,f_2)|\lesssim r'\sum_{Q\in {\cup_{\kappa=0}^k \mathcal{S}_\kappa}}\langle |f|\rangle_{L\log L,\,3Q}\langle |g|\rangle_{3Q,\,r}|Q|.
		$$
		This finishes the proof of \eqref{wmmmm}.		
	\end{proof}
	
	\section{The $(L^r,\,L^1)$  sparse domination of $T_{\Omega,\,A}$}\label{ssp2}
	In this section, we aim to establish the $(L^r,\,L^1)$-bilinear sparse domination of $T_{\Omega,\,A}$ which follows essentially from the properties of $\widetilde{T}_{\Omega,\,A}$.
		\begin{theorem}\label{thm3.2}
		Under the hypothesis of Theorem \ref{thm1.2}, for any $r\in (1,\,2]$, the operator $T_{\Omega,\,A}$ enjoys the $(L^r,\,L^1)$ bilinear
		sparse domination with a bound $Cr'^2$.
	\end{theorem}
	To show this Theorem is true, we first prepare some lemmas which will be used later.
	Let $\Omega$ be a function of homogeneous of degree zero and $\Omega\in L^{\infty}(\mathbb{S}^{d-1})$. Let $\mathfrak{E}^j$ ($j\in\mathbb{N}$), $P_{s-j\kappa}$, $\Lambda_k$,\,$\widetilde{\Lambda}_k, \Gamma_{\nu}^j$ be the same operators as in Section 3. For $e_\nu^j\in\mathfrak{E}^j$, set
	$$K_{s\nu}^{n,j}(x)=\frac{\Omega(x)x_n}{|x|^{d+1}}\phi_s(x)\Gamma_{\nu}^j(x').$$
	Denote by $T_{s\nu}^{n,j}$ the convolution operator with the kernel $K_{s\nu}^{n,j}$, and $\widetilde{T}_{s\nu}^{n,j}$ be the dual operator of $T_{s\nu}^{n,j}$. The following lemmas are needed in our analysis.
	\begin{lemma}\label{lem7.1}
		Let $a\in {\rm BMO}(\mathbb{R}^d)$, $b_\mathbb{L}$ be an integrable function with integral zero, supported on a cube $\mathbb{L}\in \mathcal{S}_{s-j}$. Then for each $q\in (1,\,\infty)$, it holds that
		$$\big\|(a-\langle a\rangle_\mathbb{L})P_{s-j\kappa}T_{s\nu}^{n,j}b_\mathbb{L}\big\|_{L^1(\mathbb{R}^d)}\lesssim_{q} j2^{j\kappa d/q'}2^{-(1-\kappa)j}\|\Omega\|_{L^{\infty}(\mathbb{S}^{d-1})}\|b_\mathbb{L}\|_{L^1(\mathbb{R}^d)}.
		$$
	\end{lemma}
	\begin{proof}
	By the vanishing condition of $b_\mathbb{L}$, the same reasoning as in the proof of Lemma \ref{ljj} may lead to
		\begin{eqnarray*}\aligned
		{}&|P_{s-j\kappa} {T}_{s\nu}^{n,j}b_\mathbb{L}(x)|\\&
			\leq \int_{\mathbb{S}^{d-1}}\frac{|\Omega(\theta)|}{|\mathbb{L}|}\int_{\mathbb{L}}\Big|\int_{\mathbb{R}^d}\int_0^{\infty} \Big(\omega_{s-j\kappa}(x-y-r\theta)-\omega_{s-j\kappa}(x-y'-r\theta)\Big) \phi_s(r)r^{d-1}\frac{dr}{r} b_\mathbb{L}(y)dy\Big|d{y'}d\sigma_{\theta}   \\
			&\leq 2^{-(s-j\kappa)d}2^{-(1-\kappa)j}\int_{\mathbb{S}^{d-1}}\big| \Omega(\theta)\big|\frac{1}{|\mathbb{L}|}\int_{\mathbb{R}^d}\int_{\mathbb{L}}\int_0^{\infty}
		\int^1_0\big|\nabla\omega(2^{-(s-j\kappa)}\big(x-(ty+(1-t)y')-r\theta\big)\big|dt\\	&\quad \times  \phi_s(r)\frac{dr}{r}|b_\mathbb{L}(y)|dyd{y'}d\sigma_{\theta}.\endaligned
		\end{eqnarray*}
		For $a\in {\rm BMO}(\mathbb{R}^d)$, $t\in[0,\,1]$, $y,\,y'\in \mathbb{L}$, it follows from  H\"older's inequality and John-Nrenberg inequality that
		\begin{eqnarray*}
			&&\big\|\nabla\omega(2^{-(s-j\kappa)}\big(\cdot-(ty+(1-t)y')-r\theta\big)(a-\langle a\rangle_\mathbb{L})\big\|_{L^1(\mathbb{R}^d)}\\
			&&\leq \big\|\nabla\omega(2^{-(s-j\kappa)}\big(\cdot-(ty+(1-t)y')-r\theta\big)\|_{L^q(\mathbb{R}^d)}\|(a-\langle a\rangle_\mathbb{L})\chi_{\mathbb{L}_{j,3}}\big\|_{L^{q'}(\mathbb{R}^d)}\\
			&&\leq j2^{(s-j\kappa)d/q}2^{sd/q'}.
		\end{eqnarray*}
		Let $h\in L^{\infty}(\mathbb{R}^d)$. An arguement involving duality then yields
		\begin{eqnarray*}
			\Big|\int_{\mathbb{R}^d}\big(a(x)-\langle a\rangle_\mathbb{L}\big)P_{s-j\kappa}T_{s\nu}^{n,j}b_\mathbb{L}(x)h(x)dx\Big|
			&\leq &\Big|\int_{\mathbb{R}^d}P_{s-j\kappa}\widetilde{T}_{s\nu}^{n,j}\big((a-\langle a\rangle_\mathbb{L})h\big)(x)b_\mathbb{L}(x)dx\Big|\\
			&\lesssim &j2^{j\kappa d/q'}2^{-(1-\kappa)j}\|\Omega\|_{L^{\infty}(\mathbb{S}^{d-1})}\|b_\mathbb{L}\|_{L^1(\mathbb{R}^d)}\|h \|_{L^{\infty}(\mathbb{R}^d)}.
		\end{eqnarray*}
		This completes the proof of Lemma \ref{lem7.1}.
	\end{proof}
	\begin{lemma}\label{lem7.2}
		Let $\Omega$ be homogeneous of degree zero and $\Omega\in L^{\infty}(\mathbb{S}^{d-1})$, $\mathcal{S}$ be a collection of dyadic   cubes with disjoint interiors. For
		each cube $Q\in\mathcal{S}$, let $b_Q$ be an integrable function supported in $Q$ satisfying $\|b_Q\|_{L^1(\mathbb{R}^d)}\leq |Q|$.
		Then  for $a\in{\rm BMO}(\mathbb{R}^d)$ and $j\geq 100$,
		we have		$$\Big\|\sum_{s}\sum_{Q\in\mathcal{S}_{s-j}}\sum_{\nu}G_{\nu}^j\big((a-\langle a\rangle_Q)(I-P_{s-j\kappa})T_{s\nu}^{n,j}b_Q\big)\Big\|^2_{L^2(\mathbb{R}^d)}\lesssim \|\Omega\|_{L^{\infty}(\mathbb{S}^{d-1})}^22^{-j\gamma/2}\sum_{Q\in\mathcal{S}}\|b_{Q}\|_{L^1(\mathbb{R}^d)}.$$
	\end{lemma}
	\begin{proof}Let $H_{s,\nu}^j$ be the same as in Section 3. Hu and Tao \cite{hutao} proved that
		\begin{eqnarray}\label{equation7.x}
		&&\Big\|\sum_{s}\sum_{Q\in\mathcal{S}_{s-j}}(a-\langle a\rangle_Q)T_{s\nu}^{n,j}b_Q\Big\|^2_{L^2(\mathbb{R}^d)}\lesssim  2^{-j\gamma(2d-\frac{5}{2})}\|\Omega\|^2_{L^{\infty}(\mathbb{S}^{d-1})}\sum_{Q\in\mathcal{S}}\|b_Q\|_{L^1(\mathbb{R}^d)}.
		\end{eqnarray}
		Repeating the proof of (\ref{equation7.x}) as in  \cite{hutao}, we can verify that this inequality still holds if $T_{s\nu}^j$ is replaced by $H_{s,\nu}^j$.
		\begin{eqnarray*}
			&&\Big\|\sum_{s}\sum_{Q\in\mathcal{S}_{s-j}}(a-\langle a\rangle_Q)H_{s,\nu}^j*b_Q\Big\|^2_{L^2(\mathbb{R}^d)}\lesssim  2^{-j\gamma(2d-\frac{5}{2})}\|\Omega\|^2_{L^{\infty}(\mathbb{S}^{d-1})}\sum_{Q\in\mathcal{S}}\|b_Q\|_{L^1(\mathbb{R}^d)}.
		\end{eqnarray*}
		This, along with  the facts that
		$$|(I-P_{s-j\kappa})T_{s\nu}^{n,j}b_Q(x)|\lesssim\|\Omega\|_{L^{\infty}(\mathbb{S}^{d-1})}H_{s,\,\nu}^j*|b_Q|(x),
		$$
		and
		\begin{eqnarray*}
			&&\sum_{\nu}\Big\|\sum_{s}\sum_{Q\in\mathcal{S}_{s-j}}G_{\nu}^j(a-\langle a\rangle_Q)(I-P_{s-j\kappa})T_{s\nu}^{n,j}b_Q\Big\|^2_{L^2(\mathbb{R}^d)}\\&&\lesssim 2^{j\gamma(d-2)} 2^{j\gamma(d-1)} \Big\|\sum_{s}\sum_{Q\in\mathcal{S}_{s-j}}(I-P_{s-j\kappa})T_{s\nu}^{n,j}b_Q\Big\|^2_{L^2(\mathbb{R}^d)}
		\end{eqnarray*}
		yields the desired conclusion.
	\end{proof}
	\begin{lemma}[\cite{hutao}, Lemma 5.2]\label{lem7.3}Let $j\in\mathbb{N}$ and $e_{\nu}^j\in \mathfrak{E}^j$. For $a\in{\rm BMO}(\mathbb{R}^d)$, let $G_{\nu, a}^j$ be the commutator of $G_{\nu}^j$ with symbol $a$. Then
		\begin{eqnarray*}\Big\|\Big(\sum_{\nu}|G_{\nu,\,a}^jf|^2\Big)^{1/2}\Big\|^2_{L^2(\mathbb{R}^d)}\lesssim 2^{j\gamma(d-1)-j\gamma/2}\|f\|_{L^2(\mathbb{R}^d)}^2,
		\end{eqnarray*}
	\end{lemma}
		Let
	$K_s^*(x,\,y)=\frac{\Omega(x-y)}{|x-y|^{d+1}}(A(x)-A(y)-\nabla A(x)(x-y))\phi_{s}(x-y).$
Let $\Lambda ^{A,*}$,  $\Lambda^{A,s_Q,*}_Q$, $\Lambda_Q^{A,*}$  and $\Lambda_{\mathcal{Q}}^{A,*}$ the same as in the definitions of $\Lambda ^{A}$, $\Lambda^{A,s_Q}_Q$, $\Lambda_Q^{A}$ and $\Lambda_{\mathcal{Q}}^{A}$ respectively, with $K_s$ replaced by $K_s^*$. Then
	\begin{eqnarray}\label{equation3.6}
		\frac{\big|\Lambda ^{A,*}(f_1,\,f_2)\big|}{\| f_1\|_{L^{r'}(\mathbb{R}^d)} \| f_2\|_{ L^{r}(\mathbb{R}^d)}}\lesssim r'^2  ,\,\,\,r\in (1,\,2].\label{CT}
	\end{eqnarray}
	Furthermore, for a  stopping collection $\mathcal{Q}$ with top $Q$, we have
	\begin{eqnarray*}
		|\Lambda_{\mathcal{\mathcal{Q}}}^{A,*}(f_1,f_2)|\lesssim   r'^2|Q|
		\| f_1\|_{\mathcal{Y}_{r'}(\mathcal{Q})}\| f_2\|_{ \mathcal{Y}_{r}(\mathcal{Q})}.
	\end{eqnarray*}
	
	%The next lemma is  a variation of
	\begin{lemma}\label{lem3.5}
		Let $\Omega$ be homogeneous of degree zero and $\Omega\in L^{\infty}(\mathbb{S}^{d-1})$. Then for any $u\in (1,\,\infty)$ and  $j\geq 1$, it holds that
		$$	\sum_s\sum_{\mathbb{L}\in\mathcal{Q},s_\mathbb{L}=s-j} \int_{\mathbb{R}^d}\int_{\mathbb{R}^d}|K_s^*(x,\,y)||b_\mathbb{L}(y)||h(x)|dydx\lesssim ju' \|\Omega\|_{L^{\infty}(\mathbb{S}^{d-1})}|Q|\|b\|_{\dot{\mathcal{X}}_1(\mathcal{Q})}\|h\|_{\mathcal{Y}_u(\mathcal{Q})}.$$
	\end{lemma}
	\begin{proof}
		The arguement here follows from the proof of \cite[Lemma 3.1]{ccdo}. Let $\mathbb{L}\in\mathcal{Q}$ and  $s_\mathbb{L}=s-j$. We assume that $\|\Omega\|_{L^{\infty}(\mathbb{S}^{d-1})}=1$.  Lemma \ref{lem2.3}, together with  H\"older's inequality, yields that
		\begin{eqnarray*}
			&&\int_{\mathbb{R}^d}\int_{\mathbb{R}^d}|K^*_s(x,\,y)||b_\mathbb{L}(y)||h(x)|dydx\\
			&&\quad\lesssim\int_{\mathbb{R}^d}\int_{\mathbb{R}^d}
			\frac{|A_\mathbb{L}(x)-A_\mathbb{L}(y)|}{|x-y|^{d+1}}\phi_{s}(x-y)|b_\mathbb{L}(y)||h(x)|dx\\
			&&\qquad+\sum_{n=1}^d2^{-sd}\int_{\mathbb{R}^d}\int_{\mathbb{R}^d}|b_\mathbb{L}(y)||\partial_n A(x)-\langle \partial_n A\rangle_{\mathbb{L}}||h(x)|\chi_{{\mathbb{L}^s}}(x)\,dy\,dx\\
			&&\quad\lesssim j\|b_\mathbb{L}\|_{L^1(\mathbb{R}^d)}\inf_{x\in \mathbb{L}^*}Mh(x)+\sum_{n=1}^d\Big(\frac{1}{|\mathbb{L}^s|}\int_{\mathbb{L}^s}|\partial_n A(x)-\langle \nabla A\rangle_{\mathbb{L}^s}|^{u'}dx\Big)^{\frac{1}{u'}}\\
			&&\qquad \times\inf_{x\in \mathbb{L}^*}M_uh(x)\|b_\mathbb{L}\|_{L^1(\mathbb{R}^d)}\\
			&&\quad\lesssim ju'|\mathbb{L}|\|b\|_{\dot{\mathcal{X}_1}(\mathcal{Q})}\|h\|_{\mathcal{Y}_u(\mathcal{Q})},
		\end{eqnarray*}
		where $\mathbb{L}^s=B(y,2^{s+10})$, and here we have used the fact that
		$$\Big(\frac{1}{|\mathbb{L}^s|}\int_{\mathbb{L}^s}|\partial_n A(x)-\langle\partial_n A\rangle_{\mathbb{L}^s}|^{u'}dx\Big)^{\frac{1}{u'}}\lesssim u',
		$$
		see \cite[p.128]{gra}. This completes the proof of Lemma \ref{lem3.5}.
	\end{proof}
	The following lemma gives a decay estimate for $\Lambda_{\mathcal{Q}}^{A,\,*}(b,\,h)$.
	\begin{lemma}\label{dui}
		Let $r\in (1,\,2]$,  {$\mathcal{Q}\subset\mathcal{D}$}  be a stopping collection with top cube $Q$, $b$ and $h$ be functions such that ${\rm supp}\,b\subset Q$ and ${\rm supp}\,h\subset 3Q$. Under the hypothesis of Theorem \ref{thm1.2}, the following estimate holds
		\begin{eqnarray*}
			\big|\Lambda_{\mathcal{Q}}^{A,\,*}(b,\,h)\big|\lesssim r'^2\|\Omega\|_{L^{\infty}(\mathbb{S}^{d-1})}|Q|\|b\|_{\dot{\mathcal{X}}_1(\mathcal{Q})}\|h\|_{\mathcal{Y}_r(\mathcal{Q})}.
		\end{eqnarray*}
	\end{lemma}
	\begin{proof}Without loss of generality, we may assume that $\|\Omega\|_{L^{\infty}(\mathbb{S}^{d-1})}=\|b\|_{\dot{\mathcal{X}}_1(\mathcal{Q})}=1$. Let $\mathcal{Q}\subset \mathcal{D}$ be a stopping collection with top $Q$, $\{\epsilon_j\}\in \{-1,\,0,1\}^{\mathbb{Z}}$ be a choices of signs, $b\in \dot{\mathcal{X}}_1(\mathcal{Q})$. We define
		\begin{align*}
			K^*(b,\,h)&=\sum_{j\geq 1}\sum_s\sum_{\mathbb{L}\in\mathcal{Q},s_\mathbb{L}=s-j}\epsilon_s
			\int_{\mathbb{R}^{2d}}K_s^*(x,y)b_\mathbb{L}(y)h(x)\,dy\,dx:=\sum_{j\geq 1}	K_j^*(b,\,h).
		\end{align*}
		Obviously, it suffices to show the following inequality holds
		\begin{eqnarray*}
			|K^*(b,\,h)|\lesssim r'^2\|\Omega\|_{L^{\infty}(\mathbb{S}^{d-1})}|Q|\|b\|_{\dot{\mathcal{X}}_1(\mathcal{Q})}\|h\|_{\mathcal{Y}_r(\mathcal{Q})}.
		\end{eqnarray*}
		
		To this purpose, let
		$$
		G_j^{1,*}(b,\,h)=\sum_s\sum_{\mathbb{L}\in\mathcal{Q},s_\mathbb{L}=s-j}\epsilon_s
		\int_{\mathbb{R}^{2d}}\frac{(A_\mathbb{L}(x)-A_\mathbb{L}(y))}{|x-y|^{d+1}}\phi_{s}(x-y)b_\mathbb{L}(y)h(x)\,dy\,dx;
		$$
		
		$$G_j^{2,*}(b,\,h)=\sum_{n=1}^d\sum_{\nu}\sum_s\sum_{\mathbb{L}\in\mathcal{Q},s_\mathbb{L}=s-j}\epsilon_s\int_{\mathbb{R}^d}
		\Big((\partial_nA-\langle \partial_n A\rangle_\mathbb{L})P_{s-j\kappa}T_{s\nu}^{n,j}b_\mathbb{L}\Big)(x)h(x)dx;
		$$
		$$G_j^{3,*}(b,\,h)=\sum_{n=1}^d\sum_{\nu}\sum_s\sum_{\mathbb{L}\in\mathcal{Q},s_\mathbb{L}=s-j}\epsilon_s\int_{\mathbb{R}^d}
		G_{\nu}^j\Big((\partial_nA-\langle \partial_n A\rangle_\mathbb{L})(I-P_{s-j\kappa})T_{s\nu}^{n,j}b_\mathbb{L}\Big)(x)h(x)dx;
		$$
		$$G_j^{4,*}(b,\,h)=\sum_{n=1}^d \sum_{\nu}\sum_s\sum_{\mathbb{L}\in\mathcal{Q},s_\mathbb{L}=s-j}\epsilon_s\int_{\mathbb{R}^d}
		(I-G_{\nu}^j)\Big((\partial_nA-\langle \partial_n A\rangle_\mathbb{L})(I-P_{s-j\kappa})T_{s\nu}^{n,j}b_\mathbb{L}\Big)(x)h(x)dx.
		$$
		Then
		$$|K_j^*(b,\,h)|\leq \sum_{l=1}^4|G_j^{l,*}(b,\,h)|.$$

		We now estimate these four terms. It is easy to see that $ G_j^{1,*}(b,\,h)$ can be dealt in the same way as {in the proof of Lemma} \ref{e}. Hence there exists some $\varepsilon\in(0,1) $ such that
		\begin{eqnarray}\label{equation7.4}
			|G_j^{1,*}(b,\,h)| \lesssim   2^{-\varepsilon j/2} |Q|\|b\|_{\dot{\mathcal{X}}_{1}(\mathcal{Q})}\|{h}\|_{\mathcal{Y}_{\infty}(\mathcal{Q})}.\label{de}
		\end{eqnarray}
	Moreover, Lemma \ref{lem7.1} tells us that
		\begin{eqnarray}\label{equation7.5}
			|G_j^{2,*}(b,\,h)| \lesssim   2^{j\kappa d/q'}2^{-(1-\kappa)j} |Q|\|b\|_{\dot{\mathcal{X}}_{1}(\mathcal{Q})}\|{h}\|_{\mathcal{Y}_{\infty}(\mathcal{Q})}.\label{de}
		\end{eqnarray}		
		On the other hand, we obtain from Lemma \ref{lem7.2} that
		\begin{eqnarray}\label{equation7.5x}
			G_{j}^{3,*}(b,\,h)&\leq &\Big\|\sum_s\sum_{\mathbb{L}\in\mathcal{Q},s_\mathbb{L}=s-j}\sum_{\nu}{\epsilon_s		\big(\,G_{\nu}^j}\big(\partial_nA-\langle \partial_n A\rangle_\mathbb{L} )(I-P_{s-j\kappa})T_{j\nu}^{n,s}b_\mathbb{L}\big)\Big\|_{L^2(\mathbb{R}^d)}\|h\|_{L^2(\mathbb{R}^d)}\nonumber
			\\
			&&\lesssim 2^{-j\gamma/4}\Big(\sum_{\mathbb{L}\in\mathcal{Q}}\|b_{\mathbb{L}}\|_{L^1(\mathbb{R}^d)}\Big)^{1/2}\|h\|_{L^2(\mathbb{R}^d)}
			\lesssim 2^{-j\gamma/4}|Q|\|h\|_{\mathcal{Y}_{\infty}(\mathcal{Q})}.
		\end{eqnarray}

	Consider now $G_{j}^{4,\,*}$, we write
		\begin{eqnarray*}G_j^{4,*}(b,\,h)&=&\sum_{n=1}^d \sum_{\nu}\sum_s\sum_{\mathbb{L}\in\mathcal{Q},s_\mathbb{L}=s-j}\epsilon_s\int_{\mathbb{R}^d}
			G_{\nu,\partial_n A}^j(I-P_{s-j\kappa})T_{s\nu}^{n,j}b_\mathbb{L}(x)h(x)dx\\
			&&+\sum_{n=1}^d \sum_{\nu}\sum_s\sum_{\mathbb{L}\in\mathcal{Q},s_\mathbb{L}=s-j}\epsilon_s\int_{\mathbb{R}^d}(\partial_nA(x)-\langle \partial_n A\rangle_\mathbb{L})\Big((I-G_{\nu}^j)(I-P_{s-j\kappa})T_{s\nu}^{n,j}b_\mathbb{L}\Big)(x)\\
			&&\times h(x)dx:=\sum_{n=1}^d{\rm U}_{1,n}+\sum_{n=1}^d{\rm U}_{2,n}.
		\end{eqnarray*}
		Applying Lemma \ref{lem7.3}, we obtain that for each $n$ with $1\leq n\leq d$,
		\begin{eqnarray}\label{equation7.6x}
			|{\rm U}_{1,\,n}|&\le &\Big\|\Big(\sum_{\nu}|G_{\nu,\partial_n A}^jh|^2\Big)^{\frac{1}{2}}\Big\|_{L^2}
			\Big\|\Big(\sum_{\nu}\Big|
			\sum_s\sum_{\mathbb{L}\in\mathcal{Q},s_\mathbb{L}=s-j}\epsilon_s(I-P_{s-j\kappa})T_{s\nu}^{n,j}b_\mathbb{L}\Big|^2
			\Big)^{\frac{1}{2}}\Big\|_{L^2}\\
			&\lesssim&2^{j\gamma(d-1)/2-j\gamma/4}2^{-j\gamma(d-1)/2}|Q|\|h\|_{\mathcal{Y}_{\infty}(\mathcal{Q})},\nonumber
		\end{eqnarray}
		where in the last inequality, we have invoked the fact that
		$$\Big\|\sum_s\sum_{\mathbb{L}\in\mathcal{Q},s_\mathbb{L}=s-j}\epsilon_s(I-P_{s-j\kappa})T_{s\nu}^{n,j}b_\mathbb{L}
		\Big\|^2_{L^2(\mathbb{R}^d)}\lesssim2^{-2j\gamma(d-1)}\sum_{\mathbb{L}\in\mathcal{Q}}\|b_\mathbb{L}\|_{L^1(\mathbb{R}^d)}.
		$$
		To estimate ${\rm U}_{2,n}$, for $k\in\mathbb{Z}$, let
		\begin{eqnarray*}
			{\rm U}_{2,n}^k&=&  \sum_{\nu}\sum_s\sum_{\mathbb{L}\in\mathcal{Q},s_\mathbb{L}=s-j}\epsilon_s\int_{\mathbb{R}^d}
			(\partial_nA(x)-\langle \partial_n A\rangle_\mathbb{L})\Big(\Lambda_k\widetilde{\Lambda}_k(I-G_{\nu}^j)(I-P_{s-j\kappa})T_{s\nu}^{n,j}b_\mathbb{L}\Big)(x)h(x)dx.
		\end{eqnarray*}
		Let $\widetilde{T}^{n,j}_{s\nu}$ be the adjoint of $ {T}^{n,j}_{s\nu}$. Observe that
		$$\Lambda_k(I-G_{\nu}^j)(I-P_{s-j\kappa})\widetilde{\Lambda}_k\widetilde{T}_{s\nu}^{n,j}=\widetilde{T}_{s\nu}^{n,j}
		\Lambda_k\widetilde{\Lambda}_k(I-G_{\nu}^j)(I-P_{s-j\kappa}).$$
		It then follows that
		\begin{eqnarray*}
			|{\rm U}_{2,n}^k|&\leq&\sum_{\nu}\sum_s\sum_{\mathbb{L}\in\mathcal{Q},s_\mathbb{L}=s-j}\|b_\mathbb{L}\|_{L^1(\mathbb{R}^d)}
			\big\|\Lambda_k(I-G_{\nu}^j)(I-P_{s-j\kappa})\widetilde{\Lambda}_k\widetilde{T}_{s\nu}^{n,j}
			\big(\partial_nA-\langle \partial_n A\rangle_\mathbb{L})h\big)\big\|_{L^{\infty}(\mathbb{R}^d)}\nonumber\\
			&\leq&\sum_{\nu}\sum_s\sum_{\mathbb{L}\in\mathcal{Q},s_\mathbb{L}=s-j}\|b_\mathbb{L}\|_{L^1(\mathbb{R}^d)}\|\widetilde{T}_{s\nu}^{n,j}
			\big(\partial_nA-\langle \partial_n A\rangle_\mathbb{L})h\big)\big\|_{L^{\infty}(\mathbb{R}^d)}\nonumber\\
			&&\times\|(I-P_{s-j\kappa})\widetilde{\Lambda}_k\|_{L^1(\mathbb{R}^d)\rightarrow L^1(\mathbb{R}^d)}.\nonumber
		\end{eqnarray*}
For each fixed $x\in\mathbb{R}^d$, a trivial computation involving H\"older's inequality and the John-Nirenberg inequality yields that 
		\begin{eqnarray*}
			\big|\widetilde{T}_{s\nu}^{n,j}
			\big(\partial_nA-\langle \partial_n A\rangle_\mathbb{L})h\big)(x)\big|
			&\lesssim&\|K^{n,j}_{s\nu}\|_{L^q(\mathbb{R}^d)}\|\partial_nA(y)-\langle \partial_n A\rangle_\mathbb{L}\|_{L^{q'}(\mathbb{L}^{s})}\|h\|_{L^{\infty}(\mathbb{R}^d)}\\
			&\lesssim& j2^{-\gamma j(d-1)/q}\|h\|_{L^{\infty}(\mathbb{R}^d)}.
		\end{eqnarray*}
		This inequality, together with the inequality (\ref{equation2.wks}), gives that
		\begin{eqnarray}\label{equation7.6}
			|{\rm U}_{2,n}^k|&\lesssim&\sum_{\nu}\sum_s\sum_{\mathbb{L}\in\mathcal{Q},s_\mathbb{L}=s-j}
			j2^{-\gamma j(d-1)/q}2^{s-j\kappa-k}\|b_\mathbb{L}\|_{L^1(\mathbb{R}^d)}\|h\|_{L^{\infty}(\mathbb{R}^d)}\\
			&\leq&j\sum_s\sum_{\mathbb{L}\in\mathcal{Q},s_\mathbb{L}=s-j}
			2^{\gamma j(d-1)/q'}2^{s-j\kappa-k}\|b_\mathbb{L}\|_{L^1(\mathbb{R}^d)}\|h\|_{L^{\infty}(\mathbb{R}^d)} .\nonumber
		\end{eqnarray}
{For $N=[d/2]+1$, let
	$$	\aligned |\widetilde{D}_{k,s\nu}^j(y)|&=& 2^{j\gamma(N_1+2N)}2^{(-s+k)N_1}2^{-s}2^{-kd}\|\Omega\|_{L^{\infty}(\mathbb{S}^{d-1})} \int_{S^{d-1}}|\Gamma_{\nu}^j(\theta)|\int_{2^{s-1}}^{2^{s+1}}\frac{1}{(1+2^{-2k}|y-r\theta|^2)^N}dr d\theta.\endaligned$$
	It was proved in \cite[Lemma 5.3]{hutao} that \begin{eqnarray*}
		|(I-G_{\nu}^j)\Lambda_kT_{s\nu}^{n,j}b_\mathbb{L}(x)|\lesssim_{N_1} \int_{\mathbb{R}^d}|\widetilde{D}_{k,s,\nu}^j(x-y)||b_\mathbb{L}(y)|dy.
	\end{eqnarray*}
	}
		Let $V_{k,s,\nu}^j$ be the integral operator with kernel  $|\widetilde{D}_{k,s,\nu}|$,  $W^j_{k,s}$ be the kernel of $(I-P_{s-j\kappa})\widetilde{\Lambda}_k$, and $X^j_{k,s}$ be the convolution with kernel of $|W^j_{k,s}|$. It then follows that	
		\begin{eqnarray*}
			|{\rm U}_{2,n}^k|&\leq& \sum_{\nu}\sum_s\sum_{\mathbb{L}\in\mathcal{Q},s_\mathbb{L}=s-j}\int_{\mathbb{R}^d}|\partial_nA(x)-\langle \partial_n A\rangle_\mathbb{L}|X_{k,s}^jV_{k,s,\nu}^j(|b_\mathbb{L}|)(x)|h(x)|dx\\
			&=&\sum_{n=1}^d \sum_{\nu}\sum_s\sum_{\mathbb{L}\in\mathcal{Q},s_\mathbb{L}=s-j}\int_{\mathbb{R}^d}X_{k,s}^jV_{k,s,\nu}^j\big(|h||\partial_nA-\langle \partial_n A\rangle_\mathbb{L}|)(x)|b_\mathbb{L}(x)|dx.
		\end{eqnarray*}
	For each $\nu$, $s$, $k$ and $\mathbb{L}\in\mathcal{Q}$ with $s_\mathbb{L}=s-j$, a trivial computation leads to
		\begin{eqnarray*}
			\int_{\mathbb{R}^d}\frac{|b(y)-\langle b\rangle_Q|}{(1+2^{-2k}|x-y-r\theta|^2)^N}dy
			\lesssim 2^{kd}(|k-s|+j),\nonumber
		\end{eqnarray*}
		see \cite[Section 5]{hutao} for details. Therefore,
		\begin{eqnarray*}
			\|X_{k,s}^jV_{k,s,\nu}^j\big(|h||\partial_nA-\langle \partial_n A\rangle_\mathbb{L}|)\|_{L^{\infty}(\mathbb{R}^d)}&&\leq \|X_{k,s}^j\|_{L^{\infty}(\mathbb{R}^d)\rightarrow L^{\infty}(\mathbb{R}^d)}\|
			V_{k,s,\nu}^j\big(|h||\partial_nA-\langle \partial_n A\rangle_\mathbb{L}|)\|_{L^{\infty}(\mathbb{R}^d)}\\
			&&\quad\lesssim (|k-s|+j)2^{-s\gamma(d-1)}2^{j\gamma(N_1+2N)}2^{(-s+k)N_1} \|h\|_{L^{\infty}(\mathbb{R}^d)},
		\end{eqnarray*}
		which, in turn, implies that
		\begin{eqnarray}\label{equation7.7}
			|{\rm U}_{2,n}^k|&\lesssim &\sum_{n=1}^d \sum_s\sum_{\mathbb{L}\in\mathcal{Q},s_\mathbb{L}=s-j} (|k-s|+j)2^{j\gamma(N_1+2N)}
			2^{(-s+k)N_1} \|b_\mathbb{L}\|_{L^1(\mathbb{R}^d)}\|h\|_{L^{\infty}(\mathbb{R}^d)}.
		\end{eqnarray}
		Let $m=s-[j\kappa/2]$. Combining the estimates (\ref{equation7.6}) and (\ref{equation7.7}), it yields that
		\begin{eqnarray}\label{equation7.8}
			|{\rm U_{2,n}}|&\leq &\sum_{n=1}^d\Big(\sum_{k\geq m}|{\rm U}_{2,n}^k|+\sum_{k< m}|{\rm U}_{2,n}^k|\Big)\\
			&\lesssim &j2^{j(-\kappa/2+\gamma(d-1)/q)}+j2^{j(\gamma(N_1+2N)-\kappa N_1/3)}|Q|\|h\|_{\mathcal{Y}_{\infty}(\mathcal{Q})}.\nonumber
		\end{eqnarray}
		
		Then choose $N_1=d+1$, $q\in (1,\,\infty)$ close to $1$ sufficiently, $\gamma\in (0,\,1)$ and close to $0$ sufficiently, such that
		$-\kappa/2+\gamma(d-1)/q)<0,\,\gamma(N_1+2N)-\kappa N_1/3<0.$
		It follows from (\ref{equation7.4})--(\ref{equation7.6x}), (\ref{equation7.8}) that
		\begin{eqnarray}\label{zui}
			|K_j^*(b,\,h)|\lesssim 2^{-\varsigma j}|Q|\|\Omega\|_{L^{\infty}(\mathbb{S}^{d-1})}\|b\|_{\dot{\mathcal{X}}_1(\mathcal{Q})}\|h\|_{\mathcal{Y}_{\infty}(\mathcal{Q})},  \quad \varsigma\in (0,\,1)
		\end{eqnarray}
		
		Interpolating between Lemma \ref{lem3.5} (with  $u=\frac{1+r}{2}$)  and \eqref{zui} yields that
		\begin{eqnarray*}
			|K^*(b,\,h)|&\lesssim &
			\sum\limits_{j\geq 1} (ju')^{\frac{1+r}{2r}}(2^{-\varsigma j})^{\frac{r-1}{2r}}|Q|\|\Omega\|_{L^{\infty}(\mathbb{S}^{d-1})}\|b\|_{\dot{\mathcal{X}}_1(\mathcal{Q})}\|h\|_{\mathcal{Y}_{r}(\mathcal{Q})}\\
			&\lesssim & r'^2|Q|\|\Omega\|_{L^{\infty}(\mathbb{S}^{d-1})}\|b\|_{\dot{\mathcal{X}}_1(\mathcal{Q})}\|h\|_{\mathcal{Y}_{r}(\mathcal{Q})}.
		\end{eqnarray*}
		This finishes the proof of Lemma \ref{dui}.
	\end{proof}
	
	Now, we are in a position to prove the second bilinear sparse domination of $T_{\Omega,\,A}$.

	\begin{proof}[Proof of Theorem \ref{thm3.2}]
		Let	$\widetilde{T}_{\Omega,A}$ be defined as \eqref{equation1.dualoperator}. By duality, for $r\in (1,\,2]$, it suffices to prove
		$$\Big|\int_{\mathbb{R}^d}f_1(x)\widetilde{T}_{\Omega,A}f_2(x)dx\Big|\lesssim r'^2\sup_{\mathcal{S}}\mathcal{A}_{\mathcal{S},\,L^1,\,L^r}(f_1,\,f_2),
		$$
		holds for all bounded functions $f_1$ and $f_2$ with compact supports.
		
		Let $\mathbb{L}\subset \mathbb{R}^d$ be a dyadic cube and $\mathcal{Q}$ be a stopping collection with top $Q$.  We claim that
		\begin{equation}\label{equation3.8}
			|\Lambda^{A,s_Q,*}_Q(f_1\chi_{Q},f_2\chi_{3Q})|\le Cr'^2|Q|\|f_1\|_{\mathcal{Y}_1(\mathcal{Q})}\|f_2\|_{\mathcal{Y}_r(\mathcal{Q})}
			+\sum_{\mathbb{L}\in\mathcal{Q},\,\mathbb{L}\subset Q}|\Lambda^{A,s_\mathbb{L},*}_\mathbb{L}(f_1\chi_{\mathbb{L}},f_2\chi_{3\mathbb{L}})|.
		\end{equation}
		In fact, for fixed $r\in (1,\,2]$, $f_1$, $f_2$  with ${\rm supp}\,f_1\subset Q$, ${\rm supp}\, f_2\subset 3Q$, apply the Calder\'on-Zygmund decomposition on $f_j$ with respect to the family $\mathcal{Q}$ as described in Definition \ref{defn3.2}, and write
		$$f_j=g_j+b_j,\,\, b_j=\sum_{\mathbb{L}\in \mathcal{Q}}b_{j,\,\mathbb{L}},\,,\,b_{j,\mathbb{L}}=\big (f_j-\langle f_j\rangle_{\mathbb{L}}\big)\chi_{\mathbb{L}}(x).$$
		Then we have
		\begin{eqnarray*}
			\Lambda^{A,s_Q,*}_Q(f_1,f_2)=\Lambda^{A,*}_{\mathcal{Q}}(g_1,f_2)+\Lambda^{A,*}_{\mathcal{Q}}(b_1,f_2)+\sum_{\mathbb{L}\in\mathcal{Q},\,\mathbb{L}\subset Q}\Lambda^{A,s_\mathbb{L},*}_\mathbb{L}(f_1\chi_{\mathbb{L}},f_2\chi_{3\mathbb{L}}).
		\end{eqnarray*}
		Inequality (\ref{equation3.6}) tells us that
		$$|\Lambda^{A,*}_{\mathcal{Q}}(g_1,f_2)|\lesssim r'^2\|g_1\|_{L^{r'}(\mathbb{R}^d)}\|f_2\|_{L^r(\mathbb{R}^d)}\lesssim  r'^2|Q|
		\|f_1\|_{\mathcal{Y}_1(\mathcal{Q})}\|f_2\|_{\mathcal{Y}_r(\mathcal{Q})}.
		$$
		On the other hand, Lemma \ref{dui} indicates that
		$$|\Lambda^{A,*}_{\mathcal{Q}}(b_1,f_2)|\lesssim r'^2\|b_1\|_{\dot{\mathcal{X}}_1(\mathcal{Q})}\|f_2\|_{\mathcal{Y}_r(\mathcal{Q})}\lesssim  r'^2|Q|
		\|f_1\|_{\mathcal{Y}_1(\mathcal{Q})}\|f_2\|_{\mathcal{Y}_r(\mathcal{Q})}.
		$$
	Combining these two estimates may leads to (\ref{equation3.8}).
		
		With (\ref{equation3.8}) in hand, by repeating the proof of Theorem C in \cite{ccdo}, we obtain that for each bounded function $f_1$ and $f_2$, there exists a sparse family of cubes $\mathcal{S}$, such that
		$$\Big|\int_{\mathbb{R}^d}f_1(x)\widetilde{T}_{\Omega,\,A}f_2(x)dx\Big|\lesssim r'^2\mathcal{A}_{\mathcal{S},\,L^1,\,L^r}(f_1,\,f_2).
		$$
		This completes the proof of Theorem \ref{thm3.2}.
	\end{proof}
	\section{proofs of Theorems \ref{thm1.2}-\ref{thm1.3}}\label{sec5}
	
	We begin with a preliminary lemma.

	\begin{lemma}[\cite{hulaixue}]\label{lem4.4} Let $\beta\in [0,\,\infty)$, $r\in [1,\,\infty)$ and $w$ be a weight. Then for any $t\in (1,\,\infty)$ and $p\in(1,\,r')$ such that
		$t\frac{p'/r-1}{p'-1}>1$,
		\begin{equation*}\mathcal{A}_{\mathcal{S},L(\log L)^{\beta},L^r}(f,\,g)\lesssim p'^{1+
				\beta}\big(\frac{p'}{r}\big)'\big(t\frac{p'/r-1}{p'-1}\big)'^{\frac{1}{p'}}\|f\|_{L^p(\mathbb{R}^d,\,M_tw)}
			\|g\|_{L^{p'}(\mathbb{R}^d,w^{1-p'})}.
		\end{equation*}
	\end{lemma}
	\begin{theorem}\label{thm4.1}Let  $\alpha,\,\beta\in \mathbb{N}\cup\{0\}$, $t,\,r\in [1,\,\infty)$, $p_1\in (1,\,r')$ such that
		$ t{(p_1'/r-1)}>{p_1'-1}$. Let $U$ be a  sublinear operator which enjoys a $(L(\log L)^{\beta},\,L^r)$-bilinear sparse domination with bound $r'^{\alpha}$. Then for any weight $u$ and bounded function $f$ with compact support, it holds that
		\begin{eqnarray}\label{ine:4.1}
			&&u(\{x\in\mathbb{R}^d:\,|Uf(x)|>1\})\\
			&&\quad\lesssim \Big(1+ \Big\{Dp_1'^{1+\beta}\big(\frac{p_1'}{r}\big)'\big(t\frac{p_1'/r-1}{p_1'-1}\big)'^{\frac{1}{p_1'}}
			\Big\}^{p_1}\Big)\int_{\mathbb{R}^d}|f(y)|\log^{\beta}({\rm e}+|f(y)|)M_tu(y)dy.\nonumber
		\end{eqnarray}
	\end{theorem}
Theorem \ref{thm4.1} was proved essentially in \cite{hulaixue} (note that the definition of $(L(\log L)^{\beta},\,L^r)$-bilinear sparse domination in Definition \ref{defn3.1} is somewhat different from the  definition of $(L(\log L)^{\beta},\,L^r)$-bilinear sparse domination in \cite{hulaixue}). We omit the proof here.
			\medskip
				
	We are now ready to prove Theorem \ref{thm1.2}.
	
	\medskip
	
	{\it  Proof of Theorem \ref{thm1.2}}.  For $p\in (1,\,\infty)$ and $w\in A_{p}(\mathbb{R}^d)$, let $\tau_{w}=2^{11+d}[w]_{A_{\infty}}$ and $\tau_{\sigma}=2^{11+d}[\sigma]_{A_{\infty}}$, $\varepsilon_1=\frac{p-1}{2p\tau_{\sigma}+1}$, and $\varepsilon_2=\frac{p'-1}{2p'\tau_w+1}$. It was proved in \cite[Theorem 1.6]{lik} (see also \cite[pp.147-148]{hu1}) that \begin{eqnarray}\label{equation4.6}&&\ \mathcal{A}_{\mathcal{S},L^{1+\varepsilon_1},\,L^{1+\varepsilon_2}}(f,\,g)\lesssim [w]_{A_p}^{1/p}([\sigma]_{A_{\infty}}^{1/p}+[w]_{A_{\infty}}^{1/p'})\|f\|_{L^p(\mathbb{R}^d,\,w)}
		\|g\|_{L^{p'}(\mathbb{R}^d,\sigma)}.\end{eqnarray}
	Note that
	$$\mathcal{A}_{\mathcal{S},L\log L,\,L^{1+\varepsilon_2}}(f,\,g)\lesssim \frac{1}{\varepsilon_1}\mathcal{A}_{\mathcal{S},L^{1+\varepsilon_1},\,L^{1+\varepsilon_2}}(f,\,g).$$
	Let $f$ and $g$ be bounded functions with compact supports. Theorem \ref{thm3.1} and (\ref{equation4.6}) then give that
	\begin{eqnarray*}
		\Big|\int_{\mathbb{R}^d}g(x) T_{\Omega,\,A}f(x)dx\Big|&\lesssim & \frac{1}{\varepsilon_1}\frac{1}{\varepsilon_2}\sup_{\mathcal{S}}\mathcal{A}_{\mathcal{S},\,L^{1+\varepsilon_1},\,L^{1+\varepsilon_2}}(f,\,g)\\
		&\lesssim& [w]_{A_p}^{1/p}([\sigma]_{A_{\infty}}^{1/p}+[w]_{A_{\infty}}^{1/p'})[w]_{A_{\infty}}[\sigma]_{A_{\infty}}\|f\|_{L^p(\mathbb{R}^d,\,w)}
		\|g\|_{L^{p'}(\mathbb{R}^d,\sigma)}.\nonumber
	\end{eqnarray*}
	Therefore
	\begin{eqnarray}\label{equation4.7}\|T_{\Omega,\,A}f\|_{L^p(\mathbb{R}^d,\,w)}\lesssim [w]_{A_p}^{1/p}([\sigma]_{A_{\infty}}^{1/p}+[w]_{A_{\infty}}^{1/p'})[w]_{A_{\infty}}[\sigma]_{A_{\infty}}\|f\|_{L^p(\mathbb{R}^d,\,w)}.
	\end{eqnarray}
	On the other hand, by Theorem \ref{thm3.2} and estimate (\ref{equation4.6}), one may obtain
	\begin{eqnarray*}
		\Big|\int_{\mathbb{R}^d}g(x) T_{\Omega,\,A}f(x)dx\Big|
		&\lesssim & \frac{1}{\varepsilon_1^2}\sup_{\mathcal{S}}\mathcal{A}_{\mathcal{S},\,L^{1+\varepsilon_1},\,L^{1}}(f,\,g)\\
		&\lesssim& [w]_{A_p}^{1/p}([\sigma]_{A_{\infty}}^{1/p}+[w]_{A_{\infty}}^{1/p'})[\sigma]^2_{A_{\infty}}\|f\|_{L^p(\mathbb{R}^d,\,w)}
		\|g\|_{L^{p'}(\mathbb{R}^d,\sigma)},
	\end{eqnarray*}
	which implies that
	\begin{eqnarray}\label{equation4.8}\|T_{\Omega,\,A}f\|_{L^p(\mathbb{R}^d,\,w)}\lesssim [w]_{A_p}^{1/p}([\sigma]_{A_{\infty}}^{1/p}+[w]_{A_{\infty}}^{1/p'})[\sigma]^2_{A_{\infty}}\|f\|_{L^p(\mathbb{R}^d,\,w)}.\end{eqnarray}
	Theorem \ref{thm1.2} then follows from inequalities (\ref{equation4.7}) and (\ref{equation4.8}). \qed
	
	\medskip
	{\it Proof of Theorem \ref{thm1.3}}. We invoked the argument used in the proof of Corollary 3.3 in \cite{hulaixue}.  Let $w\in A_1(\mathbb{R}^d)$. We choose $t=1+\frac{1}{2^{11+d}[w]_{A_{\infty}}}$, $r=(1+t)/2$ and $p_1=1+\frac{1}{\log ({\rm e}+[w]_{A_{\infty}})}$.  Then by Theorem \ref{thm3.1}, $T_{\Omega,\,A}$ enjoys a $(L\log L,\,L^r)$-bilinear sparse domination with bound $r'$. Denote $ C_{r,p_1}= \Big(r'p_1'^2\big(\frac{p_1'}{r}\big)'\big(t\frac{p_1'/r-1}{p_1'-1}\big)'^{\frac{1}{p_1'}}
	\Big)^{p_1}$ and since $t\frac{p_1'/r-1}{p_1'-1}>1$, by Theorem \ref{thm4.1}, it holds that
	\begin{eqnarray}\label{equation5.7}
		&&u(\{x\in\mathbb{R}^d:\,|T_{\Omega,\,A}f(x)|>1\})\lesssim (1+C_{r,p_1})\int_{\mathbb{R}^d}|f(y)|\log ({\rm e}+|f(y)|)M_tu(y)dy.\nonumber
	\end{eqnarray}
	Note that  $M_tw(y)\lesssim [w]_{A_1}w(y)$, $\big(t\frac{p_1'/r-1}{p_1'-1}\big)'\leq 5t'
	$, $r'=\frac{t+1}{t-1}\le 2^{12+d} [w]_{A_{\infty}}$, $p_1'\lesssim \log ({\rm e}+[w]_{A_{\infty}})$. Then \begin{eqnarray*} C_{r,p_1}&\lesssim &\big(r'p_1'^2\big)^{p_1}t'^{p_1-1}\lesssim [w]_{A_{\infty}} \log ^2({\rm e}+[w]_{A_{\infty}}).
	\end{eqnarray*}
 Our  desired conclusion now follows from (\ref{equation5.7}) directly.
	\qed
	
	Theorem \ref{thm3.2}, along with Theorem \ref{thm4.1}, leads to the following weighted estimate for $\widetilde{T}_{\Omega,A}$.
	\begin{corollary}\label{thm4.2}
		Let $\Omega\in L^{\infty}(\mathbb{S}^{d-1})$ be homogeneous of degree zero, satisfy the vanishing moment (\ref{equa:1.1}), $A$ be a function  on $\mathbb{R}^d$ with derivatives of order one in ${\rm BMO}(\mathbb{R}^d)$. Then for $w\in A_1(\mathbb{R}^d)$ and $\lambda>0$, the following inequality holds
		\begin{eqnarray*}w\big(\{x\in\mathbb{R}^d:\,|\widetilde{T}_{\Omega,A}f(x)|>\lambda\}\big)\lesssim [w]_{A_1}[w]^2_{A_{\infty}}\log ({\rm e}+[w]_{A_{\infty}})\lambda^{-1}\int_{\mathbb{R}^d}|f(x)|w(x)dx.
		\end{eqnarray*}
	\end{corollary}

	\textbf{Acknowledgements}. The authors would like to  thank Dr. Xudong Lai for  helpful discussion and   valuable suggestions.


\begin{thebibliography}{99}
		
		\bibitem{BC} 	B. M. Bainshansky and R. Coifman, \emph{On singular integrals}, Proc. Sympos. Pure Math. Vol.
		10, Amer. Math. Soc., Providence, R. I. (1967), 1-17.
		\bibitem{Barron2017}
		A.~Barron.
	\emph{ Weighted estimates for rough bilinear singular integrals via sparse
			domination}.
		\newblock { New York J. Math.}, {\bf 23} (2017), 779-811.
		\bibitem{cal1}A. P. Calder\'on,\emph{ Commutators of singular integral operators}, Proc. Nat. Acd. Sci. USA. \textbf{53} (1965),
		1092-1099.
		
		\bibitem{cal2} A. P. Calder\'on, \emph{Cauchy integrals on Lipschitz curves and related operators}, Proc. Nat. Acad. Sci.
		USA, \textbf{74} (1977), 1324-1327.
		
		\bibitem{cal3}A. P. Calder\'on, Commutators, \emph{singular integrals on Lipschitz curves and application}, Proc. Inter.
		Con. Math. Helsinki, 1978, 85-96, Acad. Sci. Fennica, Helsinki, 1980.
		
		\bibitem{cpcal} C. P. Calder\'on, \emph{On commutators of singular integral operators}, Studia Math. \textbf{53} (1975), 139-174.
			\bibitem{CZ1} A. P. Calder\'on, A. Zygmund, \emph{On singular integrals}, Amer. J. Math. {\bf 78} (1956), 289-309.
		
		\bibitem{chenlu} W. Chen and S. Lu, \emph{Boundedness criterion for multilinear oscillation singular integrals with rough kernels}, Studia Math. \textbf{165} (2004), 201-214.
			\bibitem{C1988} M. Christ, \emph{\it Weak type \((1,1)\) bounds for rough operators}, Ann. of Math. (2) {\bf 128} (1988), no. 1, 19-42.
		\bibitem{CR1988} M. Christ, J. L. Rubio de Francia, \emph{ Weak type \((1,1)\) bounds for rough operators II}, Invent. Math. {\bf 93} (1988), no. 1, 225-237.
		\bibitem{cohen} J. Cohen, \emph{A sharp estimate for a multilinear singular integral on $\mathbb{R}^d$}, Indiana Univ. Math. J. \textbf{30} (1981), 693-702.
		
		\bibitem{ccdo} J. Conde-Alonso, A. Culiuc, F. Di Plinio  and Y. Ou, \emph{A sparse domination
		principle for rough singular integrals}, Anal. PDE \textbf{10} (2017),   1255-1284.
			\bibitem{C1979} W. C. Connett, {\it Singular integral near $L^1$}, Proc. Sympos. Pure Math. {\bf 35} (1979), 163-165.
		\bibitem{dinglai} Y. Ding and X. Lai, \emph{Weak type $(1,\,1)$ bounded criterion for singular integral with rough kernel and its applications},
		Trans. Amer. Math. Soc.  \textbf{371} (2019),  1649-1675.
			\bibitem{D1993} J. Duoandikoetxea, \emph{ Weighted norm inequalities for homogeneous singular integrals}, Trans. Amer. Math. Soc. {\bf 336} (1993), no. 2, 869-880.
		\bibitem{DR1986} J. Duoandikoetxea, J. L. Rubio de Francia, {\it Maximal and singular integral operators via Fourier transform estimates}, Invent. Math. {\bf 84} (1986), no. 3, 541-561.
		
			\bibitem{gra0}  L. Grafakos,  \emph{Classical and Modern Fourier Analysis, Pearson Education}, Inc., Upper Saddle River, NJ, 2004.
		\bibitem{gra1}  L. Grafakos, \emph{ Classical Fourier Analysis}, GTM249, 2nd
		Edition, Springer, New York, 2008.
		
		\bibitem{gra}  L. Grafakos, \emph{ Modern Fourier Analysis}, GTM250, Third
		Edition, Springer, New York, 2014.
		
		\bibitem{GS1999} L. Grafakos, A. Stefanov, {\it Convolution Calder\'on-Zygmund singular integral operators with rough kernels}, in: Analysis of Divergence, Appl. Numer. Harmon. Anal. (1999), 119-143.
		
		\bibitem{krala} B. Krause and M. Lacey, \emph{Sparse bounds for maximally truncated oscillatory singular integral integrals}, Annali della Scuola Normale Superiore di Pisa - Classe di Scienze, to appear, available at arXiv:1701.05249.
		
		\bibitem{hansaw} Y. Han and E. T. Sawyer, \emph{Para-accretive functions, the weak boundedness properties and the $Tb$ theorem}, Rev. Mat. Iberoam. {\bf 6} (1990), 17-41.
			\bibitem{H1988} S. Hofmann, {\it Weak \((1,1)\) boundedness of singular integrals with nonsmooth kernel}, Proc. Amer. Math. Soc. {\bf 103} (1988), no. 1, 260-264.
		\bibitem{hof}S. Hofmann, \emph{On certain non-standard Calder\'on-Zygmund operators}, Studia Math. \textbf{109} (1994), 105-131.
		
		\bibitem{hof1}S. Hofmann, \emph{Weighted norm inequalities for commutators of rough singular integrals}, Indiana Univ. Math. J. \textbf{39} (1990), 1275-1304
		
		\bibitem{hu} G. Hu, \emph{Weighted vector-valued estimates for a  non-standard Calder\'on-Zygmund operator}, Nonlinear Anal. \textbf{165} (2017), 143-162.
		
		\bibitem{hu1} G. Hu, \emph{Quantitative weighted bounds for the composition of Calder\'on-Zygmund operators}, Banach J. Math. Anal. \textbf{13} (2019), 133-150.
		
		\bibitem{hy} G. Hu and D. Yang,  \emph{Sharp function estimates and weighted norm inequalities for multilinear singular integral operators}, Bull. London Math. Soc. \textbf{35} (2003), 759-769.
		
		\bibitem{hulaixue} G. Hu, X. Lai and Q. Xue, \emph{Weighted bounds for the compositions of rough singular integral operators}, J. Geom. Anal., \textbf{31} (2021), 2742-2765
		
		\bibitem{hutao} G. Hu and X. Tao, \emph{An endpoint estimate for the  commutators of   singular integral operators with rough kernels}, {Potential Anal., to appear.}
		\bibitem{huy}Y. Hu,
		\emph{An estimate for multilinear singular integrals on $\mathbb{R}^d$},
		Beijing Daxue Xuebao, (1985), no. 3, 19-26. (Chinese. English summary)
		
		\bibitem{hyla}T. Hyt\"onen and M. Lacey, \emph{The $A_p$-$A_{\infty}$ inequality for general Calder\'on-Zygmund
		operators}, Indiana Univ. Math. J. \textbf{61} (2012), 2041-2052.
		
		\bibitem{hlp} T. Hyt\"onen, M.   Lacey and C. P\'erez, \emph{Sharp weighted bounds for the $q$-variation of singular integrals}, Bull. Lond. Math. Soc. \textbf{45} (2013), 529-540.
		\bibitem{hp2} T.  Hyt\"onen and C. P\'erez, \emph{The $L(\log L)^{\epsilon}$ endpoint estimate for maximal singular integral operators}, J. Math. Anal. Appl. \textbf{428} (2015), 605-626.
		
		\bibitem{hrt}T. Hyt\"onen, L. Roncal, and O. Tapiola, \emph{Quantitative weighted estimates for
		rough homogeneous singular integrals},  Israel J. Math. \textbf{218} (2017), 133-164.
		
		\bibitem{lai2} X. Lai, \emph{Vector-valued singular integral operators with rough kernels}, Submitted.
		\bibitem{ler1} A. K. Lerner, \emph{A simple proof of the $A_2$ conjecture}, Int. Math. Res. Not. \textbf{14} (2013),
		3159-3170.
		
		\bibitem{ler3} A. K. Lerner, \emph{On pointwise estimate involving sparse operator},  New York J. Math. \textbf{22} (2016), 341-349.
		
		\bibitem{ler4} A. K. Lerner, \emph{A weak type estimates for rough singular integrals}, Rev. Mat. Iberoam.
		\textbf{35} (2019), 1583-1602.
		
		\bibitem{lno} A. K. Lerner, F. Nazarov and S. Ombrosi, \emph{On the sharp upper bound related to
		the weak Muckenhoupt-Wheeden conjecture}, arXiv:1710.07700.
		\bibitem{lik}K. Li, \emph{Two weight inequalities for bilinear forms}, Collect. Math. 68 (2017), no. 1, 129-144.
		
		
		\bibitem{lpr} K. Li, C. P\'erez, Isreal P. Rivera-Rios and L. Roncal, \emph{Weighted norm inequalities for rough singular integral operators}, J. Geom. Anal. \textbf{29} (2019), 2526-2564.
		
		\bibitem{M}B. Muckenhoupt, \emph{Weighted norm inequalities for the Hardy maximal function}, Trans. Amer.
		Math. Soc., \textbf{165} (1972), 207-C226.
		
		
		
		\bibitem{rr}M.  Rao and Z. Ren, \emph{Theory of Orlicz spaces}, Monographs and Textbooks in Pure
		and Applied Mathematics, 146, Marcel Dekker Inc., New York, 1991.
		
		\bibitem{RW1979} F. Ricci, G. Weiss, {\it A characterization of $H^1(\mathbb{S}^{d-1})$}, Proc. Sympos. Pure Math. {\bf 35} (1979), 289-294.
		
		\bibitem{se} A. Seeger, \emph{Singular integral operators with rough convolution kernels}, J.
		Amer. Math. Soc. \textbf{9} (1996), 95-105.
		
		
		\bibitem{T1999} T. Tao, {\it The weak-type (1,1) of $L\log L$ homogeneous convolution operator}, Indiana Univ. Math. J. {\bf 48} (1999), no. 4, 1547-1584.
		
		\bibitem{ste2} E. M. Stein, \emph{Harmonic Analysis, Real Variable Methods, Orthogonality, and Oscillatory
		Integrals}, Princeton Univ. Press, Princeton, NJ. 1993.
		
		\bibitem{V1996} A. M. Vargas, {\it Weighted weak type \((1,1)\) bounds for rough operators}, J. London Math. Soc. (2) {\bf 54} (1996), no. 2, 297-310.
		
		\bibitem{W1990} D. K. Watson, {\it Weighted estimates for singular integrals via Fourier transform estimates}, Duke Math. J. {\bf 60} (1990), no. 2, 389-399.
		\bibitem{wil}M. J. Wilson, \emph{Weighted inequalities for the dyadic square function without dyadic $A_{\infty}$}, Duke Math. J. \textbf{55} (1987), 19-50.
	\end{thebibliography}
\end{document}